\documentclass[11pt]{article}

\pdfoutput=1

\RequirePackage{my_packages}
\RequirePackage{my_theorems_en}
\RequirePackage{my_macros}
\RequirePackage{my_tikz}

\usepackage[backend = biber,        
            language = english ,
            style    = numeric-comp ,   
            giveninits = true,
            isbn = false,
            url = false,
            doi = false,
            sorting = nyt,
            maxnames = 6,
            backref=true
            ]{biblatex}

\DeclareNameAlias{sortname}{last-first}
\renewbibmacro{in:}{%
  \ifentrytype{article}{}{%
  \printtext{\bibstring{in}\intitlepunct}}}
\AtEveryBibitem{%
  \clearfield{note}%
  \clearfield{edition}%
}

\DeclareNameAlias{labelname}{given-family}

 \title{Relative order and spectrum \\in free and related  groups}
\author[1]{Jordi Delgado\thanks{\url{jordi.delgado.r@gmail.com
}}}
\author[2]{Enric Ventura\thanks{\url{enric.ventura@upc.edu}}}
\author[3]{Alexander Zakharov\thanks{\url{alexander.zakharov@uwr.edu.pl }}}

\affil[1]{Department of Mathematics, University of the Basque Country (EHU)}
\affil[2]{Departament de Matem\`atiques, Universitat Polit\`ecnica de Catalunya, and Institut de Matem\`atiques de la UPC-BarcelonaTech, Catalonia}
\affil[3]{Instytut Matematyczny, Uniwersytet Wroclawski, pl. Grunwaldzki 2/4,
50-384 Wroclaw, Poland}

\begin{document}

\maketitle

\begin{abstract}
In this paper, we consider a natural generalization of the concept of order of an element in a group: an element $g \in G$ is said to have order $k$ in a subgroup $H$ of $G$ (\resp \wrt a coset $Hu$) if $k$ is the first strictly positive integer such that $g^k \in H$ (\resp $g^k \in Hu$). We study this notion and its algorithmic properties in the realm of free groups and some related families.

\medskip
    
Both positive and negative (algorithmic) results emerge in this setting. On the positive side, among other results, we prove that the order of elements, the set of orders (called spectrum), and the set of preorders (\ie the set of elements of a given order) \wrt finitely generated subgroups are always computable in free and free times free-abelian groups. On the negative side, we provide examples of groups and subgroups having essentially any subset of natural numbers as relative spectrum; in particular, non-recursive and even non-recursively enumerable sets of natural numbers. Also, we take advantage of Mikhailova's construction to see that the spectrum membership problem is unsolvable for direct products of nonabelian free groups.
\end{abstract}

\medskip

\begin{flushleft}
\end{flushleft}

\section{Introduction}

For a finite group $G$, the notion of \emph{order} of an element is crucial to understand and study the algebraic structure of $G$: the order of $g\in G$, denoted by $\Ord(g)$, is the minimum positive integer $k\geqslant 1$ such that $g^k= \trivial$ in $G$ (with a previous standard argument proving that such $k$ always exists). The classical Lagrange Theorem, saying that the order of every $g\in G$ must divide the cardinal of the group, is fundamental in the study of the structure of finite groups (among many other classical results involving orders of elements, Gauss Theorem, Sylow Theorems, etc.).

When we move to infinite groups there usually show up \emph{non-torsion} elements, \ie elements $g\in G$ such that $g^k\neq 1$ for every $k\geqslant 1$. Of course, in this infinite setting we lose Lagrange Theorem and the concept of order of an element is relegated only to the torsion part of $G$ (which is usually not well behaved: in general, it is not even a subgroup of $G$). 
On one extreme we have torsion-free groups, where it is totally vacuous, and on the other extreme, we have the so-called \emph{periodic} groups (including, of course, all finite groups), \ie those where every element is a torsion element. In this area of group theory (\ie infinite torsion groups) Burnside problems have special relevance. 

Let $G$ be an arbitrary group. In this paper, we are going to relativize the notion of order of an element. Although our main focus is with respect to finitely generated subgroups, many results can be extended to cosets as well. We define the \emph{order of $g$ in a subgroup $H\leqslant G$}, denoted $\Ord_H (g)$, as the minimum $k\geqslant 1$ such that $g^k\in H$, and we say that the order is zero (instead of infinite) when there is no such $k$. Clearly, this relative notion extends the classical one to the infinite group setting, and offers more potential to investigate the structure of the group; for example, it would allow to refine the classification of nonisomorphic groups with a given spectrum; see~\cite{shi_groups_1997}. Observe that, for $H=\{1\}$, we recover the classical notion of order: the order of $g\in G$ in $H=\{1\}$ is precisely $\Ord_1 (g)=\Ord (g)$. Also, when $H\normaleq G$ is normal, the order of an element $g\in G$ in $H$ is nothing else but the classical order of its equivalence class $gH$ in the quotient group $G/H$, \ie $\Ord_H (g)=\Ord (gH)$.

Some aspects of this notion have already been considered in the literature. For example, given $g\in G$ and $H\leqslant G$, we have $\Ord_H(g)=1$ if and only if $g\in H$; hence the classical membership problem $\MP(G)$ (given $g\in G$ and a finitely generated $H\leqslant G$, decide whether $g\in H$ or not) consists, precisely, on checking whether $\Ord_H(g)=1$ or not. A natural generalization of $\MP(G)$ is then the computability of $\Ord_H(g)$, for given $g$ and $H$. In this paper we investigate this problem as well as other related algorithmic problems, providing examples of groups $G$ where it is computable and others where it not. 

Another connection with notions already present in the literature is purity, or isolation: a subgroup $H\leqslant G$ is called \emph{pure} by some authors (see, for example, \cite{birget_pspace-complete_2000,birget_two-letter_2008,miasnikov_algebraic_2007}), and \emph{isolated} by others (see, for example, \cite{kapovich_stallings_2002}) if, for every $g\in G$ and $k\in \mathbb{Z}\setminus \{0\}$, $g^k\in H$ imply $g\in H$; in our language, this happens precisely when the set of orders in $H$ is included in $\{0,1\}$ . A related concept is that of the \emph{pure closure} of a subgroup $H$, which is defined  as the intersection of all pure subgroups of $G$ containing $H$. These notions appear recurrently in the literature, but usually as technical tools with some other purposes (for example, an amalgamated product $G_1 *_{A_1=A_2} G_2$, or an HNN-extension $G*_A$ has special properties when the corresponding subgroups $A_i\leqslant G_i$ and $A\leqslant G$ are pure). In the present paper, we start a systematic study of these concepts, and some generalizations, by themselves and not just as intermediate technical devices.

We start the paper defining the relative versions of the notions of root and order and their first properties, to then consider some natural algorithmic problems involving them: computation of the order of a given element, computation of the set of elements of a given order, computation of the spectrum (\ie the set of orders), etc. Of course, for a general group $G$, these problems look undecidable, and we provide a collection of specific negative results in this sense: groups having subgroups with non-computably-enumerable spectrum, groups with subgroups having spectrum with unsolvable membership, etc. To show this negative behavior, we use classical constructions like free and amalgamated products; in one of our arguments Bass--Serre theory plays a central role to understand orders of elements in certain subgroups. Also, we use Higman's Embedding Theorem to gain finite presentability in some special situations where this is possible (thanks to a later observation by M. Chiodo that classical Higman's embeddings preserve the orders of elements). As one may suspect, in this part of the paper, direct products of free groups like $\Free[2]\times \Free[2]$ pay their contribution via Mikhailova's construction; see~\cite{mikhailova_occurrence_1958}.

Then, we turn to positive results and focus our attention, first, on free groups. We show that, in this context, all the considered questions are computable: finitely generated subgroups have finite and uniformly computable spectrum; there is a compact, explicit, and algorithmic-friendly description of the set of elements of a given order in a given finitely generated subgroup, and even in a given coset of it (as a certain disjoint union of finitely many appropriate cosets conveniently closed by conjugation); one can compute the so-called $S$-pure closure (a generalization of the notion of pure closure) of a given finitely generated subgroup, which happens to be finitely generated again, etc. In this part of the paper, the graphical representation of subgroups of a free group via Stallings automata will be central, introducing a nice and fruitful geometric component into the discussion.

In the last section we study the case of free times free-abelian groups, \ie groups of the form $\FTA$. This is an apparently simple extension of free groups whose behaviour, however, is known to differ remarkably from that of its factors in certain respects. This family of groups has been source of interest in the last years; see, for example, \cite{delgado_algorithmic_2013,sahattchieve_convex_2015,roy_degrees_2021,carvalho_dynamics_2020}. Although the analysis is more complicated here, with arguments typically combining techniques both from the free and the free-abelian parts, we are able to prove parallel results, conveniently adapted to the new framework. In particular, we see that the spectrum is bounded and computable, and we obtain a finitary description for the sets of elements of a given order. The main tools used in this part of the paper are the techniques developed by the first two authors for studying this family of groups. This includes a vectorized version of Stallings theory, allowing to understand the lattice of (finitely generated) subgroups of $\FTA$ with a combination of graphical techniques (for the free part) and linear algebra over the integers (for the free-abelian part).

The structure of the paper is the following. In \Cref{sec: orders & spectrum} we give the definitions of relative root, order, and spectrum and analyze their first basic properties. In \Cref{sec: algorithmic} we define several algorithmic problems involving these notions, and observe obvious connections,  among them and with classical algorithmic problems like the Word Problem or the Membership Problem. In \Cref{sec: realizing} we prove the existence of groups $G$ and subgroups $H\leqslant G$ with (essentially) arbitrary spectrum, and analyze whether we can assume $G$ finitely generated or finitely presented, and $H$ finitely generated. In \Cref{sec: free times free} we study the case of direct products of free groups, proving for these groups the non existence of an algorithm to decide whether a given integer shows up as an order in a given finitely generated subgroup. Finally, in \Cref{sec: free,sec: free-abelian} we focus on free and free times free-abelian groups, respectively. 

\subsection{Notation and terminology}

The set of natural numbers, denoted by $\NN$, is assumed to contain zero, and we specify conditions on this set using subscripts; for example, we denote by $\NN_{\geq 1}$ the set of strictly positive integers. To avoid degenerate situations, we adopt the convention that no natural number divides $0$, \ie $n\ndivides 0$, for every $n\in \NN$. The number of elements of a set $S$ is denoted by $\card S$.

We use lowercase italic font ($u,v,w,\ldots$) to denote elements of the free group $\Fn$, and lowercase boldface font ($\vect{a},\vect{b}, \vect{c},\ldots$) to denote vectors, \ie elements of the free-abelian group $\ZZ^m$. Uppercase boldface font ($\matr{A},\matr{B},\matr{C},\ldots $) is used to denote matrices, which --- as it happens with homomorphisms in general --- are assumed to act on the right. That is, we denote by $(x)\varphi$ (or simply $x \varphi$) the image of the element $x$ by the homomorphism $\varphi$, and we denote by  $\phi \psi$ the composition \smash{$A \xto{\phi} B \xto{\psi} C$}. Accordingly, we denote by $g^{h} = h^{-1} gh$, the conjugation of an element $g$ by an element~$h$.

In the free times free-abelian context, capitalized calligraphic font is used to distinguish subgroups ($\HH,\mathcal{K},\mathcal{L},\ldots$) and subsets ($\mathcal{S},\mathcal{R}, \mathcal{T}, \ldots$) of $\FTA$, from their counterparts in $\Fn$ and $\ZZ^m$, denoted by $H,K,L,\ldots$ and $R,S,\ldots$ respectively. 

Throughout the paper we write $H \leqslant\fg G$, $H \leqslant\ff G$, $H \leqslant\fin G$, $H \leqslant\alg G$ to denote that the subgroup $H$ is finitely generated, a free factor, of finite index, and algebraic in $G$, respectively. 

Whenever possible, we try to extend the existing terminology to our expanded context, making the original meaning correspond to the trivial or default instance in the general setting. This is the case for the definitions of order, spectrum, periodic group, the order problem, or the torsion problem, for example.

Finally, regarding computability, the terms computable (or decidable) and computably enumerable are preferred to the also very common recursive (or solvable) and recursively enumerable, respectively.

\section{Relative roots, orders, and spectra} \label{sec: orders & spectrum}

\begin{defn} \label{def: k-root}
Let $G$ be a group, let $g \in G$, let $S \subseteq G$, and let $k \in \ZZ $.
 If $g^k \in S$ we say that $g$ is a \defin{$k$-th root} (\defin{$k$-root}, for short) \defin{of $S$}, and that $k$ is a \defin{logarithm of $S$ in base $g$}.
We denote by $_{\scriptscriptstyle{G}}\hspace{-3pt}
\sqrt[k]{S}$ (or simply by $\sqrt[k]{S}$ if the ambient group $G$ is clear) the set of $k$-roots of $S$ in $G$, and by $\Log{g}{S}$ the set of logarithms of $S$ in base $g$. That is, 
 \begin{equation}
g^k \in S \Biimp g\in \sqrt[k]{S} \Biimp k\in \Log{g}{S}.
 \end{equation}
\end{defn}

\begin{defn} \label{defn: rel order}
Let $G$ be a group, let $g \in G$, and let $S \subseteq G$. The \defin{(relative) order of~$g$ in~$S$}, denoted by~$\Ord_{S}(g)$, is the smallest strictly positive logarithm of $S$ in base $g$, if it exists, and zero otherwise; that is,
 \begin{equation}
\Ord_{S}(g) \,=\, \max \left\{0 \,,\, \min \,\set{k \geq 1 \st g^k \in H}\right\} . 
 \end{equation}
For every $k\in \NN$, the set of elements from $G$ of order $k$ in $H$ is denoted by~$\ofo{G}{k}{H}$, \ie
 \begin{equation*}
\ofo{G}{k}{H} \,=\, \Set{g \in G \st \Ord_{H}(g) = k}.
 \end{equation*}
We extend the notation for roots and orders to subsets $I \subseteq \NN$ in the natural way; namely, 
$\sqrt[I]{H} = \bigcup\nolimits_{i \in I} \sqrt[i]{H}$, and $\ofo{G}{I}{H} = \bigsqcup\nolimits_{i \in I} \ofo{G}{i}{H}$.
\end{defn}
Note that:
\begin{itemize}
\item $\Ord_{S}(g)=k\geq 1$ $\Biimp$ $g^1,\ldots ,g^{k-1}\notin S$ and $g^{k} \in S$.
\item $\Ord_{S}(g)=0$ $\Biimp$ $g^i \notin S$, for every $i\geq 1$.
\end{itemize}

\begin{rem}
Our notion of \emph{element of order zero} corresponds to the usual terminology of \emph{element of infinite order} (when referred to the trivial subgroup). This interpretation will be specially meaningful when dealing with orbits and their lengths in \Cref{sec: free}. Also, we will omit the adjective `relative' since the context will be always clear from the notation (if no subset is mentioned, we will assume the trivial subgroup).
\end{rem}

\begin{defn}
Let $S$ be a subset of a group $G$. The set of orders \wrt $S$ of the elements in~$G$, denoted by $\Ord_{S}(G)$, is called the \defin{spectrum} of~$G$ \wrt $S$ (or the \defin{$S$-spectrum} of $G$, for short). That is,
 \begin{align}
\Iset{S}{G} \,=\, \set{\Ord_{S}(g) \st g\in G} \,=\, \set{ k \in \NN \st \ofo{G}{k}{S} \neq \varnothing } \, .
 \end{align}
\end{defn}

We abbreviate $\Iset{\Trivial}{G}=\Iset{}{G}$ for the set of standard orders of the elements in~$G$, which is called the (standard) spectrum of~$G$.
Also, in concordance with the notation for roots, we write $\Ord^{*}_{S}(G)=\Ord_{S}(G) \cap \NN_{\geq 2}$ for the set of proper orders of $G$ in $S$.

\medskip

Throughout the paper we will focus on subsets with algebraic structure, specifically subgroups and cosets. For the former ones the previous (and related) concepts have specially nice properties.

\begin{lem}
Let $H$ be a subgroup of $G$ and let $g\in G$. Then, the set $\Log{g}{H}$ (of logarithms of $H$ in base~$g$) is an ideal of $\ZZ$. Equivalently, for all $k,l \in \ZZ$, $\sqrt[k]{H} \subseteq \sqrt[kl]{H}$ and $\sqrt[k]{H} \cap \sqrt[l]{H} \subseteq \sqrt[k+l]{H}$.\qed
\end{lem}

\begin{lem} \label{lem: order in asubgroup}
The order of an element~$g \in G$ in a subgroup~$H\leqslant G$ is the (unique) nonnegative generator of the ideal $\Log{g}{H}$. \qed
\end{lem}
That is, $\Log{g}{H} = \gen{\Ord_{H}(g)}$, $\gen{g} \cap H = \gen{g^{\Ord_{H}(g)}}$,
and we have the order function, namely $\Ord \colon G \times \SGP{\Gpi} \to \NN$, given by:
 \begin{align} \label{eq: order function}
\Ord_{H}(g) &\,=\, \left\{\! \begin{array}{ll} \min \,\set{k>0 \st g^k \in H} & \text{if } \gen{g} \cap H \neq \Trivial \text{ or } g=\trivial , \\ 0  & \text{otherwise.} \end{array} \right.
 \end{align}

\begin{rem}
Recall that, for every $g \in G$ and every $H \leqslant G$,
 \begin{itemize}
\item[(i)] $\Ord_{H}(g)=0 \Biimp \forall k\geq 1, \ g^k \notin H$;
\item[(ii)] $\Ord_{H}(g)=1 \Biimp g\in H$\quad(in particular, $\Ord_{H}(\trivial_{G}) = 1$);
\item[(iii)] $\Ord_{H}(g)=k \geq 2 \Biimp g^1, g^2,\ldots,g^{k-1} \notin H$ and $g^{k} \in H$.
 \end{itemize}
\end{rem}

Note also that if $H$ is a subgroup of $G$, then for every $k \geq 1$,
 \begin{equation} \label{eq: roots <- orders}
\sqrt[k]{H} \,=\, \bigsqcup\nolimits_{i \divides k} \ofo{G}{i}{H},
 \end{equation}
and 
\begin{equation} \label{eq: roots -> orders}
\ofo{G}{k}{H} \,=\, \sqrt[k]{H} \setmin \bigsqcup \nolimits_{k \smallneq i \divides k} \sqrt[i]{H}\,.
 \end{equation}

\begin{rem}
For each $g \in G$ the map $H \mapsto \Ord_H(g)$ is anti-monotone \wrt inclusion in~$\SGP{G}$ and divisibility in~$\NN$; that is, if $H \leqslant K$ then $\Ord_{K}(g) \divides \Ord_{H}(g)$. 
\end{rem}

\begin{rem}
    If $g \in G$ has non-zero order in $H$, then $\Ord_{Hu}(g) \leq \Ord_{H}(g)$, for any $u \in G$. Indeed, if $g,g^2,  \ldots ,g^{n-1} \notin Hu$, and $g^n \in Hu$, then it follows immediately that $g,g^2, \ldots g^{n-1} \notin H$.
\end{rem}

Note that, by definition, $\sqrt[0]{H} = G$ and $\sqrt[1]{H} = H$. The terminology below is meant to exclude these degenerate types of roots.

\begin{defn} \label{def: roots and periodicity}
Let $G$ be an arbitrary group, let $g \in G$, and let $H$ be a subgroup of~$G$. If $\Ord_{H}(g) \geq 1$, then we say that $g$ is a \defin{nontrivial root} of $H$, or that $g$ is \defin{periodic \wrt $H$}. 
The group $G$ is said to be \defin{periodic \wrt $H$} if every element in $G$ is periodic \wrt $H$ (that is, if $0\notin \Ord_{H}(G)$) and \defin{bounded-periodic \wrt $H$} if $\Ord_{H}(G) \subseteq [1,k]$, for some $k \in \NN$. In a similar vein, $g$ is said to be a \defin{proper root} of $H$ if $\Ord_{H}(g) \geq 2$; that is, the proper roots of $H$ are the periodic elements (\wrt $H$) not in $H$.
We denote the set of non-trivial roots of $H$ by $\sqrt{H}
$, and the set of proper roots of $H$ in $G$ by $\sqrt[*]{H}
$. Note that $G=\ofo{G}{0}{H} \sqcup \sqrt{H}$ and $\sqrt{H}=H\sqcup \sqrt[*]{H}$.
\end{defn}

\begin{rem}
If $H_1 \leqslant H_2 \leqslant G$, then $\ofo{G}{0}{H_2} \subseteq \ofo{G}{0}{H_1}$, and $\sqrt[k]{H_1} \subseteq \sqrt[k]{H_2}$ for every $k \geq 1$. Observe also that $\sqrt[k]{H_1}$ is not in general a subgroup of $G$ (the product of two $k$-roots of $H_1$ is not in general a $k$-root of $H_1$), unless, for example, when $G$ is abelian.  
\end{rem}

\begin{rem} \label{rem: order wrt normal}
The standard notions of order, periodicity, etc., correspond to our respective notions \wrt the trivial subgroup $H=\Trivial$, in which case we delete the reference to $H$ from the notation. Accordingly, $\Ord(g)$ is the standard order of the element $g$, except for the case of elements of infinite order for which we put $\Ord(g)=0$; also, $\ofo{G}{k}{\phantom{H}}$ is the set of torsion elements in $G$ of (standard) order $k$. More generally, if $H$ is a normal subgroup of $G$, then the order of an element $g \in G$ in $H$ coincides with the standard order of $gH$ in the quotient group $G/H$, \ie $\Ord^G_H(g) \,=\, \Ord^{G/H}(gH)$.
Hence, 
 \begin{equation}
\ofo{G}{k}{H} \,=\, \bigsqcup \ofo{(G/H)}{k}{}.
 \end{equation}
\end{rem}

It is clear from the definition that certain properties of the standard order extend naturally to the order in a subgroup.

\begin{lem}
Let $G$ be a group, let $H$ be a subgroup of $G$, let $g \in G$, and let $d\geq 1$. Then,
 \begin{equation} \label{eq: H-order of a power}
\Ord_H(g^d) \,=\, \frac{\Ord_H(g)}{\gcd(\Ord_H(g),d)}.
 \end{equation}
In particular, if $d\divides \Ord_H(g)$ then $\Ord_H(g^d)=\Ord_H(g)/d$.\qed
\end{lem}

Below we recall that the order of an element in a subgroup is indeed an algebraic invariant.

\begin{lem} \label{lem: order modulo autos}
Let $H \leqslant G$, $u,v \in G$, and $\varphi \in \Aut(G)$. Then $\Ord_{H}(u) \,=\, \Ord_{H\varphi}(u\varphi)$. In particular, $\Ord_{H}(u) = \Ord_{H^{v}}(u^{v})$, and the set $\ofo{G}{k}{H}$ (of $G$-elements of order $k$ in $H$) is closed under conjugation by elements from $H$. \qed 
\end{lem}

From \Cref{rem: order wrt normal} it is clear that if $H$ is normal in $G$, then $\Iset{H}{G}$ is the set of standard orders of the elements in the quotient group $G/H$, \ie $\Iset{H}{G}=\Iset{}{G/H}$; hence $G$ is periodic \wrt a normal subgroup $H$ if and only if $G/H$ is periodic.

The properties below follow easily from the definitions.

\begin{lem}
Let $G$ be a group, and let $H,K$ be subgroups of $G$. Then,
 \begin{enumerate}[ind]
\item $\Iset{H}{G}$ is closed by taking divisors;
\item $1\in \Iset{H}{G}$;
\item if $K\leqslant H \leqslant G$ then $\Iset{K}{H} \subseteq \Iset{K}{G}$;
\item if $K\leqslant H$ then, for every $g\in G$, $\Ord_{H}(g) \divides \Ord_{K}(g) \leq \ind{K}{H} \cdot \ord_{H}(g)$;
\item if $H\leqslant G$ and $\varphi \in \Aut(G)$, then $\Iset{H}{G}=\Iset{H \varphi}{G}$;
\item if $G$ is torsion-free, then $0 \in \Iset{H}{G}$ if and only if $H$ admits a nontrivial $\cap$-complement (see~\cite{delgado_lattice_2020});
\item if $H\normaleq G$ then $\Iset{H}{G}=\Iset{}{G/H}$.\qed
 \end{enumerate}
\end{lem}

Particularly appealing is the relation between subgroup-relative order  and index, which mimics and extends some classical topics in group theory.

\begin{cor} \label{cor: spectrum and index}
Let $G$ be a group, let $H$ be a subgroup of $G$, and let $1 \leq k < \infty$. Then,
 \begin{enumerate}[ind]
\item\label{item: index k => order < k+1} $\ind{H}{G}=k$ implies $\Iset{H}{G} \subseteq [1,k]$; in particular,
\item\label{item: order 0 => infinte index} if $0\in \Iset{H}{G}$ then $\ind{H}{G}=\infty$. \qed
 \end{enumerate}
\end{cor}

\begin{rem}
Observe that, when applied to normal subgroups, \Cref{cor: spectrum and index}\ref{item: index k => order < k+1} resembles the classical \emph{Lagrange's Theorem}; note, however, that orders do not divide the index, in general. On the other hand, the question of whether (or when) the converse to statement \ref{item: order 0 => infinte index} is true has a clear \emph{Burnside Problem} flavour. Concretely, (the contrapositive of) \ref{item: order 0 => infinte index} is saying that every group is periodic \wrt its finite index subgroups. Whether a finitely generated group $G$ is periodic \emph{only} \wrt its finite index subgroups constitutes a natural generalization of the Burnside question. In this paper we see that this converse is true for free (\Cref{prop: order 0 Fn}) and free times free-abelian groups (\Cref{lem: order 0 FTA}).
\end{rem}

\begin{defn}
A group $G$ is said to have \defin{bounded subgroup spectra} if its spectrum \wrt every finitely generated subgroup is bounded. Analogous notions can be defined relative to cosets.
\end{defn}

A well-known notion closely related to the previous discussion is that of \emph{pure subgroup}, sometimes called \emph{isolated subgroup} in the literature; see~\cite{birget_pspace-complete_2000, birget_two-letter_2008, kapovich_stallings_2002} and \cite{miasnikov_algebraic_2007}. We extend this notion a bit further. 

\begin{defn}
Let $H$ be a subgroup of a group $G$ and let $S\subseteq \NN_{\geqslant 1}$. Then, $H$ is said to be \defin{$S$-pure} (in $G$) if 
$\sqrt[S]{H}=H$ (or, equivalently, if~$\Ord_{H}(G) \cap \Divs{S}=\set{1}$). A subgroup $H$ is said to be \defin{pure} (in $G$) if it has no proper roots in $G$; that is, if $\sqrt{H}=H$ or, equivalently, $\Iset{H}{G} \subseteq \set{0,1}$.
\end{defn}

\begin{rem}
Note that nontrivial free factors are always pure, whereas the trivial subgroup is pure in a group $G$ if and only if $G$ is torsion-free.
\end{rem}

\begin{rem}
For every $S \subseteq \NN_{\geq 1}$, the whole group $G$ is obviously $S$-pure, and the intersection of $S$-pure subgroups is again $S$-pure. 
\end{rem}

\begin{defn}
Let $H$ be a subgroup of a group $G$ and $S\subseteq \NN_{\geq 1}$. The \defin{$S$-pure closure} of $H$, denoted $\pcl[S]{H}$, is the smallest $S$-pure subgroup of $G$ containing $H$, \ie 
 \begin{equation}
\pcl[S]{H} \,=\, \bigcap \Set{K \st H\leqslant K\leqslant G \text{ and $K$ is  $S$-pure in $G$}}.
 \end{equation}
In particular, the \defin{pure closure} of $H$ is $\pcl{H}= \pcl[\NN_{\geq 1}]{H}$, namely the intersection of all pure subgroups of $G$ containing $H$.
\end{defn}

\section{Algorithmic considerations} \label{sec: algorithmic}

One of the main goals of the present paper is to understand the algorithmic behavior of the concepts introduced. In particular,
the order and spectrum functions
constitute  natural sources for algorithmic questions, starting by their own computability, and going through the algorithmic nature of some of the involved subsets. 

We recall that a total function $f \colon \NN \to \NN$ is said to be \defin{computable} if there exists an algorithm (a Turing machine) such that, on every input $n \in \NN$, outputs the image of $n$ by~$f$; otherwise, it is said to be \defin{incomputable}. Similarly, a subset $S \subseteq \NN$ is said to be  \defin{computable} (\aka \defin{decidable} or \defin{recursive}) if its indicator function is computable, \ie if there exists an algorithm for deciding whether any given $n\in \NN$ belongs to~$S$ or not; $S$ is called \defin{incomputable} or \defin{undecidable} otherwise. Finally, a subset $S \subseteq \NN$ is said to be \defin{computably enumerable} (c.e.) (\aka \defin{semicomputable}, \defin{semidecidable} or \defin{recursively enumerable}) if there exists an algorithm able to list the elements in $S$. It is easy to see that $S$ is computable if and only if both $S$ and its complement $\NN \setmin S$ are computably enumerable. We extend these notions to any functions or sets able to be appropriately encoded into $\NN$. For details, we refer the reader to general references on computability; see, for example, \cite{davis_computability_1985,bridges_computability_2013,soare_turing_2018}.

In this section, we state some natural algorithmic questions related to roots, orders, and spectra (see~\Cref{tab: problems}). Some of them --- such as the \emph{purity problem}
or the classical subgroup membership problem --- have been considered before. In this paper, we aim to provide a general framework for this kind of questions and extend their scope from free groups to some related families.

A problem is \defin{computable} (or \defin{decidable} if it is a decision problem --- \ie it has only two possible outputs) if there exists an algorithm (formally, a Turing machine) able to solve every instance of it. We shall abuse language and refer directly to the algebraic objects involved in the problems instead of their codification (as an input or output of the corresponding Turing machine). For example, we shall refer to elements $g \in G$ instead of words in the generators $X^{\pm}$, and to finitely generated subgroups $H \leqslant \fg G$ instead of tuples of words in $X^{\pm}$ representing some generators of the subgroup. It is not difficult to see that this does not suppose any loss of generality; see, \eg \cite{miller_iii_combinatorial_2004}.

The area of algorithmic group theory dates back to the beginning of the last century with the seminal 1911 paper \cite{dehn_uber_1911}, where \citeauthor{dehn_uber_1911} identifies the \defin{word problem} (\WP) (together with the \defin{conjugacy problem} (\CP), and the \defin{isomorphism probem (\IP)}) as one of the fundamental decision problems in group theory:

\begin{named}[Word Problem, $\WP(G)$] \label{defn: WP}
Decide, given an element $g\in G$ (as a word in $X^{\pm}$), whether~$g =_{G} \trivial$ or not.
\end{named}

Of course here (as well as along the present paper) we will assume that $G$ is a group given by a computable presentation~${\pres{X}{R}}$, \ie a presentation with recursively enumerable (possibly finite) sets of generators $X$, and relations $R$. This implies that we can computably enumerate the elements in $G$, and the elements in the normal closure $\normalcl{R} \subseteq \Free[X]$; this last property is sometimes referred to as the \defin{positive part of the word problem} in $G$, denoted by~$\WP^{+}(G)$, being solvable (in the sense that it provides an enumeration of \emph{all} the words on the generators representing the trivial element from $G$). One first easy consequence in this context, is that the set of nontrivial roots of any finitely generated subgroup is computably enumerable as well.

\begin{lem} \label{lem: ce -> root ce}
If $G$ is a computably presented group, then there exists an algorithm which, on input a finitely generated subgroup $H$ of $G$, enumerates the elements in the set~$\sqrt{H}$ (of nontrivial roots of $H$ in $G$). 
\end{lem}

\begin{proof}
Let $H$ be a finitely generated subgroup of $G$. From the hypotheses, we can assume computable enumerations $(g_i)_{i\geq 0}$ and $(h_j)_{j \geq 0}$ of the elements in $G$ and in $H$, respectively (with possible repetitions). Hence, we can use a standard diagonal argument to produce an algorithm enumerating the set  $G\times H\times \NN_{\geq 1}=\set{(g_i, h_j,k) \st i\geq 0,\, j\geq 0,\, k\geq 1}$. Now, it is enough to start this last enumeration and at every step, say corresponding to $(i,j,k)$, start checking whether the element $g_i^{k}h_j^{-1}$ appears in the enumeration provided by $\WP^+(G)$ (meaning, when this is the case, that $g_i^{k} =_{G} h_j$): meanwhile this does not happen, keep the enumeration of $\WP^+(G)$ running (and simultaneously keep starting new instances); and, when one of the started instances $g_i^{k}h_j^{-1}$ shows up in the enumeration of $\WP^+(G)$ (if so), add $g_i$ to the (initially empty) list $L$ of the already found nontrivial roots of $H$ (and, optionally, stop the part of the process involving $g_i$). Since this procedure goes eventually through \emph{every} non-zero power of \emph{every} element in $G$, and \emph{every} element in $H$, it will eventually identify (and list in $L$) \emph{all} the elements in the set $\sqrt{H}$, whereas it will keep running indefinitely over the instances containing elements from $G$ of order zero in $H$, if any. 
\end{proof}

In \Cref{lem: ce -> root ce} we proved that if $G$ is computably enumerable then every set $\sqrt{H}$ (of nontrivial roots of a finitely generated subgroup $H\leqslant \fg G$) is computably enumerable. Note that for $\sqrt{H}$ to be computable it is still not enough to have computable $\WP(G)$, since we need a way to guarantee that, for any given $i$ and $k$, no further matches $g_i^{k} =_{G} h_j$ involving new $h_j$'s will appear after some point. Let us state this last problem formally.

\begin{named}[Torsion Problem,  $\TP(G)$]
Decide, given an element $g\in G$ and a finitely generated subgroup ${H\leqslant\fg G}$, whether $g$ is periodic \wrt $H$ (\ie whether $\Ord_{H}(g) \neq 0$).
\end{named}

By definition, if $\TP(G)$ is computable, then every set of the form $\sqrt{H}=G\setmin \ofo{G}{0}{H}$ (of nontrivial roots of a finitely generated  subgroup $H\leqslant \fg G$) is computable as well. Hence, if $\TP(G)$ is computable, then, given an input $g \in G$, we can decide whether it has order zero in $H$, and if not we can enumerate powers of $g$ in parallel with elements in $H$ and use $\WP^{+}(G)$ to reach a guaranteed match. Note, however, that we can still not compute the exact order $\Ord_{H}(g)$ since a match $g^k =h_j$ could (unless $k=1$), in principle, be improved by a later one $g^{k\!'} = h_{j\!'}$, where $k'<k$.
Conversely, it is worth mentioning here that the computability of $\WP(G)$ does not imply that of the torsion problem \emph{even for the standard order} (\ie for $H=\Trivial$); see~\cite{mccool_unsolvable_1970}.
It turns out that, if we replace $0$ by $1$ in the condition defining the \TP, we recover a well-known and important generalization of the word problem that precisely does the missing work to compute the order. 

\begin{named}[(Subgroup) Membership Problem, $\MP(G)$]
Decide, given an element $g\in G$ and a finitely generated subgroup ${H\leqslant\fg G}$, whether $g\in H$ (\ie whether $\Ord_{H}(g) =1$).
\end{named}

That is, $\MP(G)$ is computable if all the finitely generated subgroups of $G$ are uniformly computable. Moreover, the computability of $\MP(G)$ allows us to refine the argument in \Cref{lem: ce -> root ce} to computably enumerate the set of nonzero orders \wrt $H$, and the set of elements in $G$ of a given nonzero order in $H$.

\begin{lem}
If $G$ is a computably presented group and $\MP(G)$ is computable then, for every finitely generated subgroup $H \leqslant G$, the set $\Ord_{H}^{_{\geq 1}}(G)$ (of nonzero orders \wrt $H$), and  every set $\ofo{G}{k}{H}$ (of elements in $G$ of order $k\ge1$ \wrt $H$) are computably enumerable.
\end{lem}

\begin{proof}
Since $G$ is computably enumerable, we can enumerate the set $G\times \NN_{\geq 1} =\set{(g_i,k) \st i\geq 0, k\geq 1 }$ in such a way that, for each $i\geq 0$, the elements $(g_i, k)$ appear in increasing order of the second coordinate (this can be easily done by using a standard diagonal argument). Then, at every step of the enumeration --- say corresponding to $(i,k)$ --- we use the $\MP(G)$ to check whether $g_i^{k} \in H$. If the answer is \nop, we do nothing; otherwise we add $k$ to an (initially empty) list $O$, we add $g_i$ to an (initially empty) list $L_k$, and we kill the part of the process involving $g_i$. It is clear from construction that, this process will enumerate the set of nonzero orders \wrt $H$ in the list $O$, and, for every $k \geq 1$, the set of elements of order $k$ \wrt $H$ in the list $L_k$.
\end{proof}

Note also that, if $\MP(G)$ is decidable, then it is straightforward to compute the order $\Ord_{H}(g)$ of any periodic element $g \in G$ (see the proof of \Cref{lem: OP}).
Therefore, only the ability to recognize elements of order zero is missing in order to achieve the computability of the full order function  $(H,g) \mapsto \Ord_{H}(g)$ defined in \eqref{eq: order function}. The correspondig algorithmic problem is called the \defin{order problem} in accordance with the usual terminology for the case~${H = \Trivial}$.

\begin{named}[Order Problem, $\OP(G)$]
Output, given an element $g\in G$ and a finitely generated subgroup ${H \leqslant\fg G}$, the order of $g$ \wrt $H$.
\end{named}

Of course, if $\OP(G)$  is computable, then, for every $H \leqslant \fg G$ and every $k \geq 0$, the set $\ofo{G}{k}{H}$ (of elements of order $k$ \wrt $H$) is  computable as well. It is also clear that the computability of $\MP(G)$ together with that of $\TP(G)$ immediately provides the computability of $\OP(G)$. Finally, we add one last algorithmic property that combined with \MP\ allows us to detect elements of order zero, and hence compute the order problem. 

\begin{named}[Order Boundability Problem,  $\OBP(G)$]
Output, given a finitely generated subgroup ${H\leqslant G}$ and an element $g\in G$, a finite upper bound for $\Ord_{H}(g)$.
\end{named}

\begin{lem} \label{lem: OP}
Let $G$ be a finitely presented group. Then, the following statements are equivalent:
\begin{enumerate}[dep]
\item \label{item: OP} $\OP(G)$ is computable;
\item \label{item: MP+TP} $\MP(G)$ and $\TP(G)$ are computable;
\item \label{item: MP+OBP} $\MP(G)$ and $\OBP(G)$ are computable.
\end{enumerate}
\end{lem}

\begin{proof} \label{proof: TFAE}
It is clear from the definitions that the computability of $\OP(G)$ implies that of $\MP(G)$, $\TP(G)$, and $\OBP(G)$. Conversely, if both $\TP(G)$ and $\MP(G)$ are computable, then, given $g \in G$ and $H\leqslant\fg G$, we can use $\TP(G)$ to decide whether $\Ord_{H}(g) = 0$ and, if not, use $\MP(G)$ to find the first integer $k\geq 1$ such that $g^k \in H$ (which is guaranteed to exist by hypothesis). Finally, if both $\MP(G)$ and $\OBP(G)$ are computable then, given $g\in G$ and $H\leqslant\fg G$, it is enough to find  the first positive power of $g$ belonging to $H$ (using \MP(G)) up to the bound provided by $\OBP(G)$. If such a power exists, it is the searched order $\Ord_{H}(g)$, otherwise the order of $g$ \wrt $H$ must be zero. This concludes the proof.
\end{proof}

Let us now move to algorithmic questions related to the spectrum function $H \mapsto \Ord_{H}(G)$, where $H$ is a finitely generated subgroup of $G$. 

Note that the purity problem (already studied in~\cite{birget_pspace-complete_2000, kapovich_stallings_2002, miasnikov_algebraic_2007} for free groups) admits a  natural restatement in terms of the subgroup spectrum. 

\begin{named}[Purity Problem,  $\PP(G)$]
Decide, given a finitely generated subgroup ${H \leqslant\fg G}$,  whether~$H$ is pure or not, \ie whether $\Iset{H}{G} \subseteq \set{0,1}$.
\end{named}

In the same vein, we present other natural algorithmic questions involving subgroup spectra, which are interesting to us. 

\begin{named}[Torsion Group Problem,  $\TGP(G)$]
Decide, given a finitely generated subgroup ${H \leqslant\fg G}$, whether $G$ is periodic \wrt $H$, \ie whether $0\notin \Iset{H}{G}$.
\end{named}



\begin{named}[Spectrum Membership Problem,  $\SMP(G)$]
Decide, given $m \in \NN$ and a finitely generated subgroup ${H\leqslant\fg G}$, whether $G$ has elements of order $m$ in $H$, \ie whether $m\in \Iset{H}{G}$.
\end{named}

Hence, $\SMP(G)$ being computable means that the spectra of $G$ \wrt finitely generated subgroups $H\leqslant \fg G$ are \emph{uniformly} computable. Note that this is a sort of computability of the spectrum function. However, we shall be interested in producing the explicit spectrum whenever it is finite; this is called the \defin{finite spectrum problem}.

\begin{named}[Finite Spectrum Problem,  $\FSP(G)$]
Decide, given a finitely generated subgroup ${H \leqslant\fg G}$, whether the spectrum $\Iset{H}{G}$ is finite and, if so, output it.
\end{named}

Note that the computability of $\MP(G)$ only allows to computably enumerate the elements in the spectrum $\Ord_{H}(G)$. In order to guarantee the computability of $\FSP(G)$ it is still necessary to algorithmically detect when the enumeration has finished. This is precisely the property that defines the algorithmic problem below.

\begin{named}[Spectrum Boundability Problem,  $\SBP(G)$]
Output, given a finitely generated subgroup ${H \leqslant\fg G}$, a finite upper bound for $\Iset{H}{G}$ if it is finite, and $\infty$ otherwise.
\end{named}

\begin{table}[H] 
\centering
\begin{tabular}{llll} 
    \toprule
     Problem name & & Input & Output \\ \midrule
     Order problem & $\OP(G)$  & ${H \leqslant\fg G} $, $g\in G$ & $\Ord_{H}(g)$ \\[3pt]
     \quad Torsion problem& $\TP(G)$  & ${H \leqslant\fg G} $, $g\in G$ & $\Ord_{H}(g) \geq 1$? \\[3pt]
     \quad Subgroup membership & $\MP(G)$ & ${H \leqslant\fg G} $, $g\in G$  & $g \in H$?  \\[3pt]
     \quad\quad Word Problem & $\WP(G)$ & $g\in G$  & $g = \trivial$?  \\[3pt]
     \quad Order boundability & $\OBP(G)$ & ${H \leqslant\fg G} $, $g\in G$ & $k : \Ord_{H}(g) \leq k$\\[3pt]
     Finite spectrum & $\FSP(G)$ & ${H \leqslant\fg G}$ & 
     \!\!\!\Big\{\!\!\!
     $\begin{smallmatrix}
     \Iset{H}{G} \text{, if finite}\\
     \infty \text{, otherwise}\hfill
     \end{smallmatrix}
     $ \\[4pt]
    \quad Purity problem & $\PP(G)$ & ${H \leqslant\fg G}$ & $\Iset{H}{G} \subseteq \set{0,1}$?  \\[3pt]
    Spectrum membership & $\SMP(G)$ & ${H \leqslant\fg G}$, $k \in \NN$ & $ k \in \Iset{H}{G}?$ \\[3pt]
    \quad Torsion-group problem & $\TGP(G)$ & ${H \leqslant\fg G}$ & $0 \in \Iset{H}{G}$?  \\[3pt]
    Spectrum boundability & $\SBP(G)$ & ${H \leqslant\fg G}$ &
     \!\!\!\Big\{\!\!\!
     $\begin{smallmatrix}
     k\,\geq\, \sup \Iset{H}{G} \text{, if finite}\\
     \infty \text{, otherwise}\hfill
     \end{smallmatrix}
     $ \\[4pt]
     \bottomrule
\end{tabular}
\caption{Algorithmic problems related to order and spectrum}
\label{tab: problems}
\end{table}

As can be seen in the table, we are mainly interested in the uniform versions (where the subgroup $H$ is part of the input) of the problems. If needed, the corresponding non-uniform versions (for a fixed subgroup $H$) shall be denoted by adding a subscript $H$ to the problem name, e.g. $\MP_{H}(G)$ consists of deciding, on input $g \in G$, whether
$g \in H$; and $\OP_{\trivial}(G)$ is the standard order problem (\wrt the trivial subgroup) in $G$.

Throughout the paper, we shall consider some of the coset counterparts of the problems in \Cref{tab: problems}.
In order not to overload the terminology, we will state these properties explicitly, and we shall use problem names only for the cases involving subgroups.

\section{Realizing sets as subgroup spectra}\label{sec: realizing}

A very natural question in our context is which subsets of $\NN$ are indeed realizable as spectra of groups \wrt (finitely generated) subgroups. In this section we show that for every subset ${\mathcal O}\subseteq \NN$ closed under divisors and containing~$0$, there exists a finitely generated torsion-free group $G$, and a finitely generated subgroup $H\leqslant\fg G$, such that $\Ord_{H}(G) = {\mathcal O}$. Note that the same assertion is not true if we impose $G$ to be  computably (or finitely) presented, just by cardinality reasons: there are uncountably many such sets ${\mathcal O}$ (just take ${\mathcal O}=\{0,1\}\cup P$, for any set of primes $P$, for example), and only countably many pairs $(G,H)$ with $G$ computably presented and $H\leq G$ finitely (or computably) generated. 

Recall that a subgroup $H$ of a group $G$ is called \defin{malnormal} if $H^g \cap H=\Trivial$ for every $g\in G \setmin H$. Observe that any malnormal subgroup of a torsion-free group is pure, since if $g^k \in H$ for some $k\geq 1$ and some $g\in G \setmin H$, then $g^k \in H^g \cap H=\Trivial$, in contradiction with $G$ being torsion-free.

\begin{prop}\label{prop: amalgam1} 
Let $H$, $K$ be torsion-free groups, let $H_0$ be a proper malnormal subgroup of $H$, and let $K_0$ be a proper subgroup of $K$ isomorphic to $H_0$. Then, for any amalgamated product $G $ of $H$ and $K$ with amalgamated subgroups $H_0$ and $K_0$ (and amalgamation taking place along any isomorphism between them), $\Iset{H}{G}=  \Iset{K_0}{K} \cup \{ 0 \}$.
\end{prop}

(Note that in the above proposition none of the groups nor subgroups is required to be finitely generated.)

\begin{proof}
We will use some basic Bass--Serre theory in the proof (and this is the only place in the paper where we use it). For an introduction to Bass--Serre theory see, for example, \cite{bogopolski_introduction_2008}.  For a formal definition of a walk and reduced walk see~\Cref{ssec: automata}.

Let $T$ be the Bass--Serre tree corresponding to the amalgamated product $G$. Recall that $T$ has the following structure. There are two types of vertices in $T$ -- those corresponding to left cosets of $H$ in $G$ (we call them $H$-vertices) and those corresponding to left cosets of $K$ in $G$ (we call them $K$-vertices). The edges of $T$ correspond to the left cosets of the amalgamated subgroup $H_0=K_0$ in $G$, and the incidence relations are the natural ones: for any $g \in G$ the edge $gH_0 =gK_0$ begins at the $H$-vertex $gH$ and ends at the $K$-vertex $gK$. The action of $G$ on $T$ is by left multiplication. So, the stabilizer of an $H$-vertex $gH$ is the subgroup $gHg^{-1}$, a conjugate of $H$; similarly, stabilizers of $K$-vertices and of edges are conjugates of $K$ and of $H_0=K_0$, respectively. This action is transitive on edges and has two orbits of vertices ($H$-type and $K$-type). For the proof that $T$ is indeed a tree, see \cite{bogopolski_introduction_2008}. It follows from the above that an $H$-vertex $gH$ is incident only to the edges of the form $ghH_0$, where $h$ runs through any set of left coset representatives of $H$ modulo $H_0$, and all these edges are different to each other. Similarly, a $K$-vertex $gK$ is incident only to the edges of the form $gkK_0$, where $k$ runs through any set of left coset representatives of $K$ modulo $K_0$ (and all these edges are different from each other as well). In particular, $T$ is locally finite if and only if $H_0$ and $K_0$ are of finite index in $H$ and $K$, respectively.

Recall also the standard classification of elements of $G$ by their action on $T$. There are two kinds of elements of the amalgamated product $G$: elliptic and hyperbolic. Elliptic elements are those which fix some vertex of $T$; these are, precisely, the conjugates of elements in $H$ or $K$. On the other hand, a hyperbolic element $g\in G$ does not fix any vertex of $T$, and it determines an axis (\ie a bi-infinite reduced walk on $T$) along which $g$ acts by translation by some positive integer.

Let us look more closely to the action of $G$ on $T$. By the stabilizer of a walk $\walki$ we mean the subgroup of $G$ consisting of the elements which fix every edge and every vertex in~$\walki$. We claim that any reduced walk of length 2 beginning and ending at $K$-vertices, say $\walki$, has trivial stabilizer (this is not, in general, the case for those walks of length 2 beginning and ending at $H$-vertices). Indeed, the middle vertex of $\walki$ must be an $H$-vertex, say $gH$, and the first and last vertices of $\walki$ are of the form $gh_1K$ and $gh_2K$, for some $h_1,h_2\in H$ such that $h_1H_0 \neq h_2H_0$, \ie $h_1^{-1}h_2\not\in H_0$; then, the stabilizers of the two edges are $gh_1H_0h_1^{-1}g^{-1}$ and $gh_2H_0h_2^{-1}g^{-1}$, respectively, and their intersection is trivial,
 \begin{align}
\Stab(\walki) &\,=\,gh_1H_0h_1^{-1}g^{-1} \cap gh_2H_0h_2^{-1}g^{-1} \\ &\,=\, gh_1 \Big( H_0 \cap (h_1^{-1}h_2)H_0(h_2^{-1}h_1) \Big) h_1^{-1}g^{-1} \,=\, \Trivial,   
 \end{align}
since $h_1^{-1}h_2 \notin H_0$ and $H_0$ is malnormal in $H$. Therefore, every reduced walk of length at least 3 in $T$ has trivial stabilizer (because it always contains a subwalk of length 2 of the above form). 

Let us now prove the claimed result. The inclusion $\Iset{K_0}{K} \cup \{ 0 \} \subseteq \Iset{H}{G}$ is easy: if $m\in \Iset{K_0}{K}$, there exists $k\in K$ of order $m$ in $K_0$; but then $k$, considered as an element of $G$, also has order $m$ in $H$, because $H\cap K=H_0=K_0$ in $G$; therefore, $m\in \Iset{H}{G}$, as required. To show that $0\in \Iset{H}{G}$, take any hyperbolic element $g\in G$ (such an element exists, since $H_0$ and $K_0$ are proper subgroups of $H$ and $K$, respectively); since powers of hyperbolic elements are again hyperbolic, $g^n$ is hyperbolic for every $n\geq 1$ and so, $g^n \notin H$; therefore, $0\in \Iset{H}{G}$, as required.

In order to show the other inclusion, $\Iset{H}{G}\subseteq \Iset{K_0}{K} \cup \{ 0\}$, let $w\in G$ be an element of order $n\geq 1$ in $H$, and let us see that $n\in \Iset{K_0}{K}$. Since $w^n\in H$ is elliptic, and powers of hyperbolic elements are again hyperbolic, $w$ must be elliptic as well, \ie $w$ is a conjugate of an element in $H$ or in $K$; let us consider these two cases separately.

{\it Case 1:} $w=ghg^{-1}$, for some $h\in H$ and some $g\in G$. In this case, $h=g^{-1}wg$ fixes the vertex $H$ in $T$, and $n$ is the minimum positive integer such that $h^n=g^{-1}w^ng$ fixes the vertex $g^{-1}H$ in $T$. Note that if $g^{-1}H=H$, then $g\in H$, $w\in H$, $n=1\in \Iset{K_0}{K}$, and we are done. So, let us assume that $H$ and $g^{-1}H$ are different vertices in $T$, and let $\walki$ be the (unique) reduced walk in $T$ from $H$ to $g^{-1}H$. Since $h^n$ fixes both ends of $\walki$, it also fixes $\walki$. Note that, by the bipartite structure of $T$, $\walki$ cannot have length 1 (both $H$ and $g^{-1}H$ are $H$-vertices different from each other). On the other hand, if $\walki$ has length 3 or more, then it follows from the above claim that $h^n=1$; hence, $h=1$ (since $H$ is torsion-free), $w=1$, $n=1\in \Iset{K_0}{K}$, and we are done. So, we are reduced to the case that $\walki$ has length 2.

In this case, the first vertex of $\walki$ is $H$, the second one is of the form $h'K$ for some $h'\in H$, and the last one is of the form $h'k'H=g^{-1}H$, for some $k'\in K$; moreover, the first edge of $\walki$ is $h'H_0$, and the second one $h'k'H_0$. Since $h^n$ fixes $\walki$, it fixes its first edge, so $h^n \in h'H_0(h')^{-1}$. But $h \in H$, and $H_0$ is pure (since it is malnormal) in $H$, so also the conjugate subgroup $h'H_0(h')^{-1}$ is pure in $H$. Therefore, we get that $h \in h'H_0(h')^{-1}$, \ie $h$ fixes the first edge in~$\walki$. Hence, $h$ also fixes the second vertex in $\walki$, \ie $h=h'k(h')^{-1}$, for some $k\in K$. It follows then that $n$ is the minimum positive integer such that $h^n$ fixes the second edge in $\walki$, \ie such that $h^n \in h'k'H_0(k')^{-1}(h')^{-1}$. In other words, $n$ is the minimum positive integer such that  $k^n \in k'H_0(k')^{-1} = k'K_0(k')^{-1}$. Therefore, $n \in \Iset{k'K_0(k')^{-1}}{K}=\Iset{K_0}{K}$, as required. 

{\it Case 2:} $w=gkg^{-1}$, for some $k\in K$ and some $g\in G$. In this case,  $k=g^{-1}wg$ fixes the vertex $K$ in~$T$, and $n$ is the minimum positive integer such that $k^n=g^{-1}w^ng$ fixes the vertex $g^{-1}H$ in $T$ (recall that $n$ was the minimum positive integer such that $w^n \in H$). Let $\walki$ be the (unique) reduced walk in $T$ connecting $K$ to $g^{-1}H$; then $k^n$ fixes $\walki$. From the bipartite structure of $T$, it follows that $\walki$ does not have length neither 0 nor 2. On the other hand, if $\walki$ has length 3 or more, then it follows from the above claim that $k^n=1$; hence, $k=1$ (since $K$ is torsion-free), $w=1$, $n=1\in \Iset{K_0}{K}$, and we are done. So, we are reduced to the case where $\walki$ has length $1$, \ie it consists on a single edge, say $k'K_0=k'H_0$, for some $k'\in K$ such that $k'H=g^{-1}H$. Then $k\in K$ and $n$ is the minimum positive integer such that $k^n$ stabilizes this edge, \ie $k^n \in k'K_0(k')^{-1}$. It follows that $n \in \Iset{k'K_0(k')^{-1}}{K}=\Iset{K_0}{K}$, as required, and the proof is complete. 
\end{proof}

\begin{prop}\label{prop: amalgam2}
Let $\Free[2]=\pres{a,b}{-}$ be a free group of rank $2$, let $\mathcal{O}$ be a set of non-negative integers closed under taking divisors and containing $0$, and let $K_{\mathcal{O}}=\gen{{b^{-n}a^nb^n \st n \in \mathcal{O}}} \leqslant \Free[2]$. Then $\Iset{K_{\mathcal{O}}}{\Free[2]}=\mathcal{O}$, and the rank of $K_{\mathcal{O}}$ is equal to the cardinality of $\mathcal{O} \setmin \{ 0 \}$.
\end{prop}

\begin{proof}
Let $\Ati$ denote the Stallings graph of $K_{\mathcal{O}}$ with respect to the free basis $\{a, b\}$ (see~\Cref{ssec: automata} for a brief survey on Stallings automata). It is clear that $\Ati$ has the following structure: there is a ray
of length $\sup (\mathcal{O})$
of $b$-labelled arcs pointing to and ending at the basepoint.
Denoting by $\verti_i$ the vertex where we arrive after reading $b^{-i}$ from the basepoint, there are also directed (closed) $\verti_n$-walks reading $a^n$ for every $n\in \mathcal{O} \setmin \{ 0 \}$; there are no other vertices nor arcs in $\Ati$, see \Cref{fig: Stallings KO}. In view of $\Ati$, it is clear that $K_{\mathcal{O}}$ has rank equal to~$\card{(\mathcal{O} \setmin \{ 0 \})}$.

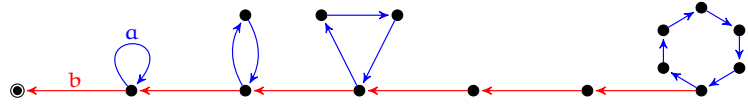
\begin{figure}[H]
    \centering
    \begin{tikzpicture}[shorten >=1pt, node distance= 1 cm and 1.5cm, on grid,auto,>=stealth']
    
    \newcommand{\dx}{1.5}
    \newcommand{\dy}{1}
    \node[state, accepting] (0) {};
    \node[state, right = of 0] (1) {};

    \node[state, right = of 1] (2) {};
    \node[state, above = of 2] (21) {};
    
    \node[state, right = of 2] (3) {};
    \node[state, above left = 1 and 0.5 of 3] (31) {};
    \node[state, above right = 1 and 0.5 of 3] (32) {};
    
    \node[state, right = of 3] (4) {};
    \node[state, right = of 4] (5) {};
    \node[state, right = of 5] (6) {};
    
    \node[state, above left = 0.3 and 0.5 of 6] (61) {};
    \node[state, above = 0.5  of 61] (62) {};
    \node[state, above right = 0.3 and 0.5 of 62] (63) {};
    \node[state, below right = 0.3 and 0.5 of 63] (64) {};
    \node[state, below = 0.5 of 64] (65) {};

    \path[->]
        (1) edge[red]
                node[pos=0.5,above=-.3mm] {$b$}
            (0);
            
    \path[->]
        (2) edge[red]
            (1);
            
    \path[->]
        (3) edge[red]
            (2);
            
    \path[->]
        (4) edge[red]
            (3);
            
    \path[->]
        (5) edge[red]
            (4);
            
    \path[->]
        (6) edge[red]
            (5);

    \path[->]
        (1) edge[blue, loop, in=50,out=130, looseness=25]
                node[pos=0.5,above=-.1mm] {$a$}
            (1);
            
    \path[->]
        (2) edge[blue, bend left]
            (21);
            
    \path[->]
        (21) edge[blue, bend left]
            (2);
            
    \path[->]
        (3) edge[blue]
            (31);
            
    \path[->]
        (31) edge[blue]
            (32);
            
    \path[->]
        (32) edge[blue]
            (3);
            
    \path[->]
        (6) edge[blue]
            (61);
            
    \path[->]
        (61) edge[blue]
            (62);
            
    \path[->]
        (62) edge[blue]
            (63);
            
    \path[->]
        (63) edge[blue]
            (64);
            
    \path[->]
        (64) edge[blue]
            (65);
            
    \path[->]
        (65) edge[blue]
            (6);

    \end{tikzpicture}
    \vspace{-5pt}
    \caption{Stallings automaton $\Gamma$ of $K_{\mathcal{O}}$, for $\mathcal{O}=\set{0,1,2,3,6}$.
    }
    \label{fig: Stallings KO}
\end{figure}

The inclusion $\mathcal{O}\subseteq \Iset{K_{\mathcal{O}}}{\Free[2]}$ is immediate: for every $n\in \mathcal{O}$, the element $b^{-n}ab^n\in \Free[2]$ clearly has order $n$ in $K_{\mathcal{O}}$.

For the other inclusion, let $0,1\neq n\in \Iset{K_{\mathcal{O}}}{\Free[2] }$, and let $w\in \Free[2]$ be an element of order $n$ in~$K_{\mathcal{O}}$. Write $w=u^{-1}\widetilde{w}u$ with no cancellations, and where $\widetilde{w}$ is a cyclically reduced word; then, $w^n=u^{-1}\widetilde{w}^{n}u$ is reduced as written, and can be read as the label of some $\bp$-walk in~$\Ati$. Note that, since $\widetilde{w}$ is cyclically reduced, the vertex where we arrive after reading $u^{-1}$ from the basepoint~$\bp$ cannot have degree $2$ and so, it must be $\verti_j$, for some $j\in \mathcal{O}\setmin \{0\}$; further, replacing $w$ with $w^{-1}$ if necessary, we can assume that the first letter in $\widetilde{w}$ is $a$. Let us now denote by $\vertii_0=\verti_j, \vertii_1, \ldots, \vertii_{j-1}$ the vertices along the cycle at $\verti_j$ (with label~$a^j$). Since $n\geq 2$, the vertex where we arrive after reading $\widetilde{w}$ from $\verti_j$ is incident to an $a$-arc. But, on the other hand, from the structure of $K_{\mathcal{O}}$, it is clear that the total $b$-exponent of $w^n\in K_{\mathcal{O}}$, and so of $w$ and of $\widetilde{w}$, must be $0$; so, the vertex we arrive after reading $\widetilde{w}$ from $\verti_j=q_0$ must be some $\vertii_i$, where $1\leq i\leq j-1$ (again, because $n\geq 2$). Now, since $\widetilde{w}$ can be read both starting from the vertices $\vertii_0$ and $\vertii_i$, it cannot contain any letter $b$; that is $\widetilde{w} = a^{kj +i}$, for some $k\in \NN$. Finally, from $n$ being the order of $w$ in $K_{\mathcal{O}}$, we see that it is also the smallest positive integer such that $\widetilde{w}^{n}$ is readable as a closed walk at $\verti_j$; therefore, $n$ is the smallest positive integer such that $ni$ is multiple of $j$, \ie $n=j/\gcd(i,j)\in \mathcal{O}$, as required. 
\end{proof} 

Finally, combining \Cref{prop: amalgam1,prop: amalgam2}, we deduce the desired example.

\begin{thm}\label{amalgam3} 
Let $\mathcal{O}$ be a set of non-negative integers closed under taking divisors, and containing~$0$. Then, there exists a finitely generated torsion-free group $G$ with a finitely generated subgroup $H$ such that the spectrum $\Iset{H}{G}=\mathcal{O}$.
\end{thm}

\begin{proof}
If $\mathcal{O}$ is finite, then \Cref{prop: amalgam2} immediately provides the required example. Otherwise, consider a free group of rank 2, say $H$, and a subgroup $H_0\leq H$ being malnormal and having infinite rank (it is not hard to construct such a subgroup explicitly, see for example~\cite{das_controlled_2015}); also, consider another free group of rank 2, say $K$, and let $K_0 =K_{\mathcal{O}}\leq K$ defined as in \Cref{prop: amalgam2}. Now fix an isomorphism from $H_0$ to $K_0$ (both are free groups of countably infinite rank), and let $G$ be the corresponding amalgamated product. Since $0\in \Iset{K_0}{K}$, it follows from \Cref{prop: amalgam1,prop: amalgam2} that $\Iset{H}{G}=\Iset{K_0}{K}=\mathcal{O}$. Note also that both $G$ and $H$ are finitely generated (while $G$ is not finitely presented, in general). 
\end{proof}

\begin{cor}
There exists a finitely generated torsion-free group $G$, and a finitely generated subgroup $H\leq G$, such that the spectrum $\Iset{H}{G}$ is non-computably enumerable (\resp computably enumerable but non-computable). 
\end{cor}

\begin{proof}
Take ${\mathcal O}=\{0,1\}\cup S$ in Theorem~\ref{amalgam3}, where $S$ is any non-computably enumerable (resp., computably enumerable but non-computable) set of prime numbers.  
\end{proof}

In fact, for computably enumerable sets, the following stronger statement holds.

\begin{thm}\label{thm: hyp}
Let $\mathcal{O}$ be any computably enumerable set of non-negative integers closed under taking divisors, and containing $0$. Then there exists a torsion-free finitely presented group $G$ and a finitely generated subgroup $H\leq G$ with $\Iset{H}{G}=\mathcal{O}$. Furthermore, the group $G$ can be chosen to be word hyperbolic and residually finite.
\end{thm}

\begin{proof}
Consider the free product of cyclic groups of orders in $\mathcal{O}$, say $K=\ast_{n\in \mathcal{O}} C_n$; by the hypotheses on $\mathcal{O}$, $\Ord(K)=\mathcal{O}$. Since $\mathcal{O}$ is computably enumerable, $K$ is recursively presented and we can use Higman's embedding theorem to embed it into a finitely presented group, $K\hookrightarrow Q$. Furthermore, by \cite[Theorem 2.2]{chiodo_torsion_2014} this embedding adds no new torsion and so $\Ord(Q)=\Ord(K)=\mathcal{O}$.

Now apply Wise's version of Rips construction to $Q$ (see \cite{wise_residually_2003}) to produce a word hyperbolic (in fact, $C'(1/6)$ small cancellation), torsion-free, residually finite group $G$ and a finitely generated normal subgroup $H\normaleq G$ such that $G/H$ is isomorphic to $Q$. Hence, ${\Iset{H}{G}=\mathcal{O}}$.
\end{proof}

\begin{cor}
There exists a word hyperbolic, torsion-free, residually finite group $G$, and a finitely generated subgroup $H\leq G$ with $\Iset{H}{G}$ being non-computable. 
\end{cor} 

\begin{proof}
Apply \Cref{thm: hyp} to any computably enumerable, non-computable set $\mathcal{O}$ of non-negative integers, closed under taking divisors, and containing 0.
\end{proof}

The question about realizing spectra not containing $0$ seems much more tricky, and related to periodic groups. Note that $0\not\in \Iset{H}{G}$ means that \emph{every} element in $G$ must have a strictly positive power inside $H$.

\section{Free times free groups}\label{sec: free times free}

The direct product of nonabelian free groups (typically $\Free[2] \times \Free[2]$) is a classic source of examples with bad algorithmic behaviour. A seminal example is the undecidability of its subgroup membership problem, proved by \citeauthor{mikhailova_occurrence_1958} in \cite{mikhailova_occurrence_1958}. Below, we take advantage of this result to expand undecidability to our context.

\begin{prop}\label{prop: freetimesfree}
Fix $n\geqslant 2$, and a prime number $p$. There is no algorithm such that, on input $H\leqslant \fg \Fn\times \Fn$, decides whether $p\in \Ord_{H}(\Fn\times \Fn)$ or not; in particular, the spectrum membership problem $\SMP(\Fn\times \Fn)$ is undecidable.
\end{prop}

\begin{proof}
Observe first that, for $r,s\geqslant 2$, there is a pure embedding of $\Free[r]$ into $\Free[s]$, \ie a pure subgroup  of $\Free[s]$ isomorphic to $\Free[r]$, say $\Free[r]\simeq H\leqslant \Free[s]$. For $r\leqslant s$ we can take, for example, $H$ to be a free factor of $\Free[s]$; and, for the general case, we just observe that the commutator~$\Comm{\Free[s]} \normaleq \Free[s]$ is pure in $\Free[s]$ (since it is normal and the corresponding quotient is torsion-free) and not finitely generated; hence, it contains a rank $r$ free factor $\Free[r]\simeq H\leqslant\ff [\Free[s] : \Free[s]]\leq \Free[s]$, which will be pure in $\Free[s]$ as well (recall that purity is transitive). Repeating this operation on each coordinate, we see that $\Free[r]\times \Free[r]$ embeds purely in $\Free[s]\times \Free[s]$, for every $r,s\geqslant 2$. Hence, proving the statement for each fixed $n\geqslant 2$ reduces to proving it for a single fixed value $n=n_0$, $n_0\geqslant 2$. For later convenience, let us do it for $n_0=3$, \ie let us work in $\Free[3]\times \Free[3]$.

We say that a group $G$ has \emph{unique $p$-roots} if, for every pair of elements $g_1,g_2 \in G$, $g_1^p=_G g_2^p$ implies $g_1=_G g_2$. For an arbitrary finite presentation on three generators, $Q=\pres{a_1, a_2, a_3}{r_1,\ldots ,r_m}$, consider the corresponding Mikhailova subgroup $M(Q)=\{ (u,v)\in \Free[3]\times \Free[3] \mid u=_{Q} v\}\leqslant \Free[3]\times \Free[3]$. It is straightforward to see that $M(Q)$ is (finitely) generated by the elements $\{ (a_1,a_1), (a_2, a_2), (a_3,a_3), (1,r_1),\ldots ,(1, r_m)\}$. Moreover, given two words $u,v$ on $a_1, a_2, a_3$, we have $u\neq _Q v$ but $u^p=_Q v^p$ if and only if $(u,v)\not\in M(Q)$ but $(u,v)^p=(u^p,v^p)\in M(Q)$, and (by the primality of $p$) if and only if $(u,v), (u^2,v^2),\ldots ,(u^{p-1},v^{p-1})\not\in M(Q)$ but  $(u,v)^p= (u^p,v^p)\in M(Q)$; that is, if and only if $\Ord_{M(Q)}(u,v)=p$. Hence, $Q$ fails to satisfy the unique $p$-root property if and only if $p\in \Iset{M(Q)}{\Free[3]\times \Free[3]}$. Therefore, for our purpose, it suffices to see that there is no algorithm which, on input a presentation $Q$ on three generators, decides whether $Q$ satisfies the $p$-root property, or not. 

It is already known that such an algorithm does not exist if we enlarge the set of inputs to \emph{all} finite presentations; see next proposition. We revisit one of the classical proofs for this fact, in order to add an easy observation at the end implying the desired undecidability, even when restricted to presentations with three generators. This concludes the proof. 
\end{proof}

\begin{prop}[Adian--Rabin]
For every fixed prime $p$, there is no algorithm which takes as an input a finite presentation with three generators $Q=\pres{a,b,c}{r_1,\ldots ,r_m}$, and decides whether $Q$ satisfies the $p$-root property, or not.  
\end{prop}

\begin{proof}
The proof is a slight variation of the argument given by Miller in \cite[pag.~13]{miller_iii_combinatorial_2004}, which was itself a variation of Gordon's proof for the original Adian--Rabin's theorem about non computability of Markov properties. We remind the context, briefly describe Miller's argument, and then add our contribution at the very last paragraph. 

An abstract property $\mathcal{M}$ for finitely presented groups is said to be \emph{Markov} if there exist finitely presented groups $G_+$ and $G_-$ such that $G_+$ satisfies $\mathcal{M}$, and every group $G$ admitting an embedding $G_-\hookrightarrow G$ does not satisfy $\mathcal{M}$ (for example, ``being finite" is a Markov property with $G_+=\{1\}$ and $G_-=\mathbb{Z}$, whereas ``being infinite" is not). The classical Adian--Rabin's Theorem states that Markov properties $\mathcal{M}$ are not computably recognizable, \ie there is no algorithm which, on input a finite presentation, decides whether the presented group satisfies the property $\mathcal{M}$, or not. 

Next, we will focus on the unique $p$-root property. Clearly, it is a Markov property: take, for example, $G_+=\mathbb{Z}$ and $G_-=\mathbb{Z}/ p\mathbb{Z}$. So, by Adian--Rabin's Theorem, there is no algorithm to decide whether a given finite presentation has unique $p$-roots. We state below a small observation showing that such an algorithm does not exist either, even restricting the input to finite presentations with exactly 3 generators. 

The classical proof goes as follows: the so-called Main Technical Lemma (see~\cite[Lem\-ma~6.13]{miller_iii_combinatorial_2004}) takes a finite presentation (of a group $K$) and a word $w$ on its generators, and it explicitly constructs a new finite presentation (of a group $L_w$) with just two generators and the following properties: (i) if $w\neq_K \trivial$ then $K$ embeds in $L_w$; (ii) if $w=_K \trivial$ then $L_w$ is the trivial group. Now, let $\mathcal{M}$ be a Markov property (with corresponding witnesses~$G_+$ and $G_-$), and let $U$ be a finitely presented group with unsolvable word problem. If we apply the Main Technical Lemma to $K=U*G_-$ and to a word $w$ in the generators of $U\leqslant K$, then the finitely presented group $L_w*G_+$ satisfies the following: if $w=_K 1$ then $L_w=1$ and hence $L_w*G_+=G_+$ satisfies property $\mathcal{M}$; otherwise (if $w\neq_K 1$), $G_-\hookrightarrow U*G_{-} = K \hookrightarrow L_w\hookrightarrow L_w*G_+$ and hence $L_w*G_+$ does not satisfy property~$\mathcal{M}$. 
That is, $L_w*G_+$ satisfies $\mathcal{M}$ if and only if $w =_{K} \trivial$.
Therefore, there exists no algorithm that, on input a finite presentation, recognizes whether the presented group satisfies $\mathcal{M}$ or not (since such an algorithm
would decide the word problem in~$U$ as well). 

Observe that, particularizing the above argument to the (Markov) property of having unique $p$-roots, the group $L_w$ in the construction above has two generators, and $G_+=\mathbb{Z}$ has one generator so, the group $L_w*G_+$ is 3-generated. Hence, the unsolvability of the word problem for $U$ implies also the unrecognizability of the unique $p$-root property, even when restricted to 3-generated presentations as inputs. 
\end{proof}

\begin{prop}
For $n \geq 2$, the group $\Fn \times \Fn$ does not have subgroup bounded spectra. More precisely, for every $n \geq 2$, there exists a finitely generated subgroup $H \leq \Fn \times \Fn$ such that $\Iset{H}{\Fn \times \Fn} = \mathbb{N}$.
\end{prop}

\begin{proof}
Let $K$ be the free product of finite cyclic groups of all finite orders, $K=\Fprod_{n\in \mathbb{N}} C_n$. Then $K$ is a computably presented group and so, by Higman embedding theorem, $K$ can be embedded into a finitely presented group $G$. It follows that $G$ is a finitely presented group containing elements of any possible order, but not a torsion group. Suppose $G$ is generated by $m$ elements and consider the Mikhailova subgroup $M(G)$ in $\Free[m] \times \Free[m]$, as in the proof of Proposition \ref{prop: freetimesfree}. Then it follows that $\Iset{M(G)}{\Free[m] \times \Free[m]} = \mathbb{N}$: namely, for a word $u$ in $\Free[m]$ representing an element of order $k\geq 1$ in $G$ (respectively, of infinite order) the element $(u,1)$ of $\Free[m] \times \Free[m]$ contributes $k$ (respectively, 0) to $\Iset{M(G)}{\Free[m] \times \Free[m]}$. Since $\Free[m] \times \Free[m]$ embeds purely into $\Fn \times \Fn$ for every $n \geq 2$, the claim follows.
\end{proof}

\section{Free groups} \label{sec: free}

Let $\Free=\Fn = \Free[X]$ denote the free group with basis $X= \set{x_1,\ldots,x_n}$, where $n \geq 2$. In this section, we provide an algorithmic description of the set $\ofo{\Free}{k}{Hu}$ (of $\Free$-elements of order $k$ in the coset $Hu$), for $k\geq 0$, and $H$ a finitely generated subgroup of $\Free$. As a consequence, we obtain an algorithmic description of $\sqrt[k]{H}$, and the computability of the order problem ($\OP(\Free)$), the spectrum membership problem ($\SMP(\Free)$), and the finite spectrum problem ($\FSP(\Free)$) within free groups. Finally, we prove the computability of (partial) pure closures within this family.

Note that the set of $k$-roots of $H$, $\sqrt[k]{H}$, is the set of solutions of the equation $X^k=Y$ subject to the rational constraint $Y\in H$. Therefore, it can be theoretically described using general approaches (see the original paper by Razborov~\cite{razborov_systems_1984}, and the version with rational constraints in~\cite{diekert_finding_2016}). However, we aim to obtain a more synthetic and intuitive description. For that purpose we will use the well known Stallings interpretation of subgroups as automata (see~\cite{stallings_topology_1983,kapovich_stallings_2002}) which we briefly summarize below.

\subsection{Automata and subgroups of free groups} \label{ssec: automata}

Recall that a \defin{directed graph} (a \defin{digraph}, for short) is a tuple $\Ati =(\Verts, \Edgs, \init,\term)$, where $\Verts$ and $\Edgs$ are disjoint sets (called the \defin{set of vertices of $\Ati$} and the \defin{set of arcs of $\Ati$}, respectively), and $\init,\term \colon \Edgs \to \Verts$ are maps assigning to each arc in $\Ati$ its initial and terminal vertex, respectively.

A \defin{walk} in a digraph $\Ati$ is a finite alternating sequence $\walki = \verti_0 \edgi_1 \verti_1 \ldots \edgi_{l} \verti_{l}$ of successively incident vertices and arcs (\ie such that $\init{\edgi_i} = \verti_{i-1}$ and $\term{\edgi_i}=\verti_{i}$ for $i = 1,\dots,l$). Then, $\verti_0$ and $\verti_l$ are called the \defin{initial} and \defin{final} vertices of $\walki$, we say that $\walki$ is a walk from $\verti_0$ to $\verti_l$ (a \defin{$(\verti_0,\verti_l)$-walk} for short), and we write $\walki \colon \smash{\verti_0 \xwalk{\ } \verti_l}$. We write $\verti \xwalk{\ } \vertii$ if there exists a walk from $\verti$ to $\vertii$. If the first and last vertices of $\walki$ coincide then we say that $\walki$ is a \defin{closed walk}. A closed walk from $\verti$ to $\verti$ is called a \defin{$\verti$-walk}. The \defin{length of a walk} is the number of arcs in the sequence, counting possible repetitions, namely $l$. The walks of length $0$ are called \defin{trivial walks}, and correspond precisely to the vertices in $\Ati$.

\begin{defn}
An \defin{$\Alfi$-labelled directed graph} (an \defin{$\Alfi$-digraph}, for short) is a pair $(\Ati,\lab)$, where $\Ati = (\Verts, \Edgs, \init,\term)$ is a digraph, and $\lab \colon \Edgs \to \Alfi$ is a map assigning to every arc in~$\Ati$ a label from some set $\Alfi$ called \defin{alphabet}, whose elements are called \defin{symbols} (or \defin{letters}).
\end{defn}

If $\edgi$ is an arc from $\verti$ to $\vertii$ with label $\lab(\edgi) = \alfi$, then we write $\verti \xarc{\alfi\,} \vertii$ and we say that $\edgi$ is an \defin{$\alfi$-arc} of $\Ati$. An $\Alfi$-labelling on arcs extends naturally to an $\Alfi^{*}$-labelling on walks by concatenating the corresponding arc labels; \ie if $\walki = \verti_0 \edgi_1 \verti_1 \cdots \edgi_{l} \verti_{l}$ is a walk on an \mbox{$\Alfi$-digraph}, then $\lab(\walki) = \lab(\edgi_1) \cdots \lab(\edgi_{l})\in \Alfi^{*}$, and the label of any trivial walk is the empty word~$\emptyword$. Then, we write $\walki\colon \verti_0 \xwalk{_{\scriptstyle{w}}} \verti_l$, where $w=\lab(\walki)$.

\begin{defn}
Let $\Ati$ be an $\Alfi$-digraph and let $P,Q$ be subsets of vertices in $\Ati$. Then, the set of words read by walks from vertices in $P$ to vertices in $Q$ is called the \defin{language from $P$ to $Q$ (in $\Ati$)}, and is denoted by $\Lang_{PQ}(\Ati)$. For simplicity, if $\verti,\vertii \in \Verts\Ati$, then the languages from $\{ \verti \}$ to $\{ \vertii \}$, and from $\{ \verti \}$ to $\{ \verti \}$, are denoted by $\Lang_{\verti \vertii}(\Ati)$ and $\Lang_{\verti}(\Ati)$, respectively.  
\end{defn}

\begin{defn}
An \defin{$\Alfi$-automaton} is a tern $\Ati_{\!PQ} = (\Ati,P,Q)$, where $\Ati$ is an $\Alfi$-digraph, and $P$ and $Q$ are distinguished nonempty sets of vertices of $\Ati$, called the sets of \defin{initial} and \defin{terminal} vertices of $\Ati_{\!PQ}$, respectively. A walk in $\Ati_{\!PQ}$ is said to be \defin{successful} if it starts at a vertex in $P$ and ends at a vertex in $Q$. Similarly, a word $w\in X^*$ is said to be \defin{successful} (in $\Ati_{\!PQ}$) if it is the label of some successful walk. The \defin{language recognized by  $\Ati_{\!PQ}$} is $\Lang(\Ati_{\!PQ})=\Lang_{PQ}(\Ati)$; \ie the set of successful words in $\Ati_{\!PQ}$. If an automaton $\Ati_{\!PQ}$ has a unique initial vertex, this vertex is called the basepoint of $\Ati$, usually denoted by $\bp$ (\ie $P=\set{\bp}$). An automaton $\Ati_{\!PQ}$ is said to be \defin{pointed} if it has a unique common initial and terminal vertex (\ie if $P=Q=\set{\bp}$); in this case, we write $\Ati_{\!PQ}=\Ati_{\!\bp}$, or even $\Ati_{\!\bp}= \Ati$ if the basepoint is clear.
\end{defn}
 
\begin{defn}
An $\Alfi$ automaton is said to be \defin{saturated} (or \defin{complete}) if for every vertex $\verti$ in $\Ati$ and every letter $\alfi \in \Alfi$, there is an $x$-arc leaving $\verti$.
\end{defn}

\begin{defn}
An $\Alfi$-automaton $\Ati$ is called \defin{deterministic} if no two different arcs with the same label leave the same vertex. That is, if for every vertex $\verti$, and every pair of arcs $\edgi,\edgi'$ leaving $\verti$, $\lab(\edgi) = \lab(\edgi')$ implies $\edgi = \edgi'$. 
\end{defn}

If $\Ati$ is deterministic, then for every vertex $\verti$ in $\Ati$ and every word $w\in \Alfi^*$ there is at most one walk in $\Ati$ reading $w$ from $\verti$; we denote by $\verti w$ its final vertex in case it exists (otherwise, $\verti w$ is undefined). Note that if $\Ati$ is also saturated then $\verti w$ is defined for every vertex $\verti$ and every word $w$.

Recall that, for an alphabet $\Alfi$, we write $\Alfi^{-1} = \set {\alfi^{-1} \st \alfi \in \Alfi}$ the \defin{set of formal inverses} of~$\Alfi$, and $\Alfi^{\pm} = \Alfi \sqcup \Alfi^{-1}$ the \defin{involutive closure} of $\Alfi$. Understanding $(x^{-1})^{-1}=x$, an alphabet $\Alfi$ is said to be \defin{involutive} if~$\Alfi^{\pm} =\Alfi$.

\begin{defn} \label{def: involutive automaton}
An \defin{involutive $\Alfi$-automaton} is an $\Alfi^{\pm}$-automaton
with a labelled involution
${\edgi \to \edgi^{-1}}$
on its arcs; \ie to every arc
$\smash{\edgi \equiv \verti \xarc{\,\alfi\ }\vertii}$ we associate a unique arc $\smash{\edgi^{-1} \equiv \verti \xcra{\,\alfi^{\text{-}1}\!} \vertii}$
(called the \defin{inverse} of $\edgi$) such that $(\edgi^{-1})^{-1} = \edgi$.
That is, in an involutive automaton $\Ati$, arcs appear by (mutually inverse) pairs. 
\end{defn}

\begin{rem} \label{rem: automata representation}
We usually represent involutive $X$-automata through their positive (\ie $X$-labelled) part, with the convention that and arc $\verti \xarc{x} \vertii$ reads the inverse label $x^{-1}$ when crossed backwards (\ie from $\vertii$ to $\verti$).
\end{rem}

If we ignore the labelling and identify all the mutually inverse pairs in an involutive automaton $\Ati$, we obtain an undirected graph called the \defin{underlying graph} of $\Ati$. Involutive automata inherit terminology from its underlying graph; for example, the diameter $\diam(\Ati)$ of $\Ati$ is defined to be the diameter of its underlying graph, and we say that an involutive automaton is `connected', `a tree', etc. if its underlying graph is so. Similarly, the degree of a vertex \emph{in an involutive digraph} refers to its degree, \ie the number of edges incident to it, in the underlying undirected graph.

A walk in an involutive automaton is said to present \defin{backtracking} if it has two successively inverse labelled arcs. A walk without backtracking is said to be \defin{reduced}.

If $\Ati$ is a pointed and involutive $\Alfi$-automaton, then it is easy to see that $\red{\Lang}(\Ati)$ (the free reduction of the language recognized by $\Ati$) describes a subgroup of the free group ${\Free =\pres{\Alfi}{-}}$. We call it the \defin{subgroup recognized} by $\Ati$, and we denote it by~$\gen{\Ati}$; that is, $\gen{\Ati}=\red{\Lang}(\Ati)\leqslant \Free_n$. It is well known that every subgroup of $\Fn$ admits such a description, which can be made unique after adding natural conditions on the involved automata.

The following lemma is straightforward to prove and will be important for us.

\begin{lem} \label{lem: translated subgroup}
Let $\Ati$ be an involutive and deterministic $\Alfi$-digraph, let $\verti,\, \vertii$ be two vertices of $\Ati$, and let $u$ be the label of a walk from $\verti$ to $\vertii$. Then, $\verti u=\vertii$, $\red{\Lang}_{\verti \vertii}(\Ati)=\gen{\Ati_{\! \verti}}u=u\gen{\Ati_{\! \vertii}}$, and  $\gen{\Ati_{\!\vertii}} = \gen{\Ati_{\!\verti}}^{u}$. \qed
\end{lem}

\begin{defn}
If a graph (or automaton) $\Ati$ can be obtained by identifying a vertex $\verti$ of a tree $\Treei$ with a vertex of some graph $\Grii$ disjoint with $\Treei$, then we say that $\Treei$ is a \defin{hanging tree} of~$\Ati$, and that $\Atii$ is obtained from $\Ati$ after \defin{removing the hanging tree}~$\Treei$.
\end{defn}

\begin{defn}
An involutive automaton $\Ati$ is said to be \defin{core} if every vertex
in $\Ati$ appears in some reduced successful walk. The core of an involutive automaton $\Ati$, denoted by $\core(\Ati)$ is the maximal core subautomaton of $\Ati$ containing all the initial and terminal vertices of $\Ati$.

Note that an involutive automaton $\Ati$ is core if and only if it is connected and
all its hanging trees contain a terminal vertex different from the identified one.
The \defin{core} of a pointed involutive automaton $\Ati$, denoted by $\core(\Ati)$, is the (pointed and involutive) automaton obtained after taking the connected component of $\Ati$ containing the basepoint, and removing from it all the possible non-trivial hanging trees not containing the basepoint, and also all those containing the basepoint as the identified vertex. It is clear that $\gen{\core (\Ati)} = \gen{\Ati}$. Finally, a pointed and involutive automaton is called \defin{reduced} if it is both deterministic and core.
\end{defn}

As suggested by the previous characterization, $\core(\Ati)$ can still have some hanging trees, but mandatorily containing terminal points at all their leaves (\ie vertices of degree one).

\begin{defn}
The \defin{restricted core} of an automaton $\Ati$, denoted by $\core^*(\Ati)$, is the labelled digraph obtained after ignoring all the initial an terminal vertices and removing all the hanging trees from $\Ati$. It is clear that $\core^*(\Ati) \subseteq \core ({\Ati})$.
\end{defn}

If not stated otherwise, from now on, the automata appearing in this section will be involutive, and deterministic; we will refer to them simply as `automata'. More precisely, Schreier and Stallings automata, defined below, will have special prominence throughout the article.

\begin{defn}
Let $\Free$ be a free group with basis $X$. The  \defin{(right) Schreier automaton} of $H$ \wrt $X$, denoted by $\schreier{H,X}$, is the automaton having $H\backslash \Free$ (the set of right cosets of $\Free$ modulo $H$) as vertex set, an arc $Hw \xarc{x\,} Hwx$ for every coset $Hw \in H \backslash \Free$ and every element $x\in X^{\pm 1}$, and the coset $H$ as the only initial and terminal point. At some point of the discussion we shall consider a variation of the Schreier automaton with $H$ as initial vertex  and any of the vertices, say $Hu$, as the (only) terminal one; we denote it by $\operatorname{Sch}_{\bp,Hu}(H,X)$. Note that $\operatorname{Sch}_{\bp,H}(H,X)=\schreier{H,X}$. 
\end{defn}

Note also that Schreier automata are involutive, deterministic, connected, and saturated, but not necessarily core.
The core of $\schreier{H,X}$ is a reduced (involutive and pointed) \mbox{$X$-automaton} called the \defin{Stallings automaton} of~$H$ (\wrt $X$), and denoted by $\Stallings{H,S}$; that is, $\stallings{H,X} = \core(\schreier{H,X})$. Note that $\gen{\schreier{H,X}}=\gen{\Stallings{H,X}}=H$ and that all vertices in $\stallings{H,X}$ have degree bigger than one except maybe the basepoint. 

Similarly, the core of $\operatorname{Sch}_{\bp,Hu}(H,X)$ is denoted by $\stallings{Hu,X}$, we have $\gen{\operatorname{Sch}_{\bp,Hu}(H,X)}=\gen{\Stallings{Hu,X}}=Hu$, and all vertices in $\Stallings{Hu,X}$ have degree bigger than one, except maybe the basepoint or $Hu$. 

Finally, the restricted core of $\stallings{H,X}$ is called the \defin{restricted Stallings automaton of $H$} \wrt $X$, and denoted by $\rstallings{H,X}=\core^*(\stallings{H,X})=\core^*(\schreier{H,X})$. 

If the chosen basis $X$ is clear from the context, we usually drop the reference to it and just write $\schreier{H}$, $\stallings{H}$, $\rstallings{H}$, etc. 

In the seminal paper \cite{stallings_topology_1983}, \citeauthor{stallings_topology_1983} proved that if we restrict to deterministic and core automata, then the description of subgroups of $\Free$ by automata is essentially unique.

\begin{thm}[\citenry{stallings_topology_1983}]\label{thm: Stallings bijection}
Let $\Free$ be a free group with basis $X$. Then, 
 \begin{equation}\label{eq: Stallings bijection}
\begin{array}{rcl} \operatorname{St}\colon \set{\,\text{subgroups of } \Free \,} & \leftrightarrow & \set{\,\text{(isomorphic classes of) reduced $X$-automata}\,} \\ H & \mapsto & \stallings{H,X} \coloneqq \core (\schreier{H,X}) \\ \gen{\Ati} & \mapsfrom & \Ati \end{array}
 \end{equation}
is a bijection. Furthermore, finitely generated subgroups correspond precisely to finite automata and, in this case, the bijection is algorithmic.
\end{thm}

The bijection \eqref{eq: Stallings bijection} being algorithmic means that there is an algorithm which, given a finite family of words $W\subseteq \Free$, produces $\stallings{\gen{W},X}$; and, conversely, given a finite reduced $X$-automaton $\Ati$, it computes a free basis for $\gen{\Ati}$. Roughly speaking, the first algorithm consists of drawing the flower automaton with a basepoint and a petal spelling each $w\in W$, and then folding; and the second one consists on choosing a spanning tree $\Treei$ in $\Ati$, and computing the free basis of $\gen{\Ati}$ given by $\set{\, \lab(\bp \xwalk{_{\scriptscriptstyle{T}}} \! \bullet \! \arc{\edgi} \!\bullet\! \xwalk{_{\scriptscriptstyle{T}}} \bp) \st \edgi \in \Edgi^{+}\Ati \setmin \Edgi \Treei \,}$, where $\verti \xwalk{_{\scriptscriptstyle{T}}} \vertii$ denotes the unique reduced walk from $\verti$ to $\vertii$ using only arcs in $\Treei$, and $\Edgi^{+} \Ati$ denotes the set of arcs in $\Ati$ labelled by elements in $X$; see~\cite{kapovich_stallings_2002, miasnikov_algebraic_2007, stallings_topology_1983} for details and proofs.

The interpretation of subgroups as automata (and their computability in the finitely generated case) given by \Cref{thm: Stallings bijection} has proved to be extremely fruitful and has become one of the main tools for studying finitely generated subgroups of a free group; see for example~\cite{stallings_topology_1983,
kapovich_stallings_2002,
miasnikov_algebraic_2007,
silva_algorithm_2008,
puder_primitive_2014,
delgado_lattice_2020,
margolis_closed_2001,
ventura_fixed_1997,
roig_complexity_2007}. For example, it follows from the definitions that the membership problem for free groups $\MP(\Free)$ is decidable by just checking whether (the freely reduced form of) the candidate element is the label of some $\bp$-walk in (the finite automaton) $\Stallings{H,X}$. It is also easy to see that a finitely generated subgroup $H \leqslant \Free$ has finite index if and only if $\stallings{H,X}$ is saturated, and hence the \defin{finite index problem} $\FIP(\Free)$ is computable. Other more elaborated applications include the study of algebraic extensions (see~\cite{miasnikov_algebraic_2007}) and of intersections. In this last case, the key fact is that the Stallings automaton of the intersection of two subgroups of $\Free$ is essentially the tensor (or categorical) product of the Stallings automata of the intersecting subgroups (see~\Cref{exm: computing preorder} below). From this fact one can immediately derive the Howson property for free groups: the intersection of any two finitely generated subgroups of $\Free$ is again finitely generated.  

Recall that we have defined different variations of Stallings automata: if $\Ati_{\!H}$ denotes the standard Stallings automaton of a subgroup $H \leqslant \Free$ and $u \in \Free$, then we denote by~$\Ati^*_{\!H}$ and~$\Ati_{\!Hu}$ the restricted and coset versions, respectively.

\begin{exm}
Let ${H = \gen{b a ^3 b^{-1}, b a b^{-1} a b^{-1}} \leqslant \Free[\set{a,b}]}$. Then:
\vspace{-10pt}
\begin{figure}[H]
    \centering
    \begin{tikzpicture}[baseline=(0.base),shorten >=1pt, node distance= 0.5 cm and 1cm, on grid,auto,>=stealth']
    
    \newcommand{\dx}{1.2}
    \newcommand{\dy}{1.2}
    \node[state, accepting] (1) {};

    \node[state, right = \dx of 1] (2) {};    
    \node[state, above right = \dy/2 and \dx of 2] (21) {};
    \node[state, below right = \dy/2 and \dx of 2] (22) {};
    \node[below = 1*\dy of 2] (st) {$\stallings{H,\set{a,b}}$};
    
            
    \path[->]
        (1) edge[red]
                node[pos=0.5,above=-.3mm] {$b$}
            (2);
            
    \path[->]
        (2) edge[blue]
                node[pos=0.5,above left=-0.05] {$a$}
            (21);
            
    \path[->]
        (21) edge[blue, bend right]
            (22);
            
    \path[->]
        (22) edge[red, bend right]
            (21);
            
    \path[->]
        (22) edge[blue]
            (2);
                
    \end{tikzpicture}
        \qquad
    \begin{tikzpicture}[baseline=(0.base),shorten >=1pt, node distance= 0.5 cm and 1cm, on grid,auto,>=stealth']
    
    \newcommand{\dx}{1.2}
    \newcommand{\dy}{1.2}

    \node[state] (2) {};
    \node[below right = 1*\dy and 0.5*\dx of 2] (st) {$\rstallings{H,\set{a,b}}$};
    \node[state, above right = \dy/2 and \dx of 2] (21) {};
    \node[state, below right = \dy/2 and \dx of 2] (22) {};
            
    \path[->]
        (2) edge[blue]
            (21);
            
    \path[->]
        (21) edge[blue, bend right]
            (22);
            
    \path[->]
        (22) edge[red, bend right]
            (21);
            
    \path[->]
        (22) edge[blue]
            (2);
            
    \end{tikzpicture}
    \hspace{20pt}
    \begin{tikzpicture}[baseline=(0.base),shorten >=1pt, node distance= 0.5 cm and 1cm, on grid,auto,>=stealth']
    
    \newcommand{\dx}{1.2}
    \newcommand{\dy}{1.2}
    \node[state, accepting] (1) {};

    \node[state, right = \dx of 1] (2) {};    
    \node[state, above right = \dy/2 and \dx of 2] (21) {};
    \node[state, above right = \dy/3 and \dx/1.5 of 21] (211) {};
    \node[state, below right = \dy/2 and \dx of 2] (22) {};
    \node[below right = \dy and \dx/3 of 2] (st) {$\stallings{H\, bab,\set{a,b}}$};

    \path[->]
        (1) edge[red]
                node[pos=0.5,above=-.3mm] {$b$}
            (2);
            
    \path[->]
        (2) edge[blue]
                node[pos=0.5,above left=-0.05] {$a$}
            (21);
            
    \path[->]
        (21) edge[red]
                node[pos=0.5,above=-.3mm] {}
            (211);
            
    \path[->]
        (21) edge[blue, bend right]
            (22);
            
    \path[->]
        (22) edge[red, bend right]
            (21);
            
    \path[->]
        (22) edge[blue]
            (2);
                
    \end{tikzpicture}
    \label{fig: Stallings variants}
\end{figure}
\end{exm}

\begin{rem}
It follows from the definitions that a subgroup $H \leqslant \Free$ is finitely generated if and only if (any of) the automata $\stallings{H}$, $\rstallings{H}$, and $\stallings{Hu}$ (for all $u \in \Free$) are finite; in this case, all of them are clearly computable.
\end{rem}

\subsection{Relative spectra in free groups}

In the present subsection, we use Stallings description to study spectra of subgroups and cosets of free groups, and related questions. Below, we consider some geometric counterparts of the algebraic notions introduced in \Cref{sec: orders & spectrum}.

\begin{defn}
Let $\Ati$ be an automaton. An \defin{open trail} in $\Ati$ is a (finite or infinite) sequence $\vverti = (\verti_i)_{i \geq 0}$ of \emph{pairwise different} vertices in $\Ati$. More formally, an open trail is an injective map, $\vverti \colon I\to \Vertexs \Ati$, $i\mapsto \verti_i$, from an initial subsequence $I$ of $\NN$ to the set of vertices of $\Ati$.  If an open trail is finite, \eg $\vverti = (\verti_0, \verti_1,\ldots, \verti_{k})$, we define the \defin{closure} of $\vverti$ , denoted by $\overline{\vverti}$, as the sequence obtained after appending the initial vertex at the end of $\vverti$, \ie $\overline{\vverti} = (\verti_0, \verti_1,\ldots, \verti_{k}, \verti_{0})$ (now injective except for the first and last positions). Closures of finite open trails are called \defin{closed trails}. If we want to highlight the first vertex $\verti_0$ in a trail $\vverti$, we say that $\vverti$ is a \defin{$\verti_0$-trail}, and if $\vverti = (\verti_0,\ldots,\verti_k)$ is a finite trail, then we say that $\vverti$ is a trail from $\verti_0$ to $\verti_k$, or a $(\verti_0,\verti_k\!)$-trail, for short.
Sometimes, we abuse language and identify a trail $\vverti$ with their image; for example we may write $\vverti \subseteq \Vertexs \Ati$; or, given a vertex $\vertii$, we write $\vertii \in \vverti$ (\resp $\vertii \notin \vverti$) to express that $\vertii$ appears (\resp does not appear) as an element in $\vverti$.
\end{defn}

\begin{defn}
The \defin{length} of a (open or closed) trail $\vverti$, denoted by $\ell(\vverti)$, is the cardinal of its domain minus one, \ie $\ell(\vverti) = \card{I} - 1$. A trail is said to be \defin{finite} or \defin{infinite} accordingly.
\end{defn}

The following are examples of trails which will play a central role in our arguments. Note that if $\Ati$ is a deterministic and complete $X$-automaton, then the sequence $(\verti w^{i})_{i\in \NN}$ is well defined for every vertex $\verti$ in $\Ati$ and every word $w\in \Free[X]$. Moreover, in case of having repeated vertices, the first repeated vertex in $(\verti w^{i})_{i\in \NN}$ must be the initial vertex $\verti$.

\begin{defn} \label{def: orbit}
Let $\Ati$ be a 
deterministic and complete $X$-automaton, let $\verti$ be a vertex in~$\Ati$, and let $w\in \Free$. The \defin{(full) $\verti$-orbit} of $w$ in $\Ati$, denoted by $\orb_{\Ati}[\verti](w)$, is 
either the full sequence $(\verti w^{i})_{i\in \NN}$ if it has no repeated vertices, or the closure of the largest initial subsequence of $(\verti w^i)_{i\geq 0}$ with no repeated vertices, otherwise. That is, $\orb_{\Ati}[\verti](w)$ is the subsequence of $(\verti w^{i})_{i\in \NN}$ obtained after stopping at the first repeated vertex (if any).
Moreover, if $\vertii$ is any vertex in $\orb_{\Ati}[\verti](w)$ different from $\verti$, we define the \defin{orbit of $w$ between $\verti$ and $\vertii$}, denoted by~$\orb_{\Ati}[\verti,\vertii](w)$, to be the subsequence of $\orb_{\Ati}(\verti,w)$ starting at $\verti$ and ending at $\vertii$.
Note that orbits (of any kind) are always trails, and that
for every $i \in \NN$, we have $\verti w^i \xwalk{_{\scriptstyle{w}}} \verti w^{i+1}$. Basepoint-orbits are simply called \defin{orbits} and we write $\orb_{\Ati}[\bp](w)=\orb_{\Ati}(w)$. Also, for a subgroup $H\leqslant \Free_X$, we abuse language and write $\orb_{H}$ instead of $\orb_{\schreier{H,X}}$, and even omit the reference to the subgroup (or the automaton) if it is clear from the context.
\end{defn}

\begin{defn} \label{def: preorbit}
Let $\Ati$ be a 
deterministic and complete $X$-automaton, and let $\vverti = (\verti_{i})_{i\in{I}}$ be a (open or closed) trail in~$\Ati$. Then, the set of elements in $\Free$ readable between any pair of successive vertices in $\vverti$ is called the \defin{preorbit} of~$\vverti$ in $\Ati$, and is denoted by $\ofo{\Free}{\vverti}{\Ati}$ (or just $\ofo{\Free}{\vverti}{H}$ if $H\leqslant \Free$ and $\Ati=\schreier{H,X}$). That is,
\begin{equation}
 \ofo{\Free}{\vverti}{H}
  \,=\,
 \big\{\, w\in \Free \st \verti_{i-1} \xwalk{_{\scriptstyle{w}}}
 \verti_{i}\,,\,
 \forall i\in I\setmin\set{0} \,\big\}  \,.
\end{equation}
\end{defn}

Of course, preorbits may be empty. When they are not, an algebraic description of them follows easily from~\Cref{lem: translated subgroup}.

\begin{lem} \label{lem: preorbit}
Let $\Ati$ be a deterministic, complete, and connected $X$-automaton, and let $\vverti =(\verti_i)_{i\in I}$ be a trail in~$\Ati$. Then, the preorbit of $\vverti$ in $\Ati$ is either empty or a coset of the form
 \begin{equation} \label{eq: preorbit}
\ofo{\Free}{\vverti}{H} \,=\, \Big( \bigcap\nolimits_{i\in I\setmin \set{0}} H^{w^{i-1}} \Big) \, w \, ,
 \end{equation}
where $H=\gen{\Ati_{p_0}}\leq \Free_X$, and $w\in \Free_X$.
\end{lem}

\begin{proof}
According to \Cref{def: preorbit}, an element $w\in \Free$ belongs to the preorbit of a trail $\vverti $ if and only if $w$ labels a walk between any two successive vertices in $\vverti$. That is, using \Cref{lem: translated subgroup},  
 \begin{equation} \label{eq: preorbit = int cosets}
\ofo{\Free}{\vverti}{H} \,=\, \bigcap\nolimits_{i\in I\setmin \set{0}} \red{\Lang}_{\verti_{i-1} ,\verti_{i}} (\Ati) \,=\, \bigcap\nolimits_{i\in I\setmin \set{0}} \gen{\Ati_{\!\verti_{i-1}}} \, u_{\verti_{i-1} ,\verti_{i}} ,
 \end{equation}
where $u_{\verti, \vertii}\in \Free$ denotes an arbitrary element labelling a walk from $\verti$ to $\vertii$ (which always exists by connectedness of $\Ati$). Therefore, the preorbit \smash{$\ofo{\Free}{\vverti}{H}$} is an intersection of certain cosets of subgroups of $\Free$. It is well known that such an intersection is either empty, or a coset of the intersection of the corresponding subgroups. In this second case, if $w$ is any coset representative of the intersection~\eqref{eq: preorbit = int cosets}, then $\verti_i = \verti_0 w^i$ for all $i \in I$ and hence, again by~\Cref{lem: translated subgroup}, 
 \begin{equation*}
\ofo{\Free}{\vverti}{H} \,=\, \Big( \bigcap\nolimits_{i\in I\setmin \set{0}} \gen{\Ati_{\!\verti_0 w^{i-1}}} \Big) w \,=\, \Big( \bigcap\nolimits_{i\in I \setmin \set{0}} \gen{\Ati_{\!\verti_0}}^{ w^{i-1}} \Big) w\,. \qedhere\tag*{\qed}
 \end{equation*}
\end{proof}

Note that expression \eqref{eq: preorbit} is not computable in general since the subgroup $H$ may not be finitely-generated, and the trail $\vverti$ may have infinite length. Below, we recall two well-known algorithmic properties of $\Free$ to prove that these are the only obstructions for the computability of preorbits in the free group.

\begin{lem} \label{lem: preorbit computable}
Let $H$ be a finitely generated subgroup of $\Free$, and let $\vverti$ be a finite trail of length $k\geq 1$ in \smash{$\schreier{H}$}. Then, the preorbit $\ofo{\FF}{\vverti}{H}$ is computable. More precisely, given a finite generating set for $H$, and coset representatives $(u_0,\ldots ,u_k)$ for the vertices in $\vverti = (\verti_0,\ldots, \verti_k)$, it is decidable whether $\ofo{\FF}{\vverti}{H}$ is empty or not and, in the negative case, \ie when
 \begin{equation} \label{eq:  preorbit computable}
\ofo{\Free}{\vverti}{H} =\,\big(H^{u_0}\cap H^{u_0w} \cap \cdots \cap H^{u_0w^{k-1}} \big) w
 \end{equation}
for some $w\in \Free$, then we can compute both such an element $w$, and a free basis for the intersection
$H^{u_0}\cap H^{u_0w} \cap \cdots \cap H^{u_0w^{k-1}}$.
\end{lem}

\begin{proof}
Applying~\Cref{lem: preorbit} to $\Ati=\schreier{H}$ we reach~\eqref{eq: preorbit = int cosets}, where the intersection is now finite (it has $k$ terms) and involves only a finite portion of $\schreier{H}$. Since $\gen{\Ati_{\!\verti_i}} = H^{u_i}$, for every $i=0,\ldots,k$, the emptiness of the preorbit $\ofo{\Free}{\vverti}{H}$ (and the quest for a witness $w$ in case it exists) is just an instance of the coset intersection problem~$\CIP(\Free)$, which is well known to be solvable; see~\cite{kapovich_stallings_2002}. And, in the case this intersection is nonempty, we have $\ofo{\Free}{\vverti}{H} =\big(H^{u_0}\cap H^{u_0w} \cap \cdots \cap H^{u_0w^{k-1}} \big) w$, where $w$ is the already computed witness, and a free basis for the intersection subgroup (of finitely many finitely generated subgroups of $\Free$) is well known to be computable using products of Stallings automata; see~\cite{kapovich_stallings_2002}.
\end{proof}

\begin{lem} \label{lem: order = orbit}
Let $H\leqslant \Free$, and let $u,w\in \Free$. Then, $\Ord_{Hu}(w)>0$ if and only if $Hu$ appears in the sequence $(\bp w^i )_{i\geq 1}$ and, in this case, $\Ord_{Hu}(w)= \ell(\orb_H[H,Hu](w))$.
\end{lem}

\begin{proof}
Since $\schreier{H}$ is deterministic and complete, all the vertices $\bp w^i \ (i\in \NN)$ are well defined. Now, $\Ord_{Hu}(w)=k\geq 1$ if and only if $w,\ldots ,w^{k-1} \not\in Hu$ and $w^k\in Hu$, \ie if and only if $\orb_{H}[H,Hu](w)=(\bp,\bp w, \ldots ,\bp w^{k})$; in this case, $\ell(\orb_{H}[H,Hu](w))=k ={\Ord_{H}(w)}$. And otherwise, $\Ord_{Hu}(w)=0$, which means that $w^i\not\in Hu$ for every $i\geqslant 1$; that is, the coset $Hu$ does not appear in the infinite sequence $(\bp w^i)_{i\geq 1}$.
\end{proof}

The particular case of subgroups (\ie when $u\in H$) can be stated as follows. 

\begin{cor}
Let $H\leqslant \Free$, and let $w\in \Free$. Then, $\Ord_{H}(w)>0$ if and only if $\ell(\orb_{H}(w))<\infty$; and, in this case, $\Ord_{H}(w)= \ell(\orb_{H}(w))$. \qed
\end{cor}

From \Cref{lem: order = orbit} it is clear that preorders are nothing more than disjoint unions of certain preorbits.

\begin{cor} \label{cor: coset preorder = infinite union}
Let $H\leqslant \Free$, $u\in \Free$, and $k\geq 1$. Then, the set of elements of order $k$ in~$Hu$ is
 \begin{equation} \label{eq: coset preorder = infinite union}
\ofo{\Free}{k}{Hu} \,=\, \bigsqcup\nolimits_{\vverti} \ofo{\Free}{\vverti}{H} \, ,  
 \end{equation}
where the union goes over all the $(\bp, Hu)$-trails $\vverti$ of length $k$ in $\Schreier{H}$.
And for $k=0$, we have that $\ofo{\Free}{0}{Hu} \,=\, \bigsqcup_{\vverti} \ofo{\Free}{\vverti}{H}$, where now the union goes over all the infinite, and all the finite closed $\bp$-trails in $\Schreier{H}$, not containing the coset~$Hu$ apart from the starting vertex. \qed
\end{cor}

In general, the union in~\Cref{eq: coset preorder = infinite union} is infinite (it is finite if and only if $H$ has finite index in $\Free$), with some preorbits $\ofo{\Free}{\vverti}{H}$ possibly empty. A first obstacle for the computability of $\ofo{\Free}{k}{Hu}$ is that, even when $H\leqslant \Free$ is finitely generated and $k\geqslant 1$, these unions may very well contain infinitely many non-empty terms, as seen in the following example.

\begin{exm}\label{exm: surt fora} 
Let $H=\gen{a^2, b} \leqslant \Free[\set{a,b}]$ and let $k=2$. All the words $w_n=b^{-n}ab^n$, $n\geq 0$, have order 2 in $H$ and they all determine different orbits $\orb_{H}(w_n)=(\bp,\bp w_n, \bp)$ since the vertices $\bp$ and $\bp w_n$ are at distance $n+1$ in $\schreier{H}$.
\end{exm}

However, we shall see that, if $H$ is finitely generated and $k\geq 1$, then  the union in \Cref{eq: coset preorder = infinite union} can be reduced to a finite union, making the set of elements of a given order $k\geqslant 1$ computable; see \Cref{prop: coset preorder}. The key point here is to observe that, if we restrict ourselves to cyclically reduced words, then we can bound the orbits contributing non-trivially in~\Cref{eq: coset preorder = infinite union} to finitely many, in the finitely generated case. 

\begin{lem} \label{lem: cr orbits}
Let $H$ be a subgroup of $\Free$ and let $u \in \Free$. If a cyclically reduced word $w \in \Free$ has order $k\geqslant 1$ in $Hu$ then $\orb_{H}[H,Hu](w) \subseteq \stallings{Hu}$. Moreover, in the particular case $u\in H$, we further have $\orb_{H}(w)=\orb_{H}[H,Hu](w) \subseteq \rstallings{H} = \stallings{H}$.
\end{lem}

\begin{proof}
Let $\vverti =\orb[H,Hu](w)=(\verti_{0},\verti_1,\ldots, \verti_{k-1},\verti_{k})$ (where $\verti_{0}=\bp$ and $\verti_{k} = Hu$) and let $\walki_i$ be the reduced walk from $\verti_{i-1}$ to $\verti_{i}$ reading $w$, for $i=1,\ldots ,k$. Since $w$ is cyclically reduced, $w^k$ is reduced and so, the path $\walki =\walki_1 \walki_2 \cdots \walki_{k}$ is reduced as well (\ie there is no backtracking in any of the products $\walki_i \walki_{i+1}$). Since $\walki$ visits each $\verti_i$, we deduce that $\verti_i \in \stallings{Hu}=\core (\operatorname{Sch}_{\bp,Hu}(H))$, for $i=0,\ldots ,k$. 

Finally, if $u \in H$, then $\verti_k = \bp$ and hence both the trail $\vverti$ and the walk $\walki$ are closed. This further implies that there is no backtracking in $\verti_{k-1} \xwalk{_{\scriptstyle{w}}} \bp \xwalk{_{\scriptstyle{w}}} \verti_{1}$ either. Therefore, ${\bp \in \rstallings{H}}$ and the second claimed result follows.
\end{proof}

That is, for $w\in \Free$ or order $k\geq 1$ in $Hu$, the vertices in the orbit $\vverti=\orb_{H}[H,Hu](w)=(\verti_{0},\verti_1,\ldots, \verti_{k-1},\verti_{k})$ belong to $\stallings{Hu}$ \emph{if $w$ is cyclically reduced}, while in general they can go far out of $\stallings{Hu}$, as~\Cref{exm: surt fora} illustrates. The strategy to keep orbits under control, even when $w$ is not cyclically reduced, is the following: let $w=v^{-1}\widetilde{w}v$ be the expression of $w$ as a conjugate of its cyclic reduction $\widetilde{w}$, and define the new vertices $\vertii_i=\verti_i v^{-1}$, for $i=1,\ldots ,k$. Observe that now $\vvertii=(\vertii_{0},\vertii_1,\ldots, \vertii_{k-1},\vertii_{k}) = \orb_{H^{v^{-1}}}[H^{v^{-1}},H^{v^{-1}}u^{v^{-1}}](\widetilde{w})$ and so, by~\Cref{lem: cr orbits}, $\vvertii \subseteq \stallings{H^{v^{-1}}u^{v^{-1}}}\subseteq \stallings{Hu}$. Moreover, if $Hu=H$, then $\vvertii\subseteq \rstallings{H^{v^{-1}}}=\rstallings{H}$ (whereas the vertices in $\vverti$ may very well be out of $\stallings{Hu}$); see~\Cref{fig: proof coset}.

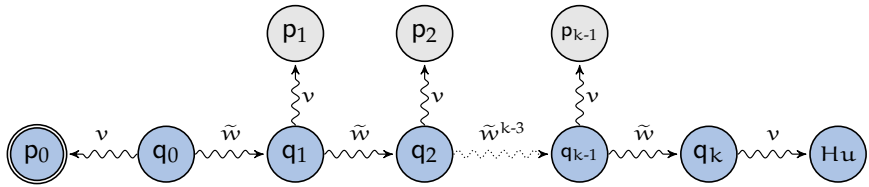
\begin{figure}[H]
    \centering
    \begin{tikzpicture}[shorten >=1pt, node distance=2cm and 2cm, on grid,auto,>=stealth',
decoration={snake, segment length=2mm, amplitude=0.5mm,post length=1.5mm},every state/.style={semithick, fill=gray!20, inner sep=2pt, minimum size = 21pt }]
    
    \newcommand{\dx}{1.7}
    \newcommand{\dy}{1.6}
    \node[state, accepting,fill=RoyalBlue!30] (0) {$\verti_0$};
    \node[state, right = \dx of 0,fill=RoyalBlue!30] (0') {$\vertii_0$};
    \node[state, right = \dx of 0',fill=RoyalBlue!30] (1') {$\vertii_1$};
    \node[state, above = \dy of 1'] (1) {$\verti_1$};
    \node[state, right = \dx of 1',fill=RoyalBlue!30] (2') {$\vertii_2$};
    \node[state, above = \dy of 2'] (2) {$\verti_2$};
    \node[state, right = 1.2*\dx of 2',fill=RoyalBlue!30] (3') {$\scriptstyle{\vertii_{k\text{-}1}}$};
    \node[state, above = \dy of 3'] (3) {$\scriptstyle{\verti_{k\text{-}1}}$};
    \node[state, right = \dx of 3',fill=RoyalBlue!30] (6) {$\vertii_k$};
    \node[state, right = \dx of 6,fill=RoyalBlue!30] (7) {$\scriptstyle{Hu}$};
    
    \path[->]
    (0') edge[snake it] node[pos=0.5,above= 0.1] {$v$} (0)
    (0') edge[snake it] node[pos=0.5,above= 0.1] {$\widetilde{w}$}(1')
    (1') edge[snake it] node[pos=0.5,above= 0.1] {$\widetilde{w}$}(2')
    (2') edge[densely dotted,snake it] node[pos=0.5,above= 0.1] {$\widetilde{w}^{k\text{-}3}$}(3')
    (3') edge[snake it] node[pos=0.5,above= 0.1] {$\widetilde{w}$}(6)
    (6) edge[snake it] node[pos=0.5,above= 0.1] {$v$}(7)
    (1') edge[snake it] node[pos=0.45,right] {$v$} (1)
    (2') edge[snake it] node[pos=0.45,right] {$v$} (2)
    (3') edge[snake it] node[pos=0.45,right] {$v$} (3);
    \end{tikzpicture}
    \caption{Scheme of the orbits $\vverti$ and $\vvertii$ for an element $w = v^{-1} \widetilde{w} v\in \Free$ or order $k\geq 1$ in $Hu$, with the vertices within $\stallings{Hu}$ shaded in blue}
    \label{fig: proof coset}
\end{figure}

\begin{prop} \label{prop: coset preorder}
Let $H$ be a finitely generated subgroup of $\Free$, let $u\in \Free$, and let $k \geq 1$. Then, the set of $\Free$-elements of order $k$ in the coset $Hu$ is either empty or
 \begin{equation} \label{eq: coset preorder}
\ofo{\Free}{k}{\!Hu} \,=\, 
\bigsqcup\nolimits_{\,\vvertii}
\big(
\ofo{\Free}{\vvertii}{\!H}
\big)^{\! v_{\vvertii} H \, \cap \, v'_{\vvertii} H^{u}},
 \end{equation}
where the union goes over the (finitely many) length $k$ trails $\vvertii =(\vertii_0, \vertii_1, \ldots, \vertii_{k})$ within $\stallings{Hu}$, $v_{\vvertii}$ is the label of some walk from $\vertii_0$ to $\bp$, and $v'_{\vvertii}$ is the label of some walk from $\vertii_k$ to $Hu$.
\end{prop}

\begin{proof}
Consider a trail $\vvertii =(\vertii_0, \vertii_1, \ldots, \vertii_{k})$ as in~\Cref{eq: coset preorder}, and let $w\in \ofo{\Free}{\vvertii}{\!H}$ and $v\in v_{\vvertii} H \, \cap \, v'_{\vvertii} H^{u}$. Then $w$ labels a walk from $\vertii_{i-1}$ to $\vertii_i$ for every $i=1,\ldots ,k$, and $v$ labels walks both from $\vertii_0$ to $\bp$, and from $\vertii_k$ to $Hu$. Therefore, for $i=1,\ldots ,k-1$, we have $\bp v^{-1}w^i=\vertii_i\neq \vertii_k$, hence $\bp v^{-1}w^i v=\vertii_i v\neq \vertii_k v=Hu$ and so, $(v^{-1}wv)^i=v^{-1}w^iv \not\in Hu$; whereas $\bp v^{-1}w^k v=\vertii_k v=Hu$ and so, $(v^{-1}wv)^k =v^{-1}w^kv\in Hu$. Hence $\Ord_{Hu}(w^v)=k$, proving the inclusion to the left.

For the opposite inclusion, let $w$ be an $\Free$-element of order $k$ in $Hu$, and consider the orbit $\vverti=\orb_{H}[H,Hu](w)=(\verti_{0},\verti_1,\ldots, \verti_{k-1},\verti_{k})$. Writing $w=v^{-1}\widetilde{w}v$ and $\vvertii=(\vertii_{0},\vertii_1,\ldots, \vertii_{k-1},\vertii_{k})\subseteq \stallings{Hu}$ as above (see~\Cref{fig: proof coset}), we have $\widetilde{w}\in \ofo{\Free}{\vvertii}{\!H}$, $\vertii_0 \xwalk{_{\scriptstyle{v}}} \bp$ and $\vertii_k \xwalk{_{\scriptstyle{v}}} Hu$ (\ie $v\in v_{\vvertii} H \, \cap \, v'_{\vvertii} H^{u}$). Therefore, $w=\widetilde{w}^v\in \big( \ofo{\Free}{\vvertii}{\!H}\big)^{\, v_{\vvertii} H \, \cap \, v'_{\vvertii} H^{u}}$, finishing the proof. 
\end{proof}

We state separately the subgroup case, which we use later and admits a slightly simpler expression.

\begin{prop} \label{prop: sgp preorder}
Let $H$ be a finitely generated subgroup of $\Free$, and let $k\geq 1$. Then, the set of $\Free$-elements of order $k$ in the subgroup $H$ is either empty or
 \begin{equation} \label{eq: order k = conj finite union}
\ofo{\Free}{k}{H} \,=\,
\left(
\bigsqcup\nolimits_{\vvertii}
\big(
\ofo{\Free}{\vvertii}{H}
\big)^{\!v_{\vvertii}}
\right)^{\!H} ,
 \end{equation}
where the union goes over the (finitely many) closed length $k$ trails $\vvertii =(\vertii_0, \vertii_1, \ldots, \vertii_{k-1}, \vertii_k=\vertii_0)$ within the restricted Stallings automaton $\rstallings{H}$, and $v_{\vvertii}$ is the label of any walk from $\vertii_0$ to $\bp$.  
\end{prop}

\begin{proof}
Apply the previous proposition to the case $u\in H$. We have $Hu=\bp$, and $\ofo{\Free}{k}{H}$ is given by~\Cref{eq: coset preorder}, where all the involved trails $\vvertii$ are within $\stallings{Hu}=\stallings{H}$, and are closed (since $\vertii_0 v=\bp$ and ${\vertii_k v=Hu=\bp}$ imply $\vertii_0=\vertii_k$) and hence, $\vvertii\subseteq \rstallings{H}$. On the other hand, $v'_{\vvertii}$ can be taken equal to $v_{\vvertii}$ and so, $v_{\vvertii} H \, \cap \, v'_{\vvertii} H^{u}=v_{\vvertii} H$ and $\big(\ofo{\Free}{\vvertii}{\!H} \big)^{\! v_{\vvertii} H \, \cap \, v'_{\vvertii} H^{u}}=\big( \big(\ofo{\Free}{\vvertii}{\!H} \big)^{\! v_{\vvertii}} \big)^H$. This concludes the proof.
\end{proof}

\begin{cor}\label{cor: SBP Free}
Free groups have bounded coset (and hence subgroup) spectra, and the corresponding bounds are computable
(in particular $\SBP(\Free)$ is computable).
More precisely, 
if $H$ is a finitely generated subgroup of a free group $\Free$ and $u \in \Free$, then
$\Iset{Hu}{\Free}$ is bounded by the (finite) number of vertices in $\stallings{Hu}$, and
$\Iset{H}{\Free}$ is bounded by the (finite) number of vertices in $\rstallings{H}$.
\end{cor}

(Note that the final claim is just a slight improvement of the result by Kapovich and Myasnikov in \cite{kapovich_stallings_2002}, where the spectrum of the free group \wrt a finitely generated subgroup $H$
was bounded by the (finite) number of vertices in $\stallings{H} \supseteq \rstallings{H} $.)

\begin{proof}
Let $H \leqslant\fg \Free$, let $w\in \Free$, and consider $w=v^{-1} \widetilde{w} v$ its expression as a conjugate of its cyclic reduction $\widetilde{w}$. Then, it is clear that~$\Ord_{Hu}(w) = \Ord_{H^{v^{\text{-}1}} u^{v^{\text{-}1}}}(\widetilde{w})$, which is either zero (and so, trivially less than or equal to $\card \Verts \,\stallings{Hu}$), or equals the length of the orbit $\vvertii =\orb[H^{v^{\text{-}1}},H^{v^{\text{-}1}} u^{v^{\text{-}1}}](\widetilde{w})$, contained in $\stallings{H^{v^{\text{-}1}} u^{v^{\text{-}1}}} \subseteq \stallings{Hu}$ by~\Cref{lem: cr orbits}. Then,
\begin{equation}
    \Ord_{Hu}(w)
    \,=\,
    \Ord_{H^{v^{\text{-}1}}u^{v^{\text{-}1}}}(\widetilde{w})
    \,=\,
    \ell(\vvertii )
    \,\leq\,
    \card \Verts \,\stallings{H^{v^{\text{-}1}} u^{v^{\text{-}1}}}
    \,\leq\,
    \card \Verts \,\stallings{Hu} \,,
\end{equation}
as claimed. Finally, note that if $u \in H$ (\ie if $Hu=H$) then $\vvertii$ is the closed orbit of a cyclically reduced word and hence, by~\Cref{lem: cr orbits}, $\vvertii\subseteq \rstallings{H}$. So, $\Ord_{H}(w)\leq \card \Verts \,\rstallings{H}$. 
\end{proof}

From \Cref{cor: SBP Free}, together with the computability of the membership problem $\MP(\Free)$, we obtain the computability of the order function
within free groups. 
\begin{thm}
There is an algorithm that, on input a finitely generated subgroup $H\leqslant \Free$, a coset $Hu$, and word $w\in \Free$,  outputs $\Ord_{Hu}(w)$ (the order of $w$ in $Hu$). In particular, the order problem $\OP(\Free)$ (and hence the torsion problem $\TP(\Free)$) is computable. \qed
\end{thm}

Also, combining the computability of preorbits of finite trails (\Cref{lem: preorbit computable})  with the finiteness of the union in~\Cref{eq: order k = conj finite union}, we obtain a precise algebraic description of the set of elements of a given order $k \geq 1$ in a given finitely generated subgroup (or coset) of $\Free$, together with its computability.

\begin{prop} \label{prop: coset preorder comp}
Let $H$ be a finitely generated subgroup of $\Free$, let $u\in \Free$, and let $k\geq 1$. Then, the set of elements of $\Free$ of order $k$ in the coset $Hu$ is computable. More precisely, given $u, k$ and a finite set of generators for $H$, we can algorithmically decide whether $\ofo{\Free}{k}{Hu} \neq \varnothing$ 
and, if so, compute (finitely many) elements $w_1,\ldots ,w_t\in \Free$ (of order $k$ in $Hu$) and a free basis for $H\cap H^u$, such that
 \begin{equation} \label{eq: coset preorder comp}
\ofo{\Free}{k}{Hu} \,=\, \Big( \bigsqcup\nolimits_{i=1}^{t} \big(H\cap H^{w_i} \cap H^{w_i^2} \cap \cdots \cap H^{w_i^{k-1}}\big) w_i \Big)^{H\cap H^u}.
 \end{equation}
\end{prop}

\begin{proof}
Compute the (finite) coset Stallings automaton $\St =\stallings{Hu}$. From~\Cref{cor: SBP Free}, we know that if $k>\card \Verts \St$, then $\ofo{\Free}{k}{Hu} =\varnothing$. So, we can assume $k\leq \card \Verts \St$. 

Let us apply~\Cref{prop: coset preorder} and, with the notation used there, for each of the finitely many length $k$ trails $\vvertii\subseteq \Ati$, let us check whether $\ofo{\Free}{\vvertii}{H}$ or $v_{\vvertii}H\cap v'_{\vvertii}H^u$ are empty or not: with a solution to the coset intersection problem~$\CIP(\Free)$ (see~\cite{kapovich_stallings_2002}) decide whether $v_{\vvertii}H\cap v'_{\vvertii}H^u\neq \varnothing$ and if so, pick $v\in v_{\vvertii}H\cap v'_{\vvertii}H^u$, \ie such that $\vertii_0 \xwalk{_{\scriptstyle{v}}} \bp$ and $\vertii_k \xwalk{_{\scriptstyle{v}}} Hu$; on the other hand, use~\Cref{lem: preorbit computable} to check the emptiness of $\ofo{\Free}{\vvertii}{H}$, and to obtain a description of the form
 \begin{equation} \label{eq:  preorbit computable2}
\ofo{\Free}{\vverti}{H} =\,\big(H^{v^{-1}}\cap H^{v^{-1}w} \cap \cdots \cap H^{v^{-1}w^{k-1}} \big) w
 \end{equation}
in case it is not empty. If, for every length $k$ trail $\vvertii$ in $\Ati$, either $\ofo{\Free}{\vvertii}{H}$ or $v_{\vvertii}H\cap v'_{\vvertii}H^u$ (or both) are empty, then $\ofo{\Free}{k}{Hu} =\varnothing$ and we are done. Otherwise, for each $\vvertii$ with $\ofo{\Free}{\vvertii}{H}\neq \varnothing$ and $v_{\vvertii}H\cap v'_{\vvertii}H^u =v(H\cap H^u)\neq \varnothing$, we have 
 \begin{align*}
\big( \ofo{\Free}{\vvertii}{\!H} \big)^{\! v_{\vvertii} H \, \cap \, v'_{\vvertii} H^{u}} &= \Big( \big( H^{v^{-1}}\cap H^{v^{-1}w} \cap \cdots \cap H^{v^{-1}w^{k-1}} \big)w \Big)^{v(H\cap H^u)} \\ &= \Big( \big( H\cap H^{v^{-1}wv} \cap \cdots \cap H^{v^{-1}w^{k-1}v}\big) v^{-1}wv \Big)^{H\cap H^u}.
 \end{align*}
Putting together these, say $t$, trails we get the desired expression.  
\end{proof}

Note that, in the subgroup case, \ie when $u\in H$ (equivalently, $Hu=H$), the expression~\eqref{eq: coset preorder comp} is valid as well, with the conjugating subgroup being just $H\cap H^u =H$.

For subgroups $H$, the set of $k$-roots is the disjoint union of the sets of elements of order $d$, with $d$ running over the (finitely many) divisors of $k$; see~\eqref{eq: roots <- orders}. Therefore, we have a description of exactly the same type for $\sqrt[k]{H}$ (which is also algorithmically computable when $H$ is finitely generated). Note that $\sqrt[k]{H}\neq \varnothing$ since it always contains the elements from $H$ themselves (corresponding to the closed length 1 trail $\vvertii =(\bp, \bp)$). Note also that, for $w\in \Free$ of order $d\divides k$ in $H$, we have $w^d\in H$ and so, $H\cap H^w\cap H^{w^2} \cap \cdots \cap H^{w^{d-1}} =H\cap H^{w}\cap H^{w^2} \cap \cdots \cap H^{w^{k-1}}$, allowing a more compact description for $\sqrt[k]{H}$; we restate it here for later reference.

\begin{cor} \label{cor: k-roots of F computable}
Let $H$ be a finitely generated subgroup of $\Free$, and let $k\geq 1$. Then, the set of $k$-roots of $H$ is computable. More precisely, given $k$ and a finite set of generators for $H$, we can compute finitely many elements $\trivial=w_0, w_1,\ldots,w_t \in \Free$ (of order in $H$ being a divisor of $k$), such that
 \begin{equation*} \label{eq: k-roots of F}
\sqrt[k]{H} \,=\, \left(\bigsqcup\nolimits_{i=0}^{t} \big( H\cap H^{w_i}\cap H^{w_i^2} \cap \cdots \cap H^{w_i^{k-1}}\big) w_{i} \right)^{\!H}. \qedhere\tag*{\qed}
 \end{equation*}
\end{cor}

Once we have a description for the set of elements of any order $k\geq 1$, it is straightforward to decide whether there exist elements of order zero (the remaining ones), and to give an explicit description of them in the finite index case.

\begin{prop}\label{prop: order 0 Fn}
Let $H$ be a finitely generated subgroup of $\Free$, and let $u\in \Free$. Then, for $u\not\in H$, $\ofo{\Free}{0}{Hu}$ is always nonempty, and for $u\in H$, $\ofo{\Free}{0}{H} \neq \varnothing$ if and only if $H$ has infinite index in $\Free$. Moreover, in case $u\not\in H$ and $H$ is of finite index in $\Free$, we can compute finitely many elements $\trivial=w_0, w_1,\ldots,w_t \in \Free$, of orders $1=k_0, k_1,\ldots , k_t \geq 1$ in $H$ respectively, such that
 \begin{equation} \label{eq: elements of order 0 (fg)}
\ofo{\Free}{0}{Hu} \,=\, \left( \bigsqcup\nolimits_{i=0}^{t} \big( H\cap H^{w_i} \cap H^{w_i^2} \cap \cdots \cap H^{w_i^{k_i-1}}\big) w_i \right)^H.
 \end{equation}
In particular, the torsion-group problem $\TGP(\Free)$ is computable.
\end{prop}

\begin{proof}
Clearly, if $u\not\in H$, then $H\subseteq \ofo{\Free}{0}{Hu} \neq \varnothing$. In the case $Hu=H$, the implication to the right is a well-known general fact (see~\Cref{cor: spectrum and index}). And, for the converse, suppose that $H$ is a finitely generated subgroup of infinite index in $\Free =\Free_X$, \ie with $\stallings{H,X}$ being finite and unsaturated; take a vertex $Hv\in \stallings{H,X}$ which is $x$-deficient for some generator $x\in X$, take a cyclically reduced word of the form $w=hvx$, for some $h\in H$, and it is clear that $\bp w^k\not\in \stallings{H,X}$ for every $k\geq 1$; therefore, $w\in \ofo{\Free}{0}{H} \neq \varnothing$. 

Let us assume now that $u\not\in H$ and that $H$ is of finite index in $\Free$, \ie $\stallings{Hu}=\stallings{H}=\rstallings{H}$ is finite and saturated. By~\Cref{cor: coset preorder = infinite union} (and taking into account that there are no infinite trails in $\rstallings{H}$), we have $\ofo{\Free}{0}{Hu} \,=\, \bigsqcup_{\vverti} \ofo{\Free}{\vverti}{H}$, where the union goes over all the finite closed $\bp$-trails in $\rstallings{H}$ not containing the coset~$Hu$. Finally, applying~\Cref{lem: preorbit computable} to each such preorbit $\ofo{\Free}{\vverti}{H}$, we can disregard the empty ones, and get the corresponding description for the nonempty ones; putting them together we obtain~\eqref{eq: elements of order 0 (fg)}.
\end{proof}

Finally, the computability of preorders together with the algorithmic boundability of the spectra (\Cref{cor: SBP Free}) provides (through the analysis in \Cref{sec: algorithmic}) the computability of the spectrum function $H \mapsto \Ord_{H}(\Free)$ and related problems. 

\begin{thm} \label{thm: spectrum F computable}
There exists an algorithm which, on input a finitely generated subgroup $H\leq\fg \Free$, and an element $u\in \Free$, it outputs the finite spectrum $\Iset{Hu}{\Free}$. In particular, the finite spectrum problem $\FSP(\Free)$, and hence the spectrum membership problem $\SMP(\Free)$, the torsion-group problem $\TGP(\Free)$, and the purity problem $\PP(\Free)$, are computable for free groups. \qed
\end{thm}

We recall that the purity problem $\PP(\Free)$ was previously proved in~\cite{birget_pspace-complete_2000} and~\cite{kapovich_stallings_2002}, constituting one of the first results in this direction.

Finally, let us point out an easy asymptotic consequence of our results. Recall that a subset $S \subseteq \Free$ is called \defin{negligible} if it has asymptotic density zero, it is called \defin{visible} if it is not negligible, and it is called \defin{generic} if its complement is negligible.  

\begin{rem}
Let $H$ be an infinite index subgroup of $\Free$. Since any coset of $H$ is a negligible subset of $\Free$, and a finite union of negligible subsets is negligible, we deduce from~\Cref{prop: coset preorder comp} that the preorders $\ofo{\Free}{k}{Hu}$ are also negligible for $k\geq 1$. Moreover, from~\Cref{prop: order 0 Fn}, $\ofo{\Free}{0}{H}$ is the complement of a negligible set so it is generic in $\Free$. For a finite index $H\leq \Free$, and $k\geq 0$, each $\ofo{\Free}{k}{Hu}$ is either empty or visible in $\Free$.
\end{rem}

\subsection{Computing pure closures within free groups}

Let $S\subseteq \NN_{\geq 1}$ be any set of strictly positive natural numbers. In this section we study the $S$-pure closure $\pcl[S]{H}$ of a subgroup $H\leq \Fn$. More concretely, we shall see that if $H$ is finitely generated then so is $\pcl[S]{H}$; additionally, if $S$ is finite (more generally, computable) one can effectively compute generators for $\pcl[S]{H}$ from given generators for $H$. 

A first naive approximation to computing the $S$-pure closure of a finitely generated subgroup $H\leqslant \Free$ is to add to $H$ all its (proper) $S$-roots. For each $s\in S$, we can apply~\Cref{cor: k-roots of F computable} and get
 \begin{equation}\label{eq: proper S-roots}
\sqrt[s]{H} \,=\, \left(\bigsqcup\nolimits_{i=0}^{t} \big( H\cap H^{w_i}\cap H^{w_i^2} \cap \cdots \cap H^{w_i^{s-1}}\big) w_{i} \right)^{\!H}, 
 \end{equation}
where, $1=w_0, w_1,\ldots ,w_t$ are certain computable elements from $\Free$ (depending on $s$). Note that, \emph{even if $S$ is infinite}, \Cref{cor: SBP Free} tells us that only a finite subset of $S$ contributes significant roots: $\smash{\sqrt[S]{H}=\sqrt[^{S_H}]{H}}$, where $S_H=S\cap \Iset{H}{\Free}=S\cap [0,\card V\rstallings{H}]$ is finite. Therefore, taking the union of expressions~\eqref{eq: proper S-roots} for $s\in S_H$, we get a similar finite description for $\sqrt[S]{H}$. 

Now, by definition, if $H$ has no proper $S$-roots then $\pcl[S]{H}=H$, and we are done. Otherwise, $\pcl[S]{H}$ must contain all the elements in~\eqref{eq: proper S-roots}; that is $\gen{ \sqrt[S]{H}}\leq \pcl[S]{H}$. However, this inclusion is not necessarily an equality because the attachment of $\sqrt[S]{H}$ to $H$ may create new proper $S$-roots to be attached in a second round, as the following example shows.

\begin{exm} \label{exm: computing preorder}
Consider $\Free[2]=\Free[\{a,b\}]$ and the subgroup $H=\gen{a^2, ab^2}\leq \Free[2]$. Clearly, $a\in \ofo{\Free}{2}{H}$ (and hence $H$ is not $2$-pure), whereas $b\not\in \ofo{\Free}{2}{H}$ (in fact, $\Iset{H}{b}=0$, since no element in $H$ starts with $b$). However, after adding $a$ to $H$, we get $\gen{H \cup \set{a}} = \gen{a^2, ab^2, a}=\gen{a, b^2}$, which is still not $2$-pure because now $\Ord_{\gen{a, b^2}} (b) = 2$.

We complete the example by seeing that $\ofo{\Free}{2}{H} = \set{a^{2n+1} \st n\in \ZZ}^H$; \ie the odd powers of $a$ are the \emph{unique} proper $2$-roots of $H$ up to conjugacy by $H$. Hence, after adding to $H$ \emph{all} its proper $2$-roots we reach the (same) subgroup $\gen{a, b^2}$, which still has $b$ as a proper $2$-root.

To compute $\ofo{\Free}{2}{H}$, let us follow the argument given in~\Cref{prop: coset preorder comp}, applied to the subgroup $H=\gen{a^2, ab^2}\leq \Free[2]$. The Stallings graph $\St_{\!H} =\stallings{H,\{a,b\}}=\rstallings{H,\{a,b\}}$ 
is depicted in~\Cref{fig: Stallings and neighborhood}.

\begin{figure}[H]
\centering
\begin{tikzpicture}[baseline=(0.base),shorten >=1pt, node distance=1.2cm and 2cm, on grid,auto,auto,>=stealth',every state/.style={ 
semithick, fill=gray!20, inner sep=2pt, minimum size = 10pt }]
    \newcommand{\dx}{1.5}
    \newcommand{\dy}{1.5}
    \node[] (0)  {};
    \node[state, accepting,right = \dx of 0] (a1) {$\scriptstyle{0}$};
    \node[state] (a2) [right = \dx of a1] {$\scriptstyle{1}$};
    \node[state] (a3) [right = \dx of a2] {$\scriptstyle{2}$};

    \path[->]
        (a1) edge[thick,blue, bend left = 30]
            (a2);

    \path[->]
        (a2) edge[blue, bend left = 30]
                node[pos=0.5,below = 0.05] {$a$}
            (a1);

    \path[->]
        (a2) edge[red]
                node[pos=0.5,below = 0.05] {$b$}
            (a3);
    \path[->]
        (a3) edge[thick,red,bend right = 40]
            (a1);

    \end{tikzpicture}
\caption{The Stallings automaton $\Ati_{\!H}$
} \label{fig: Stallings and neighborhood}
 \end{figure}
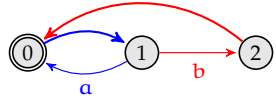

Since $\card{\Vertexs \St_{\!H}} =3$, there are two closed $\bp$-trails of length $2$ in $\St_{\!H}$, namely $(\bp, 1,\bp)$ and $(\bp, 2,\bp)$, and so,  
\begin{equation*}
\ofo{\Free}{2}{H} \,=\, \Big( \ofo{\Free}{(\bp, 1,\bp)}{H} \sqcup \ofo{\Free}{(\bp, 2,\bp)}{H} \Big)^{\!H}.
\end{equation*}
In order to compute the preorbits $\ofo{\Free}{(\bp, 1,\bp)}{H}$ and $\ofo{\Free}{(\bp, 2,\bp)}{H}$, we compute in~\Cref{fig: pullback} the pullback of $\Ati_{\!H}$ with itself, $\Ati_{\!H}\times \Ati_{\!H}$, and read the intersections from there:

\begin{figure}[H]
\centering
\begin{tikzpicture}[shorten >=1pt, node distance=1.2cm and 2cm, on grid,auto,auto,>=stealth',every state/.style={semithick, fill=gray!20, inner sep=2pt, minimum size = 10pt }]

\newcommand{\dx}{1.4}
\newcommand{\dy}{1.3}
\node[] (0)  {};
\node[state,accepting,style={ 
semithick,fill=gray!20,inner sep=2pt,minimum size = 10pt}] (a3) [right = \dy-1/3 of 0] {$\scriptstyle{0}$};
\node[state,style={semithick,fill=gray!20,inner sep=2pt,minimum size = 10pt}] (a4) [right = \dx of a3] {$\scriptstyle{1}$};
\node[state,style={semithick,fill=gray!20,inner sep=2pt,minimum size = 10pt}] (a5) [right = \dx of a4] {$\scriptstyle{2}$};

\node[state,accepting,style={semithick,fill=gray!20,inner sep=2pt,minimum size = 10pt}] (b1) [below = 2*\dy/3 of 0] {$\scriptstyle{0}$};
\node[state,style={semithick,fill=gray!20,inner sep=2pt,minimum size = 10pt}] (b2) [below = \dy of b1] {$\scriptstyle{1}$};
\node[state,style={semithick,fill=gray!20,inner sep=2pt,minimum size = 10pt}] (b3) [below = \dy of b2] {$\scriptstyle{2}$};

\foreach \y in {3,...,5}
\foreach \x in {1,...,3} 
\node[state] (\x\y) [below right = (\x-1/3)*\dy and (\y-2-1/3)*\dx of 0] {};
    
\node[state,accepting] (13) [below right = 2*\dy/3 and 2*\dx/3 of 0] {};

\path[->]
    (a3) edge[blue,bend left=30]
        (a4);

\path[->]
    (a4) edge[blue,bend left=30]
        (a3);

\path[->]
    (a4) edge[red]
        (a5);

\path[->]
    (a5) edge[red,bend right=40]
        (a3);

\path[->]
    (b1) edge[blue,bend left=30]
        (b2);

\path[->]
    (b2) edge[blue,bend left=30]
        (b1);

\path[->]
    (b2) edge[red]
        (b3);  
        
\path[->]
    (b3) edge[red,bend left =40]
        (b1);  

\path[->]
         (25) edge[red, dashed]
             (33);

\path[->]
         (23) edge[blue, dashed, bend right = 20]
             (14);

\path[->]
         (14) edge[blue, dashed, bend right = 20]
             (23);

\path[->]
         (13) edge[blue,bend left = 20]
             (24);

\path[->]
         (24) edge[blue,bend left = 20]
             (13);

\path[->]
         (24) edge[red,bend left = 20]
             (35);

\path[->]
         (35) edge[red,bend left = 30]
             (13);

\path[->]
         (34) edge[red, dashed, bend left = 10]
             (15);
\end{tikzpicture}
\caption{The product $\Ati_{\!H}\times \Ati_{\!H}$ with the non-core parts dashed}
\label{fig: pullback}
\end{figure}
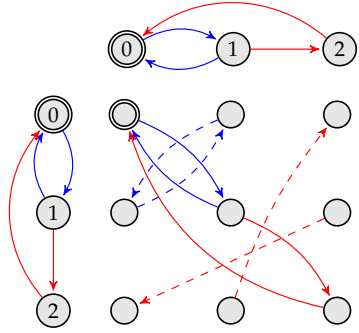

 \begin{align*}
\ofo{\Free}{(\bp, 1,\bp)}{H} &= \red{\Lang}_{\bp, 1} (\Ati_{\!H}) \cap \red{\Lang}_{1, \bp} (\Ati_{\!H}) =\red{\Lang}_{(0,1), (1,0)} (\Ati_{\!H}\times \Ati_{\!H})=\set{a^{2n+1} \st n\in \ZZ}, \\ 
 \ofo{\Free}{(\bp, 2,\bp)}{H} &= \red{\Lang}_{\bp, 2} (\Ati_{\!H}) \cap \red{\Lang}_{2, \bp} (\Ati_{\!H}) =\red{\Lang}_{(0,2), (2,0)} (\Ati_{\!H}\times \Ati_{\!H})=\varnothing.
 \end{align*}
Therefore 
 $$
\ofo{\Free}{2}{H} \,=\, \Big( \ofo{\Free}{(\bp, 1,\bp)}{H} \sqcup \ofo{\Free}{(\bp, 2,\bp)}{H} \Big)^{\!H} =\big(\set{a^{2n+1} \st n\in \ZZ}\big)^{\!H},
 $$
as wanted. 
\end{exm}

As the above example shows, attaching to $H$ all its proper $S$-roots is not enough to reach $\pcl[S]{H}$, because these attachments can create new proper $S$-roots to be attached again. Repeating the process, we get closer and closer to $\pcl[S]{H}$. In \Cref{prop: closure estabilization} we will see that, for $H$ finitely generated, this procedure always stabilizes in a finite number of steps. 

The crucial concept to prove this stabilization is that of algebraic extension. Recall that an extension $H\leq K\leq \Free$ is called \emph{algebraic}, denoted by $H\leq\alg K$, if $H$ is not contained in any proper free factor of $K$; it is not hard to see that any $H\leq\fg \Free$ has only finitely many algebraic extensions, all of them finitely generated again; see~\cite{kapovich_stallings_2002, miasnikov_algebraic_2007}.

\begin{lem}\label{lem: H algebraic}
Let $H$ be a finitely generated subgroup of $\Free$, and let $S\subseteq \NN_{\geq 1}$. Then, $H$ is algebraic in $\gen{\sqrt[S]{H}}$, which is again finitely generated.
\end{lem}

\begin{proof}
Suppose $\gen{\sqrt[S]{H}} =A*B$ and $H\leq A$. Then, since free factors are pure, we have that $\sqrt[S]{H} \subseteq \sqrt[S]{A} = A$. Hence, $\gen{\sqrt[S]{H}}= A$ and $B=1$. Therefore, $H\leq\alg \langle \sqrt[S]{H}\rangle$ and, in particular, $\langle \sqrt[S]{H}\rangle$ is finitely generated. 

(Finite generation of $\gen{\sqrt[S]{H}}$ can also be deduced from the above argument: even if $S$ is infinite, $S_H=S\cap \Iset{H}{\Free}$ is finite and $\smash{\sqrt[S]{H}=\sqrt[^{S_H}]{H}}$ can be obtained from $H$ by attaching the finitely many $w_i$'s from expressions~\eqref{eq: proper S-roots}, $s\in S_H$.)
\end{proof}

\begin{prop}\label{prop: closure estabilization}
Let $H$ be a finitely generated subgroup of $\Free$, let $S\subseteq \mathbb{N}_{\geq 1}$, and for $i\geq 1$, let $H_{i+1}=\gen{\sqrt[S]{H_i} }$, with $H_0=H$. Then, all these subgroups are finitely generated, and the ascending sequence $H=H_0\leq H_1\leq H_2\leq \cdots$ stabilizes at~$\pcl[S]{H}$, \ie there exists $j\geq 0$ such that $H_j = H_{j+1} = \pcl[S]{H}$. 
\end{prop}

\begin{proof}
It is immediate (by induction and using Lemma~\ref{lem: H algebraic}) that all the subgroups $H_i$ are finitely generated. A soon as we have $H_j=H_{j+1}=\gen{ \sqrt[S]{H_j}}$, we deduce $\sqrt[S]{H_j} = H_j$; that is, $H_j$ is $S$-pure and so, $H_j=\pcl[S]{H}$. Therefore it all comes down to prove that the sequence $(H_i)_{i\geq 0}$ stabilizes after a finite number of steps. And this follows immediately from~\Cref{lem: H algebraic}, from transitivity of algebraic extensions ($A\leq\alg B\leq\alg C$ implies $A\leq\alg C$), and from the fact that $H$ has finitely many algebraic extensions; see~\cite{kapovich_stallings_2002, miasnikov_algebraic_2007}.  
\end{proof}

\begin{cor}\label{cor: compute closure}
Let $H\leq \Free$ and $S\subseteq \mathbb{N}_{\geq 1}$. If $H$ is finitely generated, then $\pcl[S]{H}\leq \Free$ is finitely generated as well. If, additionally, $S$ is computable then a free basis for $\pcl[S]{H}$ can be effectively computed from a finite set of generators for $H$.
\end{cor}

\begin{proof}
The first claim follows immediately from \Cref{lem: H algebraic} and \Cref{prop: closure estabilization}. If, additionally, $S$ is computable, then we can list the elements in $S_H=S\cap [2, \card V\rSt_{\!H}]$, and effectively compute (free bases for) the subgroups in the ascending chain $(H_i)_{i\geq 0}$ until it stabilizes (which we will realize, computationally, by the fact $\sqrt[S]{H_j}=H_j$). At this point, we have $H_j=\pcl[S]{H}$ and we have just computed a free basis for $\pcl[S]{H}$.    
\end{proof}

\section{Free times free-abelian groups}\label{sec: free-abelian}

The family of free times free-abelian groups, namely the groups of the form
 \begin{equation} \label{eq:pres F_n x Z^m}
\Fta \,=\, \FTA =\langle \, x_1,\ldots,x_n,t_1,\ldots,t_m\, \mid t_it_j \,=\, t_jt_i,\, t_ix_k=x_kt_i \, \rangle,
 \end{equation}
has some remarkable properties (\eg it is not Howson) and
has been the subject of several studies over the last few years (see~\cite{delgado_algorithmic_2013,sahattchieve_convex_2015, delgado_stallings_2021, delgado_extensions_2017, roy_fixed_2020, roy_degrees_2021, carvalho_dynamics_2020}). In the present section we aim to study order and spectra on this family. The main results in the section are the computability of the sets $\sqrt[k]{\HH}$ (of $k$-roots) and $\smash{\ofo{\Fta}{k}{\HH}}$ (of elements of order $k$) of a finitely generated subgroup $\HH \leqslant \Fta=\FTA$, and that of the corresponding spectrum $\Iset{\HH}{\Fta}$ (which we will see that is always finite).

As has been customary with this kind of groups, we shall refer to elements in~$\Fta$ by their normal forms (with vectors on the right), and abbreviated as $\fta{u}{a} \,=\, u\, t_1^{a_1} \cdots t_m^{a_m}$, where $u=u(x_1,\ldots ,x_n) \in \Free_n$ is called the \defin{free part} of $\fta{u}{a}$, $\vect{a} = (a_1\ldots,a_m) \in \ZZ^m$ is called the \defin{abelian part} of $\fta{u}{a}$, and where the meaningless symbol $t$ allows us to use the standard additive notation in $\ZZ^m$ inserted in a multiplicative environment, \ie $(\fta{u}{a})(\fta{v}{b})=\fta{uv}{a+b}$. We extend this notation to subsets in the natural way: for every $U\subseteq \Fn$ and every $A\subseteq \ZZ^m$, we write $U\, \t^{A} =\set{ \fta{u}{a} \st u\in U,\, \vect{a} \in A}$.

We denote by $\prf$ and $\pra$ the morphisms given by the projections to the free and free-abelian parts, respectively, \ie $\prf \colon \Fta \to \Fn$, $\fta{u}{a} \mapsto u$, and $\pra \colon \Fta \to \ZZ^m$, $\fta{u}{a} \mapsto \vect{a}$. 

\subsection{Subgroups of free times free-abelian groups, bases and completion}

Note that free times free-abelian groups $\Fta =\FTA$ are defined by the splitting short exact sequence \eqref{eq: ses FTA}, which restricts to every subgroup $\HH \leqslant \Fta$ as follows:  
 \begin{labeledcd}{5} \label{eq: ses FTA}
1 & \longrightarrow &\ZZ^m & \longrightarrow & \Fta & \xto{\ \prf\ } & \Fn & \longrightarrow & 1 \phantom{\, ,} \\[-4pt] && \rotatebox[origin=c]{90}{$\leqslant$} && \rotatebox[origin=c]{90}{$\leqslant$} && \rotatebox[origin=c]{90}{$\leqslant$} \nonumber\\[-7pt]
1 & \longrightarrow & \HH \cap \ZZ^m  & \longrightarrow & \HH & \xto{\prf_{|\HH}} & \HH \prf & \longrightarrow & 1 \, .                         \label{eq: sses FTA}
 \end{labeledcd}
Since $H\prf$ is a free group, the restricted short exact sequence \eqref{eq: sses FTA} also splits and, hence, for any given section $\sect$ of $\prf_{|\HH}$, we have 
 \begin{equation} \label{eq: FTA subgroup decomposition}
\HH \,=\, (\HH \cap \ZZ^m) \times \HH\prf \sect \,.
 \end{equation}
In particular, the subgroups of $\Fta$ are again free times free-abelian. The subgroup $\HH \cap \ZZ^m$ is called the \defin{basepoint subgroup} of $\HH$, and is usually denoted by $L_{\HH} =\HH \cap \ZZ^m \leqslant \ZZ^m$ (see \Cref{ssec: enriched automata} for the motivation for this name). Taking a basis for each of the factors in \eqref{eq: FTA subgroup decomposition} we reach our notion of basis of a free times free-abelian group.

\begin{defn} \label{def: basis FTA}
A \defin{basis} of a subgroup $\HH \leqslant \Fta$ is a tuple of the form $(\abasis ; \basis \sigma)$, where $\abasis$ is a free-abelian basis of $L_\HH$, $\basis$ is a free basis of $\HH \prf$, and $\sect$ is a section of $\prf_{|\HH}$.   
(Note that we use the words \emph{free basis}, \emph{free-abelian basis} and just \emph{basis}, depending on whether we refer to the free, free-abelian, or free times free-abelian case, respectively.) 
\end{defn}

Finally we introduce (abelian) completions, a key concept in our analysis.

\begin{defn} \label{def: completion}
Let $\HH$ be a subgroup of $\Fta$, let $\sect$ be a section of $\prf_{|\HH}$, and let $w\in \Free_n$. Then, the \defin{$\sect$-completion} of $w$ in $\HH$ is undefined if $w\not\in \HH \prf$, and $\cab{w}{\HH,\sect}=w\sect \pra \in \ZZ^m$ otherwise; and the \defin {(full) completion} of $w$ in $\HH$ is $\Cab{w}{\HH}=(w\prf_\HH\preim) \pra =\Set{ \vect{a} \in \ZZ^m \st \fta{w}{a} \in \HH}$. This is nothing else than a vector (or the full set of vectors) which \emph{complete} $w$ into an element of $\HH$.
\end{defn}

It is straightforward to see that abelian completions in $\HH$ are either empty or cosets of the basepoint subgroup $L_{\HH}$ of $\HH$.

\begin{lem} \label{lem: completions are cosets}
Let $\HH$ be a subgroup of $\Fta$, and let $w\in \Free$. Then, the completion of $w$ in $\HH$ is
 \begin{equation}
\Cab{w}{\HH} \,=\, \bigg\{\! \begin{array}{ll} \varnothing & \text{if } w\notin \HH \prf, \\ \cab{w}{\HH,\sect} + L_{\HH} & \text{if } w \in \HH \prf \,, \end{array}
 \end{equation}
where $\sect$ is any section of $\prf_{|\HH}$. \qed
\end{lem}

\begin{rem}
When restricted to $\HH \prf$, the completion maps $\vect{c}_{\HH,\sect}\colon \HH \prf \to \ZZ^m$ and $C_{\HH}\colon \HH \prf \to \ZZ^m / L$ are group homomorphisms. In particular, completions work modulo conjugation by $\HH \prf$, \ie for every $w\in \Fn$, every $v\in \HH \prf$, and every section $\sect$ of $\prf_{|\HH}$, $\cab{v^{-1} wv}{\HH,\sect}=\cab{w}{\HH,\sect}$, and $\Cab{v^{-1} wv}{\HH}=\Cab{w}{\HH}$.
\end{rem}

From \eqref{eq: FTA subgroup decomposition} it is clear that a subgroup $\HH \leqslant \Fta$ is finitely generated if and only if $\HH\prf$ is finitely generated, and, in this case, bases are finite and have the form:
 \begin{equation}
 \Basis \,=\,
\big\{ \t[b_1],\ldots,\t[b_r] ; \fta{u_1}{c_1},\ldots,\fta{u_s}{c_s} \big\}\,,
 \end{equation}
where $\abasis =\set{\vect{b_1},\ldots,\vect{b_r}}$ is a free-abelian basis of $L_{\HH}$, $\Basis \prf =\set{u_1, \ldots,u_s}$ is a free basis of $\HH\prf$, and $\vect{c_1},\ldots ,\vect{c_s}\in \ZZ^m$ (the $\Basis$-completions in $\HH$ of the elements in $\Basis \prf$). Then, we denote by $\matr{B}$ the $r\times m$ integral matrix having the row vector $\vect{b_i} \in \ZZ^m$ as $i$-th row, $i=1,\ldots ,r$; and we denote by $\matr{C}$ (or $\matr{C}_{\sect}, \matr{C}_{\scriptscriptstyle{\Basis}}$ if we want to emphasize the associate section or basis) the $s\times m$ integral matrix having the row vector $\vect{c_j} \in \ZZ^m$ as $j$-th row, $j=1,\ldots ,s$. The matrix $\matr{B}$ is called the \defin{basepoint matrix} of $\HH$ \wrt $\Basis$, and the matrix $\matr{C}_{\scriptscriptstyle{\Basis}}$ is called the \defin{completion matrix} of $\HH$ \wrt $\Basis$.

It is not difficult to see that bases for finitely generated subgroups of $\Fta = \FTA$ are computable from a given finite set of generators (see~\cite{delgado_stallings_2021} for an algorithmic-friendly geometric description of the subgroups of $\FTA$ in the spirit of Stallings, briefly surveyed in \Cref{ssec: enriched automata}). Consequently, we shall usually assume that the finitely generated subgroups of $\Fta$ appearing in the discussion are given by some basis.

\begin{lem} \label{lem: completion computation}
Let $\Basis =\set{\t[b_1],\ldots,\t[b_r];\fta{u_1}{c_1},\ldots,\fta{u_s}{c_s}}$ be a basis of a (finitely generated) subgroup $\HH$ of $\Fta$, and  let $w \in \HH \prf$ be a word in the original generators of $\Fn$. Then, $\cab{w}{\HH,\Basis} = (w) \phi_{\scriptscriptstyle{\Basis}} \rho_{s} \matr{C}_{\scriptscriptstyle{\Basis}} $ and hence
 \begin{equation}\label{eq: fg completion}
\Cab{w}{\HH} \,=\, (w) \phi_{\scriptscriptstyle{\Basis}} \rho_{s} \, \matr{C}_{\scriptscriptstyle{\Basis}}
+\gen{\matr{B}}, 
 \end{equation}
where $\phi_{\scriptscriptstyle{\Basis}} \colon \HH \prf \to \Free[s]$ is the isomorphism sending every $w \in \HH \prf$ to its representation in base $\Basis \prf$, $\rho_s \colon \Free[s] \to \ZZ^s$ is the abelianization map, $\matr{C}_{\scriptscriptstyle{\Basis}}$ is the completion matrix of $\HH$ \wrt $\Basis$,
and $\gen{\matr{B}}$ denotes the row space of the basepoint matrix of $\HH$ \wrt $\Basis$. \qed
\end{lem}

\begin{nott}
If the ambient free group (say $\Free[s]$) is clear from the context, and $u\in \Free[s]$, then we usually write $\vect{u}$ (in boldface) the abelianization of $u$, \ie $\vect{u} = u\rho_{s}\in \ZZ^s$. Hence, we can lighten notation in \eqref{eq: fg completion} and just write
 \begin{equation} \label{eq: fg completion2}
\Cab{w}{\HH} \,=\, \vect{u} \, \matr{C} +\gen{\matr{B}} \,,
 \end{equation}
where $u(u_1,\ldots,u_s)=w\phi_{\scriptscriptstyle{\Basis}}$ is the expression of $w\in \HH\prf$ as a word in the free basis $\Basis \prf$ (and $\vect{u}=u\rho_{s}\in \ZZ^s$ is its abelianization).
\end{nott}

\subsection{Enriched automata} \label{ssec: enriched automata}

In \cite{delgado_stallings_2021} the first two authors extended the classical Stallings theory to direct products of free and abelian groups after introducing \defin{enriched automata}; see~\Cref{ssec: automata}. As it happens with their free counterparts, enriched automata constitute a convenient geometric representation of subgroups (of $\FTA$ in this case) and provide a neat understanding of some of their properties. These properties include the behavior of intersections of finitely generated subgroups (which can be not finitely generated) and index among others, but here we shall only focus on their first natural application, namely subgroup recognition.
We summarize the key points below; see \cite{delgado_stallings_2021} for a complete description.

\begin{defn} \label{def: enriched automaton}
A \defin{$\ZZ^m$\!-enriched $X$-automaton} (\defin{enriched automaton}, for short) is an involutive $\ZZ^m\!\times X \times \ZZ^m$-automaton with a subgroup of $\ZZ^m$ attached to the basepoint.
\end{defn}

Ignoring all the abelian information (\ie the basepoint subgroup and all the abelian labels) from a $\ZZ^m$\!-enriched $X$-automaton $\Eti$, we get a standard $X$-automaton, denoted $\sk (\Eti)$ and called the \defin{skeleton} of $\Eti$ (we usually use tildes to distinguish between enriched automata, $\Eti, \Etii, \ldots$, and standard automata $\Ati,\Atii,\ldots$).

In the same vein as for standard $X$-automata, we usually represent enriched automata by a (positive) $X$-automata \emph{with extra $\ZZ^m$-labels at the beginning and at the end of each arc}. And we use the convention that an enriched arc $\edgi \colon \verti \xarc{\hspace{-3pt} \vect{a} \hspace{7pt} x \hspace{7pt} \vect{b}} \vertii$ reads $\llab(\edgi) = \t[-a] x\t[b] =x\t[b-a] \in \Fta$ when crossed forward (from $\verti$ to $\vertii$) and $\llab(\edgi^{-1}) =\t[-b] x^{-1} \t[a] =x^{-1} \t[a-b]=(x\t[b-a])^{-1}\in \Fta$ when crossed backwards (from $\vertii$ to $\verti$).

Accordingly, the \emph{enriched label} of a non-trivial walk $\walki =\edgi_1^{\epsilon_1}\cdots \edgi_k^{\epsilon_k}$ in an enriched automaton~$\Eti$ is $\llab(\walki)=\llab(\edgi_1)^{\epsilon_1}\cdots \llab(\edgi_k)^{\epsilon_k}$; note that the label of $\walki$ as a walk in the skeleton is precisely the free part of the enriched label.

An element in $\FTA$ is said to be \defin{recognized by} the enriched automaton $\Eti$ if it is of the form $w \t^{\vect{a} + \vect{l}}$, where $\fta{w}{a}$ is the label of some nontrivial $\bp$-walk in $\Eti$, and $\vect{l} \in L_{\HH}$ (one may think of $L_{\HH}$ as labelling infinitesimal $\bp$-walks, formally invisible in $\Eti$ because their free labelling is trivial). It is straightforward to see that the set of all the elements recognized by an enriched automaton $\Eti$ is a subgroup of $\FTA$: it is called the \defin{subgroup recognized by $\Eti$}, and denoted by $\gen{\Eti}$.

It is also easy to see that every subgroup in $\FTA$ is recognized by some enriched automata. However, this representation of subgroups by automata is
(as it happens in the free setting) far from unique. In~\cite{delgado_stallings_2021}, the first two authors develop an enriched version of Stallings automata that makes this representation bijective, and also constructive when restricted to the finitely generated case; see~\cite[Theorem 2.7]{delgado_stallings_2021}. Moreover, enriched Stallings automata not only establishes a one-to-one connection with subgroups of $\Fta$, but also encode their structure in a very transparent way. 

\begin{prop}\label{prop: enriched Stallings parts}
Let $\Eti$ be an enriched Stallings automaton recognizing $\HH \leqslant \Fta$. Then, the skeleton of $\Eti$ recognizes $\HH \prf$, whereas the basepoint subgroup of $\Eti$ is $L_{\HH} = \HH \cap \ZZ^m$. \qed
\end{prop}

\begin{cor} \label{cor: h = label + L}
If $\Eti$ is an enriched Stallings automaton recgnizing $\HH \leqslant \Fta$ and $\fta{w}{a} \in \HH$, then there exists a $\bp$-walk $\walki$ in $\Eti$ such that $\fta{w}{a} \in \llab(\walki) \, \t^{L_{\HH}}$. \qed
\end{cor}

\subsection{Order and spectrum within free times free-abelian groups}

We aim to take advantage of our results for free groups (see~\Cref{sec: free})  together with our geometric understanding of subgroups of free times free-abelian groups (see~\Cref{ssec: enriched automata}) to study order and spectra in the latter family. In this section we obtain computable expressions for the sets $\sqrt[k]{\HH}$ (of $k$-roots) and $\smash{\ofo{\Fta}{k}{\HH}}$ (of elements of order $k$) of a finitely generated subgroup $\HH \leqslant \Fta=\FTA$, and we  derive the finiteness of the corresponding spectrum. As it happens with some other problems, the solutions for order and spectrum problems within $\FTA$ are quite more intricate than the mere juxtaposition of the  behavior of these problems in the factors.  

We start with a standard argument solving these problems in the free-abelian case. Recall that, given a finite subset of $\ZZ^m$, we can use linear algebra to compute a free-abelian basis, say $\set{\vect{b_1},\ldots,\vect{b_r}}$ (of course, with $r\leq m$), of the subgroup $L$ it generates, and solve membership on it. Then, if we call $\matr{B}$ the $r\times m$ integral matrix whose rows are the linearly independent vectors $\vect{b_1},\ldots,\vect{b_r} \in \ZZ^m$, we can compute invertible matrices $\matr{P} \in \GL_{r}(\ZZ)$ and $\matr{Q}\in \GL_{m}(\ZZ)$, and (unique) nonzero positive integers $d_1 \divides \cdots \divides d_r$ such that $\matr{P} \matr{B} \matr{Q}=\diag(d_1,\ldots,d_r)$. The $r\times m$ matrix $\matr{S} =\diag(d_1,\ldots,d_r)$ is called the \defin{Smith Normal Form} (SNF) of $\matr{B}$, and the integers $d_1,\ldots,d_r$ are called the \defin{elementary divisors} of $L = \gen{\matr{B}}$. Note that 
 $$
L=\gen{\matr{B}}=\gen{\matr{P}\matr{B}}=\gen{\matr{S}\matr{Q}^{-1}}=\gen{d_1\vect{q}_1,\ldots ,d_r\vect{q}_r} \leq_{fi} \gen{\vect{q}_1,\ldots ,\vect{q}_r}\leq_{\oplus} \gen{\matr{Q}^{-1}}=\ZZ^m,
 $$
where $\vect{q}_1,\ldots ,\vect{q}_m$ are the rows of $\matr{Q}^{-1}$. Hence, $\widetilde{L}= \gen{\vect{q}_1,\ldots ,\vect{q}_r}$ is the smallest direct summand of $\ZZ^m$ containing $L$, and 
 \begin{equation}\label{eq: finite group}
\frac{\ZZ^m}{L} \ \isom \ \frac{\,\widetilde{L}\,}{L} \oplus \ZZ^{m-r} \, \isom \, \frac{\ZZ}{d_1 \ZZ} \oplus \cdots \oplus \frac{\ZZ}{d_r \ZZ} \oplus \ZZ^{m-r}.    
 \end{equation}
From this expression it is easy to derive an explicit description of the $\ZZ^m$-elements of a given order in $L$.

\begin{prop}\label{prop: preorder FTA}
Let $L$ be a subgroup of $\ZZ^m$ of rank $r\leq m$, and let $d_1 \divides d_2 \divides \cdots \divides d_r$ be its elementary divisors. Then, $\ofo{(\ZZ^m)}{0}{L} \,=\, \ZZ^m \setmin\widetilde{L}$ and, for every $k\geq 1$, 
 \begin{align}
\ofo{(\ZZ^m)}{k}{L} & \,=\, \bigg\{\! 
\begin{array}{ll}
\varnothing & \text{if } k \ndivides d_{r} \\
\textstyle{ \bigsqcup \ofo{(\ZZ^m/L)}{k}{} \neq \varnothing} & \text{if }k \divides d_{r} \end{array}
 \end{align}
where $\ofo{(\ZZ^m/L)}{k}{}$ is the nonempty and computable set of elements (namely, cosets modulo $L$) of order $k$ in the finite abelian group $\frac{\widetilde{L}}{L}\simeq \frac{\ZZ}{d_1 \ZZ} \oplus \cdots \oplus \frac{\ZZ}{d_r \ZZ}$.
In particular, $\SMP(\ZZ^m)$ is computable.
\end{prop}

\begin{proof}
Let $\set{\vect{b_1},\ldots,\vect{b_r}}$ be a free-abelian basis for $L$, and let $\matr{B}$, $\matr{P}$, $\matr{Q}$, and $\matr{S}$ be as above. By~\eqref{eq: finite group}, an element $\vect{a} \in \ZZ^m$ has order zero in $L$ if and only if $\vect{a} \notin \widetilde{L}$, \ie $\ofo{(\ZZ^m)}{0}{L} \,=\, \ZZ^m \setmin\widetilde{L}$. 

Moreover, we claim that $\Iset{L}{\ZZ^m}\setmin \{0\}=\Iset{L}{\widetilde{L}}=\Divs{d_r}$. Since $\Ord_{L}(\vect{q_r})=d_r$, the inclusions $\Iset{L}{\ZZ^m}\setmin \{0\}\supseteq \Iset{L}{\widetilde{L}} \supseteq \Divs{d_r}$ are clear. Conversely, given $\vect{q}\in \ZZ^m$, write $\vect{q}=\sum_{i=1}^{m} \lambda_i \vect{q_i}$ and note that $\Ord_{L}(\vect{q})\neq 0$ if and only if $\lambda_{r+1}=\cdots =\lambda_m=0$; furthermore, in this case, $d_r \vect{q}=d_r \sum_{i=1}^{r} \lambda_i \vect{q_i}=\sum_{i=1}^{r} \big(\lambda_i \frac{d_r}{d_i}\big) d_i \vect{q_i} \in L$ and, hence, $\Ord_{L}(\vect{q}) \divides d_r$. The result follows. 
\end{proof}

\begin{cor} \label{cor: purity free-abelian}
Let $L$ be a subgroup of $\ZZ^m$ of rank $r \leq m$, and let $d_r$ be the largest elementary divisor of $L$. Then, the spectrum of $\ZZ^m$ \wrt $L$ is
 \begin{equation*}
\Iset{L}{\ZZ^m}=\left\{ \hspace{-3pt}    \begin{array}{ll}
\Divs{d_r} & \text{if } r=m \quad\text{(\ie if $\Ind{L}{\ZZ^m} < \infty$)} \\ \Divs{d_r} \cup \set{0} & \text{if } r<m \quad\text{(\ie if $\Ind{L}{\ZZ^m} = \infty$)\,,}
\end{array} \right.
 \end{equation*}
and hence computable. 
In particular, $\FSP(\ZZ^m)$ is computable.\qed
\end{cor}

\begin{cor}
Free-abelian groups have subgroup bounded spectra; furthermore, the spectrum membership problem $\SMP(\ZZ^m)$ and the order problem $\OP(\ZZ^m)$ are computable as well. \qed
\end{cor}

We introduce the following notation for later use. 

\begin{nott}
Given a subset $S\subseteq \ZZ^m$, and a positive integer $k\geq 1$, we denote by $\tfrac{1}{k}S$ the set of $k$-roots of $S$ (in $\ZZ^m$), namely $\tfrac{1}{k}S = \frac{1}{k}\big{(} S\cap k\ZZ^m \big{)}\subseteq \ZZ^m$ (this is nothing else than $\sqrt[k]{S}$, but expressed in additive notation). Observe that, in general, $\frac{1}{k}S$ may be empty whereas, if $L\leq \ZZ^m$ is a subgroup, then $\frac{1}{k}L$ is again a subgroup containing $L$; that is, $L\leq \frac{1}{k}L\leq \ZZ^m$. Moreover, if $S=\vect{a}+L$ is a coset of a subgroup $L\leq \ZZ^m$, then $\frac{1}{k}S$ is either empty or a coset of~$\frac{1}{k}(L\cap k\ZZ^m)$. 
\end{nott}

Obviously, the computability of preorders (\Cref{prop: preorder FTA}) immediately produces the computability of roots, which in this abelian context are indeed subgroups of $\ZZ^m$.

\begin{cor}\label{cor: ddd}
For every $L\leq \ZZ^m$ and every $k\geq 1$, one can compute an abelian basis for $\tfrac{1}{k}L$. \qed 
\end{cor}

\medskip
We start the study within free times free-abelian groups by analyzing when a group $\Fta = \FTA$ has elements of order zero \wrt a finitely generated subgroup $\HH \leqslant\fg \Fta$.

\begin{lem} \label{lem: order 0 FTA}
For every finitely generated subgroup $\HH \leqslant \fg \Fta =\FTA$, $0\in \Iset{\HH}{\Fta}$ if and only if $\ind{\HH}{\Fta}=\infty$. Therefore, $\TGP(\FTA)$ is computable.
\end{lem}

\begin{proof}
The implication to the right is true in general. 
For the converse, note first that it is already known to be true for subgroups of both $\Free_n$ and $\ZZ^m$ (see, respectively, \Cref{prop: order 0 Fn} and~\Cref{cor: purity free-abelian}), and let us assume that $\HH$ is a finitely generated subgroup of $\Fta$ of infinite index. Then, from~\cite[Proposition 4.1]{delgado_stallings_2021} we know that either $\ind{L_{\HH}}{\ZZ^m}=\infty$ or $\ind{\HH \prf}{\Fn} =\infty$. Hence, either $0\in \Iset{L_{\HH}}{\ZZ^m}$ or $0\in \Iset{\HH\prf}{\Fn}$. In the first case, there exists $\vect{a}\in \ZZ^m$ such that $k\vect{a}\notin L_{\HH}=\HH \cap \ZZ^m$ and so, $(\t[a])^k =\t^{k \vect{a}} \notin \HH$, for every $k\geq 1$; hence, $0\in \Iset{\HH}{\Fta}$. In the second case, there exists $u\in \Fn$ such that $u^k\not\in \HH\pi$ and so, $u^k \not\in \HH$), for all $k\geq 1$; hence, $0\in \Iset{\HH}{\Fta}$ as well. Since the finite index problem is solvable for free times free-abelian groups (see~\cite[Proposition 4.2]{delgado_stallings_2021}) the computability of $\TGP(\FTA))$ is immediate from the previous claim.
\end{proof}

In free times free-abelian groups it is more convenient to start studying the sets of $k$-roots, to then refining the arguments and getting a description for sets of elements of a given order $k$, and the corresponding spectra. 

Let $\HH$ be a subgroup of $\Fta = \FTA$, and let $k \geq 1$. Then,
\begin{align} 
\sqrt[k]{\HH}
&\,=\, \Set{\fta{u}{a} \in \Fta \st (\fta{u}{a})^k \in \HH} \nonumber\\
&\,=\, \Set{\fta{u}{a} \in \Fta \st u^k \t^{k \vect{a}} \in \HH} \nonumber\\
&\,=\, \Set{\fta{u}{a} \in \Fta \st
       u^k \in \HH \prf \ \text{ and }\  k\vect{a} \in \Cab{u^k}{\HH} \cap k\ZZ^m } \nonumber  \\
&\,=\, \big\{\fta{u}{a} \in \Fta \st
      u \in \sqrt[k]{\HH \prf} \ \text{ and }\  \vect{a} \in \tfrac{1}{k} \Cab{u^k}{\HH} \big\}, \label{eq: k-root(HH) description}
\end{align}
where $\tfrac{1}{k} \Cab{u^k}{\HH} =\tfrac{1}{k} \big{(} \Cab{u^k}{\HH} \cap k\ZZ^m \big{)}$, and $\Cab{u^k}{\HH} \cap k\ZZ^m$ is either empty or a coset of $L_{\HH}\cap k\ZZ^m$. So, the set of $k$-roots of $\HH$ can be described in terms of $k$-roots (of subgroups or cosets) in the factors. However, if $\Cab{u^k}{\HH} \cap k\ZZ^m =\emptyset$, then the element $u\in \HH \prf$ does not contribute to $\sqrt[k]{\HH}$. The following lemma allows us to tighten the description~\eqref{eq: k-root(HH) description} by cleaning off these negligible elements.

Let $\HH\leq\fg \Fta$, let $\Basis$ be a basis for $\HH$, and let $\phi\colon \HH \prf \to \Free[s]$ be the isomorphism sending each element in $\HH\prf$ to its representation in base~$\Basis \prf$. Now, compose $\phi$ with the abelianization map $\rho_s\colon \Free[s] \to \ZZ^s$, and then with right multiplication by the completion matrix $\matr{C}$ of $\HH$ \wrt $\Basis$. It is clear that $L_{\HH} + k \ZZ^m$ is a (normal) subgroup of finite index (at most $k^m$) in $\ZZ^m$, and hence its preimage $M_K = (L_{\HH} + k \ZZ^m)\matr{C}\preim \rho_s \preim \phi\preim $ is a normal subgroup of finite index (at most $k^m$) in $\HH \prf$; see~\Cref{fig: Mk}.

\begin{figure}[H]
\centering\begin{tikzcd}
&[-31pt] \HH \prf \arrow[r,"\phi"] & \Free[s] \arrow[r,"\rho_s"] & \ZZ^s \arrow[r,"\matr{\cdot C}"] & \ZZ^m \\[-20pt]
&\rotatebox[origin=c]{90}{$\normaleq$} &\rotatebox[origin=c]{90}{$\normaleq$}
&\rotatebox[origin=c]{90}{$\normaleq$}
&\rotatebox[origin=c]{90}{$\normaleq$}
\\[-20pt]
M_k = &[-31pt] (L_{\HH}^{_{\!(k)}}) \matr{C}\preim \rho_s \preim \phi\preim & (L_{\HH}^{_{\!(k)}}) \matr{C}\preim \rho_s \preim   \arrow[l, maps to] & (L_{\HH}^{_{\!(k)}}) \matr{C}\preim \arrow[l, maps to]  & L_{\HH}^{_{\!(k)}} \arrow[l, maps to] &[-34pt] = L_{\HH}  + k\ZZ^m 
\end{tikzcd}
\caption{The construction of $M_k \normaleq_{fi} \HH\prf$}\label{fig: Mk}
\end{figure}

\begin{lem} \label{lem: Mk}
For $v\in \Free_n$, $\Cab{v}{\HH} \cap k \ZZ^m \neq \varnothing$ if and only if $v\in M_k$.
\end{lem}

\begin{proof}
From \Cref{lem: completion computation}, if $v \notin \HH \prf$ then $\Cab{v}{\HH}$ is empty, and we are done. Otherwise, $\Cab{v}{\HH} =v\phi \rho_s \matr{C} + L_{\HH}$ intersects $k \ZZ^m$ if and only if $v\phi \rho_s \matr{C} \in L_{\HH}+k\ZZ$, \ie if and only if $v\in (L_{\HH}+k\ZZ^m) \matr{C}\preim \rho_s\preim \phi\preim =  M_k$, as claimed.
\end{proof} 

In summary, an element $v\in \Free_n$ admits an abelian completion in $\HH \leqslant \fg \Fta$ if and only if $v\in \HH\prf$; and it admits an abelian completion in $\HH$ \emph{which is  a multiple of $k$} if and only if $v\in M_k \normaleq_{fi} \Free_n$. In this last case, $\Cab{v}{\HH} \cap k \ZZ^m$ is a coset of $L_{\HH} \cap k \ZZ^m$ and we write $\tfrac{1}{k} \Cab{v}{\HH} = \tfrac{1}{k} (\Cab{v}{\HH} \cap k \ZZ^m)$, a coset of $\tfrac{1}{k} L_{\HH} = \tfrac{1}{k} (L_{\HH} \cap k \ZZ^m)$ for which a representative is easily computable.

\Cref{lem: Mk} allows us to clean off \Cref{eq: k-root(HH) description} to:
 \begin{align} \label{eq: k-roots tight}
\sqrt[k]{\HH} &\,=\, \big\{\fta{u}{a} \in \Fta \st u\in \sqrt[k]{M_k} \ \text{ and }\  \vect{a} \in \tfrac{1}{k}\Cab{u^k}{\HH} \big\}\,,
 \end{align}
where now none of the contributions is negligible.

Now, let us take advantage of the algorithmic description of $\sqrt[k]{M_k}$ (\Cref{cor: k-roots of F computable}), together with our geometric understanding of the subgroups of $\Fta$ (\Cref{ssec: enriched automata}), to obtain a description of the set $\sqrt[k]{\HH}$ of $k$-roots of a finitely generated subgroup $\HH \leqslant \fg \FTA$.

Clearly, we can compute a free basis for $M_k$ and, by \Cref{cor: k-roots of F computable}, we know that
 \begin{equation}\label{eq: sqrtM}
\sqrt[k]{M_k}=\left( \bigsqcup\nolimits_w \big( M_k \cap {M_k}^{\!\!w}\cap \cdots \cap {M_k}^{\!\!w^{k-1}}\big) w\right)^{M_k},
 \end{equation}
where the union runs over a finite set of computable $k$-roots $w$ of $M_k$. Hence, for each of these coset representatives $w$, we have that $w^k\in M_k$ and so there exists a (computable and not necessarily unique) vector $\vect{a}_w\in \ZZ^m$ such that $(w \t^{\vect{a}_w})^k = w^k \t^{k\vect{a}_w} \in \HH$; \ie such that $w \t^{\vect{a}_w} \in \sqrt[k]{\HH}$.

Similarly, for every $v\in M_k \cap {M_k}^{\!\!w}\cap \cdots \cap {M_k}^{\!\!w^{k-1}}$ and every $j=0,\ldots,k-1$, we have $v^{w^{-j}}\in M_k$ and so, there exist a (computable, and non-necessarily unique) vector $\vect{b}_{v,j}\in \ZZ^m$ such that $v \t^{k\vect{b}_{v,j}}\in \HH^{w^j}$. Moreover, these vectors $\vect{b}_{v,j}$ can be chosen linearly: compute a (finite) free basis $\vasis =\set{v_1,\ldots ,v_q}$ for $M_k \cap {M_k}^{\!\!w}\cap \cdots \cap {M_k}^{\!\!w^{k-1}}$ and respective vectors $\vect{b}_{v_i,j} \in \ZZ^m$ such that $v_i\t^{k\vect{b}_{v_i ,j}}\in \HH^{w^j}$; then, for $v\in M_k$, take $\vect{b}_{v,j} = \vect{v} \matr{B_j}$, where ${\vect{v}=v\ab}$ is the abelianization of $v$ \wrt $\vasis$, and $\matr{B_j}$ is the $q\times m$ integral matrix having $\vect{b}_{v_i, j}$ as  $i$-th row; indeed,
 \begin{equation*}
v\t^{k \, \vect{b}_{v,j}} \,=\, v(v_1,\ldots ,v_q)\, \t^{k\, \vect{v} \matr{B_j}} \,=\, v(v_1 \t^{k \, \vect{b}_{v_1, j}},\ldots , v_q \t^{k \, \vect{b}_{v_q, j}}) \,\in\, \HH^{w^j} \!.
 \end{equation*}
Note that conjugates of $M_k$ by different powers of $w$ may coincide, ${M_k}^{\!\!w^{j_1}} = {M_k}^{\!\!w^{j_2}}$; when this happens, we take $\vect{b}_{v_i,j_1} = \vect{b}_{v_i,j_{2}}$, $\matr{B_{j_1}}=\matr{B_{j_2}}$ and so, $\vect{b}_{v,j_1} = \vect{b}_{v,j_2}$, for each $v\in M_K$. Finally, we define $\matr{B} = \sum_j \matr{B_j}$ and $\vect{b}_v =\sum_j \vect{b}_{v,j} = \vect{v}\matr{B}\in \ZZ^m$.

The following claim uses enriched Stallings automata (see~\Cref{ssec: enriched automata}) to build, around each of the already computed $w\t^{\vect{a}_w}\in \sqrt[k]{\HH}$, a whole family of new $k$-roots of $\HH$.

\begin{lem}\label{lem: inH}
For each $w\t^{\vect{a}_w}$ and each $v\in M_k \cap {M_k}^{\!\!w}\cap \cdots \cap {M_k}^{\!\!w^{k-1}}$, we have $vw \,\t^{\vect{b}_v + \vect{a}_w} \in \sqrt[k]{\HH}$.
\end{lem}

\begin{proof}
Let $\Eti$ be a Stallings automaton for $\HH$. Since $w^k \in M_k \leqslant \HH \prf$, we can read $w$ as the free label of $k$ successively adjacent walks $\gamma_j \colon \verti_{j} \xwalk{\ } \verti_{j+1}$ $(j = 0,\ldots,k-1)$ in~$\Eti$ such that $\verti_0 = \verti_{k} = \bp\,$; see~\Cref{fig: hexagon}. Note that the sequence of vertices $\verti_0,\, \verti_1, \ldots \verti_{k-1}$ is not necessarily a trail because there may be repetitions due to the fact that the order of $w$ in $M_k$ is a (maybe proper) divisor of $k$; and due to the fact that $M_k $ is a (maybe proper) subgroup of $\HH \prf$, and hence the Stallings automaton of $\HH\pi$ is a (maybe proper) quotient of the Stallings automaton of $M_k$. 

For every $j=0,\ldots,k-1$, let $w \t^{\vect{a'_j}}$ be the enriched label of $\walki_j$ (of course, with possible repetitions according to those at the sequence of vertices $\verti_0,\, \verti_1, \ldots, \verti_{k-1}$). Then, it is clear that the product $w \t^{\vect{a'_0}}\cdots w \t^{\vect{a'_{k-1}}} = w^k  \smash{\t^{\sum_j \vect{a'_j}}}$ labels a closed $\bp$-walk in $\Eti$, and hence belongs to~$\HH$. Since, by hypothesis,  $w^k \t^{k \vect{a}_w}\in \HH$,
from \Cref{lem: completions are cosets} we have that $\sum_j \vect{a'_j} - k \vect{a}_w \in L_{\HH}$.

On the other hand, $v\in M_k \cap {M_k}^{\!\!w}\cap \cdots \cap {M_k}^{\!\!w^{k-1}}$; hence, for every $j=0,\ldots ,k-1$, $v$ can be read as the free label of a closed $\verti_j$-walk $\xi_{v,j}$ in $\Eti$. Let $\smash{v \t^{\vect{b\!'}_{\!v,j}}\in \HH^{w^j}}$ be the enriched label of $\xi_{v,j}$ (again, with possible  repetitions according to those at the sequence of vertices $\verti_0,\, \verti_1, \ldots, \verti_{k-1}$); see \Cref{fig: hexagon}. Since, by construction, $v\t^{k \vect{b}_{v,j}}\in \HH^{w^j}$, we have that $\vect{b'}_{\!v,j} - k \vect{b}_{v,j}\in L_{\HH^{w^j}}=L_{\HH}$. Therefore, $\sum_j \vect{b'}_{\!\!v,j} - k \vect{b}_v = \sum_j (\vect{b'}_{\!\!v,j} - k \vect{b}_{v,j})\in L_{\HH}$ as well.

\begin{figure}[H]
\centering
\begin{tikzpicture}[shorten >=1pt, node distance=2cm and 2cm, on grid,auto,>=stealth',
decoration={snake, segment length=2mm, amplitude=0.5mm,post length=1.5mm}]
    
\newcommand{\dx}{1.7}
\newcommand{\dy}{1}
\node[state, accepting,style={ 
semithick,fill=gray!20,inner sep=2pt,minimum size = 17pt}] (0) {$\verti_0$};
\node[state,style={ 
semithick,fill=gray!20,inner sep=2pt,minimum size = 17pt}] (1) [above right = \dy and \dx of 0] {$\verti_1$};
\node[state,style={ 
semithick,fill=gray!20,inner sep=2pt,minimum size = 17pt}] (2) [right = 2 of 1] {$\verti_2$};
\node[state,style={ 
semithick,fill=gray!20,inner sep=2pt,minimum size = 17pt}] (3) [below right = \dy and \dx of 2] {$\verti_3$};
\node[state,style={ 
semithick,fill=gray!20,inner sep=2pt,minimum size = 17pt}] (4) [below left = \dy and \dx of 3] {$\verti_4$};
\node[state,style={ 
semithick,fill=gray!20,inner sep=1pt,minimum size = 15pt}] (5) [left = 2 of 4] {$\scriptstyle{\verti_{_{k\text{-}1}}}$};

\path[->]
    (0) edge[blue,snake it]
                node[pos=0.5,below right = 0.01] {$\B{w}$}
                node[pos=0.87,above left = -0.05] {$\vect{a'_{0}}$}
            (1);

\path[->]
        (1) edge[blue,snake it]
                node[pos=0.5,below = 0.05] {$\B{w}$}
                node[pos=0.87,above] {$\vect{a'_{1}}$}
            (2);
\path[->]
        (2) edge[blue,snake it]
                node[pos=0.5,below left] {$\B{w}$}
                node[pos=0.87,above right=-0.05] {$\vect{a'_{2}}$}
            (3);
            
\path[->]
        (3) edge[blue,snake it]
                node[pos=0.5,above left] {$\B{w}$}
                node[pos=0.87,below right=-0.05] {$\vect{a'_{3}}$}
            (4);
            
\path[-]
        (4) edge[blue,dashed]
                node[pos=0.5,above = 0.05] {($k$ times)}
            (5);
            
\path[->]
        (5) edge[blue,snake it]
                node[pos=0.5,above right] {$\B{w}$}
                node[pos=0.95,below = 0.05] {$\vect{a'_{k-1}}$}
            (0);
\path[->,thick]
        (0) edge[red,loop,out=200,in=160,looseness=8,min distance=15mm,snake it]
            node[left=0.1] {$v$}
            node[pos=0.95,above] {$\vect{b'}_{\!\!v,0}$}
            (0);
            
\path[->,thick]
        (1) edge[red,loop,out=140,in=100,looseness=8,min distance=15mm,snake it]
            node[above left=0.1] {$v$}
            node[pos=0.95,right] {$\vect{b'}_{\!\!v,1}$}
            (1);
            
\path[->,thick]
        (2) edge[red,loop,out=80,in=40,looseness=8,min distance=15mm,snake it]
            node[above right] {$v$}
            node[pos=0.95,below right =-0.05] {$\vect{b'}_{\!\!v,2}$}
            (2);
            
\path[->,thick]
        (3) edge[red,loop,out=20,in=-20,looseness=8,min distance=15mm,snake it]
            node[right=0.05] {$v$}
            node[pos=0.95,below] {$\vect{b'}_{\!\!v,3}$}
            (3);
            
\path[->,thick]
        (4) edge[red,loop,out=-40,in=-80,looseness=8,min distance=15mm,snake it]
            node[below right=0.05] {$v$}
            node[pos=0.95,left] {$\vect{b'}_{\!\!v,4}$}
            (4);
            
\path[->,thick]
        (5) edge[red,loop,out=-100,in=-140,looseness=8,min distance=15mm,snake it]
            node[below left=0.05] {$v$}
            node[pos=0.95,above left= -0.05] {$\vect{b'}_{\!\!v,k-1}$}
            (5);
    
\end{tikzpicture}
\vspace{-10pt}
\caption{The walks $\gamma_j$ (in blue) and $\xi_{v,j}$ (in red) within $\Eti$}
\label{fig: hexagon}
\end{figure}
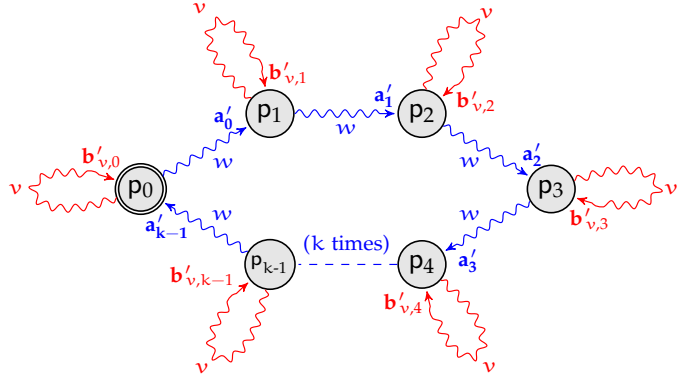

Now, it is clear from \Cref{fig: hexagon} that  $\xi_{v,0}\gamma_0\xi_{v,1}\gamma_1\cdots \xi_{v,k-1}\gamma_{k-1}$ is a closed $\bp$-walk in $\Eti$, and hence its label belongs to $\HH$. That is:
 \begin{align*}
v\t^{\vect{b'}_{\!\!v,0}} \,  w\t^{\vect{a'_0}} \, v\t^{\vect{b'}_{\!\!v,1}} \, w\t^{\vect{a'_1}} \,\cdots\, v\t^{\vect{b'}_{\!\!v,k-1}} \, w\t^{\vect{a'_{k-1}}}
&\,=\ (vw)^k\, \t^{\sum_j \vect{b'}_{\!\!v,j} + \sum_j \vect{a'_j}} \\
&\,=\  (vw)^k  \,  \t^{k\vect{b}_v + k \vect{a}_w}  \,  \t^{\sum_j \vect{b'}_{\!\!v,j} - k \vect{b}_v}  \,  \t^{\sum_j \vect{a'_j} - k \vect{a}_w} 
\in\, \HH \,.
 \end{align*}
Since both $\sum_j \vect{b'}_{\!\!v,j} -k \vect{b}_v$ and $\sum_j \vect{a'_j} - k \vect{a}_w$ belong to $L_{\HH}\leqslant \HH$, we deduce that $(vw\, t^{\vect{b}_v+\vect{a}_w})^k =(vw)^k t^{k\vect{b}_v+ k\vect{a}_w}\in \HH$, as claimed.
\end{proof}

For later use, we introduce the graphical counterparts of the matrices $\matr{B_j}$ and $\matr{B}$ in the natural way: $\matr{B'_j}$ is the integral $q\times m$ matrix whose $i$-th row is $\matr{b'}_{\!\!v_i, j}$, and $\matr{B'} = \sum_j \matr{B'_j}$. Recall that, for every $j=0,\ldots,k-1$, the label of the closed $\verti_j$-walk $\xi_{v_i,j}$ is $v_i \t^{\vect{b'}_{\!\!v_i,j}}$ and so, for every $v(v_1, \ldots ,v_q) \in M_k \cap {M_k}^{\!\!w}\cap \cdots \cap {M_k}^{\!\!w^{k-1}}$, the label of the closed $\verti_j$-walk $v(\xi_{v_1,j},\ldots ,\xi_{v_q,j})$ is $v(v_1 \t^{\vect{b'}_{\!\!v_1,j}},\ldots ,v_q \t^{\vect{b'}_{v_q,j}}) = v(v_1, \ldots ,v_q) \t^{\vect{v}\matr{B'_j}}$; therefore, $\vect{b'}_{\!\!v,j} = \vect{v} \matr{B'_j}$.

\begin{thm} \label{thm: FTA k-roots}
Let $\HH$ be a finitely generated subgroup of $\Fta = \FTA$ 
and let $k\geq 1$. Then, the set of $k$-roots of $\HH$ in $\Fta$ is
 \begin{equation}\label{eq: FTA k-roots}
\sqrt[k]{\HH} \,=\, \left(\bigsqcup\nolimits_{w \t^{\vect{a}_w}\in S} \HH_{w}  \, w \t^{\vect{a}_w} \right)^{\HH},
 \end{equation}
where $S$ is a finite set of $k$-roots of $\HH$, and $\HH_w$ are finitely generated subgroups of $\Fta$, all of them computable from a given basis for $\HH$.
\end{thm}

\begin{proof}
Within the free group $\Free[n]$, compute a free basis for $M_k$ and then finitely many $k$-th roots of $M_k$ giving~\Cref{eq: sqrtM}. Following the procedure described above, for each such $w\in \sqrt[k]{M_k}$, compute a vector $\vect{a}_w\in \ZZ^m$ such that $wt^{\vect{a}_w}\in \sqrt[k]{\HH}$, a free basis for 
$\vasis =\set{v_1,\ldots ,v_q}$ for $M_k \cap {M_k}^{\!\!w}\cap \cdots \cap {M_k}^{\!\!w^{k-1}}$, and respective vectors $\vect{b}_{v_i,j} \in \ZZ^m$ such that $v_i\t^{k\vect{b}_{v_i ,j}}\in \HH^{w^j}$; finally, take $\vect{b}_{v_i} =\sum_j \vect{b}_{v_i,j}\in \ZZ^m$ and consider the finitely generated subgroup
 \begin{align}
\HH_w &\,=\, \big\langle v_1 \t^{\vect{b}_{v_1}},\, \ldots ,v_q \t^{\vect{b}_{v_q}} ,\, \t^{\tfrac{1}{k}L_{\HH}} \big\rangle \\ &\,=\,
\big\{v \t^{\vect{b}_v + \vect{c}} \st v \in M_k \cap {M_k}^{\!\!w}\cap \cdots \cap {M_k}^{\!\!w^{k-1}} \text{ and } \vect{c} \in \tfrac{1}{k}L_{\HH} \big\} \,\leqslant\fg\, \Fta \,.
 \end{align}
Note that, by~\Cref{cor: ddd}, we can compute an abelian basis for $\tfrac{1}{k}L_{\HH}$, and get a basis for $\HH_w$. 

With this data, we claim that~\eqref{eq: FTA k-roots} holds. In fact, the inclusion to the left follows easily from \Cref{lem: inH} and the general form of the elements in $\HH_w$: $(v\t^{\vect{b}_v + \vect{c}} \,w \t^{\vect{a}_w})^k=(vw\t^{\vect{b}_v + \vect{a}_w})^k \,\t^{k \vect{c}} {\in \HH}$.

For the other inclusion, consider an arbitrary element $u \t^\vect{c} \in \sqrt[k]{\HH}$. Then, $u^k t^{kc}\in \HH$ and so, $u^k\in M_k$, $u\in \sqrt[k]{M_k}$, and hence, from Equation~\eqref{eq: sqrtM}, $u=(vw)^{v'}$ for some coset representative $w$, some $v\in M_k \cap {M_k}^{\!\!w}\cap \cdots \cap {M_k}^{\!\!w^{k-1}}$, and some $v'\in M_k$. But, by \Cref{lem: inH}, $(vw)^k \,\t^{k\vect{b}_v+k\vect{a}_w} \in \HH$, and conjugating by $v'\in M_k\leqslant \HH\pi$, we get $u^k \,\t^{k\vect{b}_v+k\vect{a}_w} = ((vw)^k)^{v'} \,\t^{k\vect{b}_v + k\vect{a}_w}\in \HH^{v'} =\HH$. Therefore, $k \vect{c} - k \vect{b}_v - k \vect{a}_w \in L_{\HH}$ and so, $\vect{l}= \vect{c}-\vect{b}_v-\vect{a}_w\in \frac{1}{k}L_{\HH}$. Thus,
 \begin{equation*}
u \t^{\vect{c}} \,=\, u\, \t^{\vect{b}_v + \vect{a}_w} \,\t^{\vect{c}-\vect{b}_\vect{v}-\vect{a}_w} \,=\, (vw)^{v'} \,\t^{\vect{b}_v+\vect{a}_w}\, \t^{\vect{l}} \,=\, \big( v \t^{\vect{b}_v+\vect{l}} \, w \t^{\vect{a}_w}\big)^{v'} \,\in\,
\big(\HH_w \, w\t^{\vect{a}_w}\big)^{\HH}.
 \end{equation*}
Finally, the union under consideration is disjoint because projecting $\HH_w wt^{\vect{a}_w}$ to the free part we get, precisely, $\big( M_k \cap {M_k}^{\!\!w}\cap \cdots \cap {M_k}^{\!\!w^{k-1}}\big) w$; see~\Cref{eq: sqrtM}.
\end{proof}

\begin{rem}
Observe that the description in~\Cref{eq: FTA k-roots} for the free times free-abelian context is parallel to that of~\Cref{cor: k-roots of F computable}: in both cases, the set of $k$-th roots of $H$ (resp., $\HH$) is a finite union of computable cosets, conjugated by $H$ (resp., $\HH$); however the corresponding subgroups are intersections of conjugates of $H$ in the free case, concretely $H\cap H^{w}\cap H^{w^2} \cap \cdots \cap H^{w^{k-1}}$, while they have a slightly more complicated form (although still computable) in the free times free-abelian case, namely $\HH_w$.
\end{rem}

\begin{rem}
In the special case where $\HH\pi=1$ (\ie when $\HH=L_{\HH}\leqslant \ZZ^m$), we have $M_k=1$, $\sqrt[k]{M_k}=1$, and Equation~\eqref{eq: sqrtM} consists of just a single trivial coset, namely $w=1$ and $\{1\}=(\{1\}w)^{\{1\}}$; further, $w\t^{\vect{a}_w}=1\t^{\vect{0}}$, $\HH_{w}=\tfrac{1}{k}L_{\HH}$, and Equation~\eqref{eq: FTA k-roots} consists on just a single coset as well, $\sqrt[k]{\HH} =\big( (\tfrac{1}{k}L_{\HH} \big) 1\t^{\vect{0}})^{\HH}=\tfrac{1}{k}L_{\HH}$. This agrees with the general fact $\sqrt[k]{\HH}\cap \ZZ^m =\sqrt[k]{L_{\HH}}$.  
\end{rem}

As we know, the set of  $k$-roots $\sqrt[k]{\HH}$ consists precisely of the elements in $\Fta$ whose orders in $\HH$ are divisors of $k$ (including the elements of order $1$, which are the elements in $\HH$ themselves). Let us now refine the description from \Cref{thm: FTA k-roots} by breaking each coset $\HH_{w} \,w\t^{\vect{a}_w}$ there into a finite disjoint union of computable cosets (of a certain subgroup of finite index in~$\HH_w$), in such a way that the order function becomes constant over each of these new smaller cosets.

\begin{thm}
Let $\HH$ be a finitely generated subgroup of $\Fta = \FTA$. 
Then, the set of elements of $\Fta$ of any given nonzero order in $\HH$ is computable. More precisely, given a finite set of generators for $\HH$ and an integer $k \geq 1$, we can algorithmically decide whether $\ofo{\Fta}{k}{\HH} \neq \varnothing$ 
and, if so,
the set of elements of $\Fta$ of order $k$ in $\HH$ is
 \begin{equation}
\ofo{\Fta}{k}{\HH} \,=\, \left(\,\bigsqcup\nolimits_{z\t^{\vect{a}_z}\in S'} \HH'_{z} \, z\t^{\vect{a}_{z}}\right)^{\HH}, 
 \end{equation}
where $S'$ is a finite set of elements of order $k$ in $\HH$, and $\HH'_z$ are finitely generated subgroups of $\Fta$, all of them computable from a given basis for $\HH$.
\end{thm}

\begin{proof}
From description \eqref{eq: FTA k-roots} for $\sqrt[k]{\HH}$, fix one of the finitely many cosets, say $\HH_{w} \,w\t^{\vect{a}_w}$, and consider the part of $\Eti$ (a Stallings automaton for $\HH$) corresponding to it; see \Cref{fig: hexagon}. Recall that the vectors $\vect{a'_j}$ and $\vect{b'}_{\!\!v,j}$ in that picture are related to $\vect{a}_w$ and $\vect{b}_{v,j}$, respectively, through the equations $k\vect{a}_w -\sum_{j=0}^{k-1} \vect{a'_j} \in L_{\HH}$ and $k\vect{b}_{v,j}-\vect{b'}_{\!\!v,j} \in L_{\HH}$, and therefore
 \begin{equation}\label{eq: www}
\vect{v} \big( k \matr{B}- \matr{B'}\big) \,=\, k\vect{v}\matr{B} -\sum_{j=0}^{k-1} \vect{v}\matr{B'_j} \,=\, k\vect{b}_v -\sum_{j=0}^{k-1} \vect{b'}_{\!\!v,j} \,\in\, L_{\HH} \,.
 \end{equation}
Note that the elements in the coset $\HH_{w} \,w\t^{\vect{a}_w}$ may have different orders in $\HH$, but the order in $\HH$ of such an element $v \t^{\vect{b}_v + \vect{c}} \, w \t^{\vect{a}_w}$ is always a divisor of $k$, and a multiple of  $r_0 = \Ord_{\HH\pi} (vw)$, which is independent from $v\in M_k \cap {M_k}^{\!\!w}\cap \cdots \cap {M_k}^{\!\!w^{k-1}}$; indeed, $r_0$ is the minimum index~$j>0$ such that $\verti_j=\bp$ in \Cref{fig: hexagon}.

Now, fix an element $v\t^{\vect{b}_v + \vect{c}} \, w \t^{\vect{a}_w} \in \HH_w\, w \t^{\vect{a}_w}$, let $r\in \mathbb{N}$ be such that $r_0 \divides r \divides k$, and let $k=rs$. By construction, $\verti_r=\bp$, and therefore the sequences of vectors $\vect{b'}_{\!\!v,0}, \ldots ,\vect{b'}_{\!\!v,k-1}$ and $\vect{a'_0}, \ldots ,\vect{a'_{k-1}}$  repeat in $s$ blocs of $r$, \ie
\begin{align*}
(\vect{b'}_{\!\!v,0}, \ldots ,\vect{b'}_{\!\!v,k-1}) & \,=\, (\vect{b'}_{\!\!v,0}, \ldots ,\vect{b'}_{\!\!v,r-1} \,,\, \vect{b'}_{\!\!v,0}, \ldots ,\vect{b'}_{\!\!v,r-1} \,,\, \stackrel{s)}{\ldots} \,,\, \vect{b'}_{\!\!v,0}, \ldots ,\vect{b'}_{\!\!v,r-1}), \\ (\vect{a'_0}, \ldots ,\vect{a'_{k-1}}) & \,=\, (\vect{a'_0}, \ldots ,\vect{a'_{r-1}}\,,\, \vect{a'_0}, \ldots ,\vect{a'_{r-1}}\,,\, \stackrel{s)}{\ldots} \,,\, \vect{a'_0}, \ldots ,\vect{a'_{r-1}} ).
 \end{align*}
Hence, from Equation~\eqref{eq: www}, $rs\vect{b}_v -s\sum_{j=0}^{r-1} \vect{b'}_{\!\!v,j} =k\vect{b}_v -\sum_{j=0}^{k-1} \vect{b'}_{\!\!v,j} \in L_{\HH}$ and so, $r\vect{b}_v -\sum_{j=0}^{r-1} \vect{b'}_{\!\!v,j}=\frac{1}{s} \big(k \vect{b}_v -\sum_{j=0}^{k-1} \vect{b'}_{\!\!v,j} \big) \in \frac{1}{s}L_{\HH}$. Similarly, $k\vect{a}_w-\sum_{j=0}^{k-1} \vect{a'_j} =rs\vect{a}_w -s\sum_{j=0}^{r-1} \vect{a'_j} \in L_{\HH}$ and so, $r \vect{a}_w - \sum_{j=0}^{r-1} \vect{a'_j} =\frac{1}{s}\big(k \vect{a}_w -\sum_{j=0}^{k-1} \vect{a'_j} \big) \in \frac{1}{s}L_{\HH}$. Furthermore, observe that
 \begin{align*}
(v \t^{\vect{b}_v+\vect{c}} \, w \t^{\vect{a}_w})^r & \,=\, (vw)^r \,\t^{r\vect{b}_v + r \vect{a}_w + r\vect{c}} \\ & \,=\, \big( (vw)^r \,\t^{\sum_{j=0}^{r-1} \vect{b'}_{\!\!v,j} + \sum_{j=0}^{r-1} \vect{a'_j}} \big) \t^{r \vect{b}_v - \sum_{j=0}^{r-1} \vect{b'}_{\!\!v,j}} \,\t^{r \vect{a}_w - \sum_{j=0}^{r-1} \vect{a'_j}} \,\t^{r\vect{c}},
 \end{align*}
where the element $(vw)^r \,\t^{\sum_{j=0}^{r-1} \vect{b'}_{\!\!v,j} + \sum_{j=0}^{r-1} \vect{a'_j}} $ between the last pair of parentheses belongs to $\HH$ since it labels closed walk at $\bp$ in $\Eti$. Hence,
 \begin{equation*}
(v\t^{\vect{b}_v+\vect{c}}\,w\t^{\vect{a}_w})^r \in \HH \ \Biimp \ \frac{1}{s}\bigg( \Big( k \vect{b}_v - \sum_{j=0}^{k-1} \vect{b'}_{\!\!v,j} \Big) + \Big( k \vect{a}_w - \sum_{j=0}^{k-1} \vect{a'_j} \Big) + k \vect{c} \bigg) \,\in\, L_{\HH} \,.
 \end{equation*}
 This means that the order of $v\t^{\vect{b}_v+\vect{c}} \, w\t^{\vect{a}_w}$ in $\HH$ is:
 \begin{equation}\label{eq: FTA order}
\min \Big\{ r\in \NN \st r_0 \divides r \divides k \text{\, and \,} \vect{v}(k\matr{B}-\matr{B'}) + (k\vect{a}_w -\textstyle{\sum_{j=0}^{k-1} \vect{a'_j}}) +k\vect{c} \in \tfrac{k}{r} L_{\HH} \Big\} \,.
 \end{equation}
Recall that the vector $k \vect{a}_w-\sum_{j=0}^{k-1} \vect{a'_j}$ is fixed, and consider the group homomorphism
 \begin{equation*}
\begin{array}{rcl} \varphi \colon \HH_w & \to & L_{\HH}/kL_{\HH} \\ v \,\t^{\vect{b}_v+\vect{c}} & \mapsto & \vect{v} (k \matr{B}-\matr{B'})+ k \vect{c} + k L_{\HH}, \end{array}
 \end{equation*}
which is onto since $\vect{c}$ moves along $\tfrac{1}{k}L_{\HH}$. Since $L_{\HH}/kL_{\HH}$ is a finite (abelian) group, $\HH'_w =\ker \varphi$ is a finite index subgroup of $\HH_w$ (in fact, of index $p=k^{\rk(L_{\HH})}$), for which we can easily compute a (finite) basis. Then, computing a $\varphi$-preimage for each element in $L_{\HH}/kL_{\HH}$, say $u_1 \t^{\vect{b}_{u_1}+\vect{c_1}}, \ldots ,u_p \t^{\vect{b}_{u_p}+\vect{c_p}} \in \HH_w$, we obtain the decomposition $\HH_w=\bigsqcup_{i=1}^p \HH'_w \, u_i \t^{\vect{b}_{u_i} + \vect{c}_i}$. Therefore,
 \begin{equation}
\HH_w w\t^{\vect{a}_w} \,=\, \bigsqcup_{i=1}^p \HH'_w \, u_i t^{\vect{b}_{u_i} + \vect{c}_i}\, w \t^{\vect{a}_w} \,=\, \bigsqcup_{i=1}^p \HH'_w \, (u_i w)\, \t^{\vect{b}_{u_i} + \vect{a}_w + \vect{c_i}} \,.
 \end{equation}
Finally, we claim that the order function with respect to $\HH$ is constant over each of these smaller cosets. Indeed, according to~\Cref{eq: FTA order}, for every $v \t^{\vect{b}_v+\vect{c}}\in \HH'_w$, the order of $v \t^{\vect{b}_v+\vect{c}}(u_i w) \, \t^{\vect{b}_{u_i} + \vect{a}_w + \vect{c_i}} = (vu_i w) \t^{\vect{b}_v + \vect{b}_{u_i} + \vect{a}_w + \vect{c}+ \vect{c_i}}$ is the smallest integer $r>0$ such that $r_0 \divides r\divides k$ and 
 \begin{equation}\label{eq: aaa}
(\vect{v} +\vect{u_i})(k\matr{B}-\matr{B'}) + (k\vect{a}_w- \sum_{j=0}^{k-1} \vect{a'_j}) + k(\vect{c}+\vect{c_i}) \,\in\, sL_{\HH}.
 \end{equation}
But $\vect{v}\big( k\matr{B}-\matr{B'})+k\vect{c} \in k L_{\HH}\leqslant sL_{\HH}$ since $v\t^{\vect{b}_v+\vect{c}}\in \HH'_w = \ker \varphi$. Therefore, Equation~\eqref{eq: aaa} holds if and only if 
 \begin{equation*}
\vect{u_i}(k\matr{B}-\matr{B'})+(k\vect{a}_w -\sum_{j=0}^{k-1} \vect{a'_j})+k\vect{c_i}\,\in\, sL_{\HH} \,,
 \end{equation*}
which is to say that $\Ord_{\HH} ( v \, \t^{\vect{b}_v+ \vect{c}}\,u_i w\, \t^{\vect{b}_{u_i} + \vect{a}_w + \vect{c_i}}) =\Ord_{\HH} (u_i w\, \t^{\vect{b}_{u_i} + \vect{a}_w+ \vect{c_i}})$. This proves the claim. 

Finally, collecting together all the cosets (if any) containing elements of order $k$ and resetting the notation,
we obtain the claimed result.
\end{proof}

\begin{cor}\label{cor: SBP FTA}
Free times free-abelian groups $\Fta =\FTA$ have subgroup bounded spectra and the corresponding bounds are computable (in particular $\SBP(\Fta)$ is computable).
More precisely, if $\HH$ is a finitely generated subgroup of $\Fta$, then $\Ord_{\HH}(\Fta) \subseteq \Ord_{\HH \prf}(\Free) \cdot \Ord_{L_\HH}(\ZZ^m)$; in particular, $\Ord_{\HH}(\Fta) $ is bounded above by $\card V\rstallings{H}\cdot d$, where $d$ is the largest elementary divisor of $L_{\HH}=\HH \cap \ZZ^m \leq \ZZ^m$.
\end{cor}

\begin{proof}
First note that
\begin{align*}
    0 \in \Ord_{\HH}(\Fta)
    &\Biimp
    \Ind{\HH}{\Fta} = \infty\\
    &\Biimp
    \Ind{\HH \prf}{\Free} = \infty \text{\, or \,} \Ind{L_\HH}{\ZZ^m} = \infty\\
     &\Biimp
    0 \in \Ord_{\HH \prf}(\Free) \text{\, or \,} 0 \in \Ord_{L_\HH}(\ZZ^m)\,,
\end{align*}
where we have used \Cref{lem: order 0 FTA}, Proposition 4.1 in~\cite{delgado_stallings_2021}, \Cref{prop: order 0 Fn}, and \Cref{cor: purity free-abelian} successively.
Hence, it is clear that if $0 \in \Ord_{\HH}(\Fta)$ then $0 \in \Ord_{\HH \prf}(\Free) \cdot \Ord_{L_\HH}(\ZZ^m)$.
So, it only remains to prove that the inclusion holds for strictly positive orders, \ie that $\Ord^{*}_{\HH}(\Fta) \subseteq \Ord^{*}_{\HH \prf}(\Free) \cdot \Ord^{*}_{L_\HH}(\ZZ^m)$. 
This is obvious if $\HH$ is pure. Otherwise, let $k\in \Ord^{*}_{\HH}(\Fta)$; this means that some element $u\t^{\vect{a}}\in \Fta$ satisfies $u\t^{\vect{a}},\, \ldots , (u\t^{\vect{a}})^{k-1}\not\in \HH$, but $(u\t^{\vect{a}})^k\in \HH$. Then, $u^k\in \HH\pi$ and so, $k$ is a multiple of $r=\Iset{\HH\pi}{u}$, say $k=rs$. But now $u^r\in \HH\pi$ and so $u^r\t^{\vect{b}}\in \HH$ for some vector ${\vect{b}\in \ZZ^m}$. By construction, $u^r\t^{r\vect{a}},\, (u^r\t^{r\vect{a}})^2,\, \ldots ,\, (u^r\t^{r\vect{a}})^{s-1}\not\in \HH$ whereas $(u^r\t^{r\vect{a}})^s=u^k\t^{k\vect{a}}\in \HH$; and this implies that the vector $r\vect{a}-\vect{b}\in \ZZ^m$ satisfies $r\vect{a}-\vect{b},\, 2(r\vect{a}-\vect{b}),\, \ldots ,\, (s-1)(r\vect{a}-\vect{b})\notin L_{\HH}$ whereas $s(r\vect{a}-\vect{b})=k\vect{a}-s\vect{b}\in L_{\HH}$. Therefore, $\Ord_{L_\HH}(r\vect{a} - \vect{b}) = s$ and hence $s\in \Iset{L_{\HH}}{\ZZ^m}$; the claimed result follows. Finally, the upper bound follows from~\Cref{cor: SBP Free} and ~\Cref{cor: purity free-abelian}.
\end{proof}

Since the membership problem is decidable for free times free-abelian groups (\eg using enriched Stallings automata), the computability of $\OP$, $\SMP$, and $\FSP$ follows from the previous results and the discussion in \Cref{sec: algorithmic}.

\begin{thm} \label{thm: OP FTA}
Let $\Fta = \FTA$ be a free times free-abelian group. Then, the order problem $\OP(\Fta)$ is computable; \ie there exists an algorithm which, on input a finitely generated subgroup $\HH \leqslant\fg \Fta$ and an element $\fta{u}{a} \in \Fta$, outputs the order $\Ord_{H}(\fta{u}{a})$ of $\fta{u}{a}$ in $\HH$.  \qed
\end{thm}

\begin{thm} \label{thm: SMP FTA}
Let $\Fta = \FTA$ be a free times free-abelian group. Then the spectrum membership problem $\SMP(\Fta)$ is computable; \ie there exists an algorithm which, on input a finitely generated subgroup $\HH \leqslant\fg \Fta$ and an integer $k \in \NN$, decides whether $k$ belongs to the spectrum $\Ord_{H}(\Fta)$ of $\Fta$ \wrt $\HH$. \qed
\end{thm}

\begin{thm} \label{thm: FSP FTA}
Let $\Fta = \FTA$ be a free times free-abelian group. Then
there exists an algorithm which, on input a finitely generated subgroup $\HH$ of $\Fta$, outputs the (finite) spectrum $\Iset{H}{\Free}$. In particular, 
the finite spectrum problem $\FSP(\Fta)$ is computable. \qed
\end{thm}

\subsection{Computing pure closures within $\Fn \times \mathbb{Z}^m$}

We now show that the results concerning pure $S$-closures in a free group can be easily generalized to free times free-abelian groups: \Cref{lem: closure2} works eactly like \Cref{lem: H algebraic} and~\Cref{prop: closure2} just puts into play the Noetherianity of $\ZZ^m$.

\begin{lem} \label{lem: closure2}
Let $\HH$ be a finitely generated subgroup of $\Free_n \times \mathbb{Z}^m$, and let $S\subseteq \NN_{\geq 1}$. Then $\HH\pi$ is algebraic in $\gen{ \sqrt[S]{\HH}}\pi$, which is again finitely generated. \qed
\end{lem}

\begin{prop}\label{prop: closure2}
Let $\HH$ be a finitely generated subgroup of $\Free_n \times \mathbb{Z}^m$, let $S\subseteq \mathbb{N}_{\geq 1}$, and for $i\geq 0$, let $\HH_{i+1}=\gen{\sqrt[S]{\HH_i} }$, with $\HH_0=\HH$. Then, all these subgroups are finitely generated, and the ascending sequence $\HH =\HH_0 \leqslant \HH_1 \leq \HH_2 \leq \cdots$ stabilizes at $\pcl[S]{\HH}$, \ie there exists $j\geq 0$ such that~$\HH_j = \HH_{j+1} =\pcl[S]{\HH}$. 
\end{prop}

\begin{proof}
Exactly like in the proof of~\Cref{prop: closure estabilization}, the subgroups $\HH_i$ are finitely generated and 
the sequence $\HH\pi=\HH_0 \pi \leq \HH_1 \pi \leq \HH_2 \pi \leq \cdots$ stabilizes, \ie there exists $j_1 \geq 0$ such that $\HH_{j_1}\pi= \HH_{j_1+1}\pi$.

On the other hand, $\mathbb{Z}^m$ is a Noetherian group and so, the ascending sequence $L_{\HH}=L_{\HH_0}\leq L_{\HH_1}\leq L_{\HH_2}\leq \cdots$ stabilizes as well: there exists $j_2 \geq 0$ such that $L_{\HH_{j_2}} =L_{\HH_{j_2 +1}}$. Take $j=\max(j_1,j_2)$, and it is straightforward to see that $\HH_j=\HH_{j+1}=\pcl[S]{\HH}$.
\end{proof}

Finally, an argument analogous to that in \Cref{cor: compute closure} provides the computability of the pure closure in the free times free-abelian case.

\begin{cor}
Let $\HH\leq \Free_n \times \mathbb{Z}^m$ and $S\subseteq \mathbb{N}_{\geq 1}$. If $\HH$ is finitely generated, then $\pcl[S]{\HH}$ is finitely generated as well. If, additionally, $S$ is computable then a basis for $\pcl[S]{\HH}$ can be effectively computed from a finite set of generators for $\HH$. \qed
\end{cor}

\subsection*{Acknowledgements}

The first named author was partially supported by MINECO grant PID2019-107444GA-I00 and the Basque Government grant IT974-16. The second named author acknowledges partial support from the Spanish Agencia Estatal de Investigación, through grant MTM2017-82740-P (AEI/ FEDER, UE), and also from the Graduate School of Mathematics through the María de Maeztu Programme for Units of Excellence in R\&D (MDM-2014-0445). 

\renewcommand*{\bibfont}{\small}
\printbibliography

\end{document}